\documentclass[11pt]{amsart}
\usepackage{amssymb}
\usepackage{graphicx}
\usepackage{tensor}
\usepackage{mathrsfs}

\usepackage[margin=1in, marginpar=.6in]{geometry} % Sets all margins to 1 in.

\usepackage[usenames,dvipsnames,svgnames,table]{xcolor}
\usepackage[colorlinks = true,linkcolor = BrickRed,citecolor = Green, pdfencoding=auto, psdextra]{hyperref}
\hypersetup{linktocpage}
\usepackage{enumitem}

\allowdisplaybreaks				% To prevent a page break inside the displayed mode, use \\*

\usepackage{mathtools}
\mathtoolsset{showonlyrefs}

\definecolor{green}{rgb}{0,0.8,0} % Redefines the color green.

\newtheorem{maintheorem}{Main Theorem}
\newtheorem{theorem}{Theorem}[section]
\newtheorem{corollary}[theorem]{Corollary}
\newtheorem{lemma}[theorem]{Lemma}
\newtheorem{proposition}[theorem]{Proposition}
\theoremstyle{definition}
\newtheorem{definition}[theorem]{Definition}
\newtheorem{example}[theorem]{Example}
\theoremstyle{remark}
\newtheorem{remark}[theorem]{Remark}
\numberwithin{equation}{section}
\newcommand{\relphantom}[1]{\mathrel{\phantom{#1}}}

\makeatletter
\newcommand{\nrm}{\@ifstar{\nrmb}{\nrmi}}
\newcommand{\nrmi}[1]{\Vert{#1}\Vert}
\newcommand{\nrmb}[1]{\left\Vert{#1}\right\Vert}
\newcommand{\abs}{\@ifstar{\absb}{\absi}}
\newcommand{\absi}[1]{\vert{#1}\vert}
\newcommand{\absb}[1]{\left\vert{#1}\right\vert}
\newcommand{\brk}{\@ifstar{\brkb}{\brki}}
\newcommand{\brki}[1]{\langle{#1}\rangle}
\newcommand{\brkb}[1]{\left\langle{#1}\right\rangle}
\newcommand{\set}{\@ifstar{\setb}{\seti}}
\newcommand{\seti}[1]{\{#1\}}
\newcommand{\setb}[1]{\left\{ #1\right\}}
\makeatother

\newcommand{\td}[1]{\widetilde{#1}}
\newcommand{\br}[1]{\overline{#1}}
\newcommand{\ul}[1]{\underline{#1}}

\newcommand{\VERT}[1]{{\left\vert\kern-0.25ex\left\vert\kern-0.25ex\left\vert #1
    \right\vert\kern-0.25ex\right\vert\kern-0.25ex\right\vert}}

\DeclareMathOperator{\dist}{dist}

\DeclareMathOperator{\supp}{supp}

\let\Re\relax
\DeclareMathOperator{\Re}{Re}
\let\Im\relax
\DeclareMathOperator{\Im}{Im}

\newcommand{\aeq}{\simeq}
\newcommand{\aleq}{\lesssim}
\newcommand{\ageq}{\gtrsim}

\newcommand{\lap}{\Delta}

\newcommand{\ud}{\mathrm{d}}
\newcommand{\rd}{\partial}
\newcommand{\nb}{\nabla}

\newcommand{\bb}{\Big}

\newcommand{\0}{\emptyset}

\newcommand{\peq}{\relphantom{=}}			% makes an empty space with the same width as =
\newcommand{\alp}{\alpha}
\newcommand{\bt}{\beta}
\newcommand{\gmm}{\gamma}

\newcommand{\dlt}{\delta}
\newcommand{\Dlt}{\Delta}
\newcommand{\eps}{\epsilon}
\newcommand{\veps}{\varepsilon}

\newcommand{\lmb}{\lambda}

\newcommand{\vphi}{\varphi}

\newcommand{\sgm}{\sigma}

\newcommand{\tht}{\theta}

\newcommand{\bfh}{{\bf h}}

\newcommand{\bfPhi}{\boldsymbol{\Phi}}

\newcommand{\bbC}{\mathbb C}

\newcommand{\bbR}{\mathbb R}
\newcommand{\bbS}{\mathbb S}

\newcommand{\bbZ}{\mathbb Z}

\newcommand{\calF}{\mathcal F}

\newcommand{\calH}{\mathcal H}

\newcommand{\calK}{\mathcal K}
\newcommand{\calL}{\mathcal L}
\newcommand{\calM}{\mathcal M}
\newcommand{\calN}{\mathcal N}

\newcommand{\calS}{\mathcal S}

\newcommand{\frkc}{\mathfrak c}

\DeclareMathAlphabet{\mathbbold}{U}{bbold}{m}{n}	% Blackboard bold for numbers
\newcommand{\R}{{\mathbb R}}
\newcommand{\p}{\partial}
\newcommand{\scrS}{\mathscr S}
\renewcommand{\div}{{\mathrm{div}}}
\newcommand{\ob}{\mathrm{ob}}
\newcommand{\rst}{|}			% restriction
\usepackage{titletoc}
\newcommand{\nocontentsline}[3]{}
\let\origcontentsline\addcontentsline
\newcommand\stoptoc{\let\addcontentsline\nocontentsline}
\newcommand\resumetoc{\let\addcontentsline\origcontentsline}

\begin{document}

\title[]{The good commutator approach to global asymptotics \\ for the Schr\"odinger equation with variable coefficients}%: Title of the article
\author{Sung-Jin Oh}%
\address{Department of Mathematics, UC Berkeley, Berkeley, CA 94720, USA and School of Mathematics, KIAS, Seoul, 02455, Korea}%
\email{sjoh@math.berkeley.edu}%

\author{Federico Pasqualotto}%
\address{Department of Mathematics, UC San Diego, La Jolla, CA 92093, USA}%
\email{fpasqualotto@ucsd.edu}%

\author{Ning Tang}%
\address{Department of Mathematics, UC Berkeley, Berkeley, CA 94720, USA}%
\email{ning\_tang@math.berkeley.edu}%

%\thanks{}%
%\subjclass{}%
%\keywords{}%

%\date{\today}%
%\dedicatory{}%
%\commby{}%
% ----------------------------------------------------------------
\begin{abstract}
We present a robust physical-space approach to establish time decay and global asymptotics of solutions to variable-coefficient Schr\"odinger equations in $(3+1)$-dimensions. As an immediate nonlinear application, we obtain new small data global existence and asymptotics results for quasilinear Schr\"odinger equations with cubic, Hamiltonian nonlinearity, variable coefficients in their linear part, and possibly outside obstacles, even in the presence of trapped bicharacteristics (provided that they are suitably unstable).

Our approach relies on three primary ingredients.
First is the concept of a \emph{good commutator}, which extends Klainerman's classical commuting vector field method, and develops upon earlier approaches of Cuccagna--Georgiev--Visciglia and Rodnianski--Tao in the variable-coefficient case, and of Ifrim--Tataru, Ifrim--Koch--Tataru in the nonlinear case. Second, we apply Ifrim and Tataru's \emph{testing-by-wave-packets method}, which is key for obtaining global asymptotics and handling sharp decay assumptions on the coefficients. Finally, we introduce a systematic technique for analyzing the control provided by the good commutator, termed \emph{two-scale elliptic analysis}. Together, these techniques significantly broaden the applicability of physical-space methods to a wider class of variable-coefficient nonlinear problems.

\end{abstract}
\maketitle
% ----------------------------------------------------------------
\setcounter{tocdepth}{2}
\titlecontents{section}
[0em]
{\vspace{0.4em}}% adding a vertical space before each section entry
{\contentslabel{2em}}
{\bfseries}
{\bfseries\titlerule*[0.5pc]{$\cdot$}\contentspage}
[\vspace{0.2em}]% adding a vertical space after each section entry
\titlecontents{subsection}
              [1.5em]
              {}%{\normalsize}%
              {\contentslabel{2em}}%
              {}
              {\titlerule*[0.5pc]{$\cdot$}\contentspage}%
\titlecontents{subsubsection}
              [2.5em]
              {}%{\normalsize}%
              {\contentslabel{2em}}%
              {}
              {\titlerule*[0.5pc]{$\cdot$}\contentspage}%

\tableofcontents
%----------------------------------------------------------------

%%Version2 revisions mainly focus on:
%-The previous version stated Proposition~\ref{prop:elliptic-near} in a form that was too weak for later use, specifically in Corollary~\ref{cor:near-zone-asymp-higher-order}. In fact, its original proof was able to establish a slightly stronger version with an error cutoff to the near region, which is sufficient for us.
%-We provide more details in Step 1 of Proof of Main~Theorem~\ref{thm:nonlinear-1} in order to handle $A_c$ large in \ref{hyp:ult-stat}. In addition, Step 4 has also been revised.
%-We fixed the proof of Proposition~\ref{prop:est-LN(u)-Hartree}.

\section{Introduction}\label{sec-intro}
Let $\calH$ be a time-dependent second-order elliptic operator on $\calM^{3}:= \bbR^{3} \setminus \calK$,
\begin{equation} \label{eq:calH}
	\calH u = D_{\mu} (a^{\mu \nu}(t, x) D_{\nu} u) + b^{\mu}(t, x) D_{\mu} u + D_{\mu} (b^{\mu}(t, x)  u) + c(t, x) u,
\end{equation}
where $\calK$ is a compact (possibly empty) subset of $\bbR^{3}$ and $D_{\mu} = i^{-1} \rd_{\mu}$ ($\mu=1, 2, 3$). We consider the initial(-boundary, if $\calK \neq \0$) value problem associated with the time-dependent Schr\"odinger equation
\begin{equation}\label{eq:sch-gen-rough}
    \begin{cases}
        (i\p_t - \calH)u = \calN(u) + f, \quad &\hbox{ in } \calM^{1+3}, \\
        u \rst_{\set{t = 0}} = u_0 &\hbox{ on } \bbR^{3} \setminus \calK, \\
	    u = 0 &\hbox{ on } [0, \infty) \times \rd \calK.
    \end{cases}
\end{equation}
posed on a spacetime region
\begin{equation*}
	\calM^{1+3}:= [0, \infty)_{t} \times \calM^{3} = [0,\infty)_t \times (\bbR^{3}_{x} \setminus \calK),
\end{equation*}
which is a submanifold of $\bbR^{1+3}$ equipped with rectangular coordinates $(t, x^{1}, x^{2}, x^{3})$.

Our aim in this paper is to present an approach to establishing the \emph{time decay} and \emph{global asymptotics} of solutions to \eqref{eq:sch-gen-rough} in $(3+1)$-dimensions based on robust \emph{physical-space} techniques (i.e., \emph{not} directly on spectral theory or Fourier methods for $\calH$ and \emph{not} on assuming any special structure such as complete integrability). It is applicable under a set of black-boxed hypotheses that have proved useful in the context of other variable-coefficient dispersive equations with a view towards nonlinear problems; specifically, see \ref{hyp:af}, \ref{hyp:iled*}, \ref{hyp:se}, and \eqref{eq:E-u} in Section~\ref{subsec-results}, as well as Remark~\ref{rem:linear}.(1)--(2). In particular, this setting includes possibly \emph{large}, \emph{time-dependent}, \emph{complex-valued} lower-order coefficients (in this case, $\calH$ would be non-symmetric), as well as for domains involving \emph{obstacles} and \emph{(unstable) trapped geodesics}.

As a consequence of the robustness of our approach, we obtain new small data global existence results, along with precise global asymptotics, for \emph{quasilinear} problems with cubic nonlinearity and variable-coefficient linear part, (possibly) outside of obstacles; see Main~Theorems~\ref{thm:nonlinear-1}~-~\ref{thm:nonlinear-2} below. For example, a simple-to-describe setup where our approach yields the first small data global existence result is the initial-boundary value problem \eqref{eq:sch-gen-rough} for the cubic Hamiltonian quasilinear Schr\"odinger equation outside of $\calK$,
\begin{equation*}
    (i \rd_{t} + \lap) u = D_{\mu} (G^{\mu \nu} \abs{u}^{2} D_{\nu} u) - G^{\mu \nu} u D_{\mu} \br{u} D_{\nu} u \quad \hbox{ in } [0, \infty)_{t} \times (\bbR^{3} \setminus \calK),
\end{equation*}
where $G^{\mu \nu}$ is any constant symmetric real-valued matrix, and $\calK$ consists of two disjoint compact strictly convex sets with smooth boundaries; see Example~\ref{ex:obstacle}, Main~Theorem~\ref{thm:nonlinear-1}, and Remark~\ref{rem:nonlinear1} (in particular, the term ``Hamiltonian'' refers to the matched structure of the first- and second-order terms on the right-hand side). This situation is complicated by the existence of a trapped geodesic (which, however, lies exactly on the line segment connecting the two closest points on the two obstacles and is unstable), and previous results in this setting concerned power-type nonlinearities.

Our approach relies on three main ingredients.
\begin{enumerate} [leftmargin=*, label=\arabic*.]
    \item \emph{Good commutator:} We utilize a second-order operator $\calL$, called a \emph{good commutator} in this paper, that obeys a favorable commutation property with $i \rd_{t} - \calH$ and possesses good spacetime weights (for the precise expression, see Section~\ref{subsubsec-mod-general-case} below). At a rudimentary level, it is an extension of the classical commuting vector field method of Klainerman (adapted to the setting of nonlinear Schr\"odinger equations in \cite{HT86, Const90}), which fails to apply in the variable-coefficient case. In the case of (massless) wave equations, the idea of modifying the weighted commuting operators to gain better commutation properties can be traced back to the geometric analysis of Christodoulou--Klainerman \cite{ChrKla1993}; we also refer to the more explicit approach of Alinhac \cite{Ali2003, Ali2005}.
    
    Our primary inspiration is drawn from Rodnianski--Tao \cite[Section 16]{RT15}, in which a version of this commuting operator was introduced to obtain almost optimal decay for the linear Schr\"odinger equation on three-dimensional asymptotically flat manifolds with, in particular, no zeroth order term (i.e., $c = 0$). Other key prior works are Cuccagna--Georgiev--Visciglia \cite{CGV14}, who employed a similar construction for the operator $\calH = -\lap + c$ in one space dimension to prove time decay for associated nonlinear equations with power-type nonlinearities; and Ifrim--Tataru \cite{IfrTat2019} and Ifrim--Koch--Tataru \cite{IfrKocTat2023}, who introduced and used nonlinear versions of such a commuting operator for the Benjamin--Ono and KdV equations, respectively. For earlier applications of ``good commutators'' in the context of local smoothing for the operator $\calH = -\lap + c$, we refer the reader to \cite{HO89}.

    \item \emph{Ifrim and Tataru's testing-by-wave-packets method} \cite{IT15, IT24}: This is a powerful method for extracting the leading order asymptotics, which optimizes and balances the previous approaches based solely on ODE integration in physical space (see, e.g., \cite{LS06}) or tracking Fourier coefficients (see, e.g., \cite{HN98, KP11}). 
    
    The testing-by-wave-packets method is not only our main vehicle for obtaining sharp global asymptotics, but also a crucial ingredient for handling sharp decay assumptions for the variable coefficients; see Section~\ref{subsubsec-IT-wp} below. We apply this method with test wave packets adapted to the free Schr\"odinger flow $i \rd_{t} + \lap$, \emph{not} to $i \rd_{t} - \calH$; this version is robust and easy to implement in physical space. We also note the prior work of Stewart \cite{Ste24} and Krieger--Schlag--Widmayer \cite{KSW25}, which used free test wave packets in the one-dimensional problem with a potential in the faraway region; see Remark~\ref{rem:wp-test-free} for further discussion.

    \item \emph{Two-scale elliptic analysis}: The favorable commutation property of our good commutator $\calL$ comes at a price, namely, its coercivity properties do not follow directly, unlike in the well-understood constant-coefficient case. The \emph{two-scale elliptic analysis} is a systematic technique we introduce in order to handle this difficulty. Its basic idea is to analyze the control provided by the good commutator $\calL$ by modelling it by two different elliptic operators depending on the spacetime region.
\end{enumerate}

The applicability of our framework extends beyond the Schr\"odinger equation. In a companion paper \cite{OPT25+}, we develop an analogous approach for the variable-coefficient Klein--Gordon equation in $(3+1)$-dimensions, and obtain new small data global existence and asymptotics results for quasilinear problems with \emph{quadratic} nonlinearity (as opposed to cubic in the Schr\"odinger case), variable coefficients in their linear part, both in the obstacle and non-obstacle case, even in the presence of (suitably unstable) trapped geodesics. 

\subsection{Main results} \label{subsec-results}

\subsubsection{The main theorem in the linear case, simple version (Main Theorem~\ref{thm:linear-f=0})}

We first formulate our main result in the linear case with $f = 0$, $\calK = \0$, and simplified assumptions. We consider the initial value problem \begin{equation}\label{eq:sch-gen-simple}
    \begin{cases}
        (i\p_t - \calH)u = 0, \quad &\hbox{ in } [0, \infty) \times \bbR^{3}, \\
        u \rst_{\set{t = 0}} = u_0 &\hbox{ on } \bbR^{3}.
    \end{cases}
\end{equation}

Given parameters $\eta > 0$, $s_c \in \bbZ_{\geq 0}$ and $M_{c} \in \bbZ_{\geq 0}$, we make the following assumptions:
\begin{enumerate}[label=$(\mathrm{AF}_{0})$]
\item (Asymptotic flatness)\label{hyp:af-0} We assume $a^{\mu\nu}$ to be real and positive definite. For simplicity, we set $\bfh^{\mu\nu}:= a^{\mu\nu} - \delta^{\mu\nu}$. Moreover, for some positive number $A_{c}$, the following holds true for all $(t, x) \in [0, \infty) \times \calM^{3}$,
\begin{align*}
	\abs{\rd_{x}^{(\leq M_c)} (\brk{x} \rd_{x})^{(\leq 3)}\bfh^{\mu \nu} (t, x)}
	+ \abs{\rd_{x}^{(\leq M_c)}\calS  \bfh^{\mu \nu} (t, x)} & \leq A_{c} \brk{x}^{-2 \eta}, \\
	\abs{  \rd_{x}^{(\leq M_c)} (\brk{x} \rd_{x})^{(\leq 2)} b^{\mu}(t, x)}
	+ \abs{\rd_{x}^{(\leq M_c)}\calS b^{\mu}(t, x)}& \leq A_{c}\brk{x}^{-1 - 2 \eta}, \\
	\abs{\rd_{x}^{(\leq M_c)}(\brk{x} \rd_{x})^{(\leq 1)}c(t, x)} + \abs{\rd_{x}^{(\leq M_c)}\calS c(t, x)} &\leq A_{c} \brk{x}^{-2 - 2 \eta},
\end{align*}
where we denote $\calS:= (t \rd_{t} + x \cdot \rd_{x})$.
We furthermore require $\bfh^{\mu\nu}$ to satisfy two additional assumptions:
\begin{align*}
	\bb|\p_x^{(\leq M_c)}(\brk{x}\p_x)^{(\leq 2)} (\bfh^{\mu \nu} (t, x) \rd_{\mu} r \rd_{\nu} r) \bb| \leq A_{c} t \brk{x}^{-2-2\eta} \quad \hbox{ for } \abs{x} > \brk{t}^{\frac{1}{2}}, \label{eq:af-extra1-0}\tag{AF${}_{0}$-e1}\\
    \bb|\bfh^{\mu \nu} (t, x) \rd_{\mu} r \xi_\nu \bb| \leq A_{c}|\xi| t^{\frac12} \brk{x}^{-1-2\eta} \quad \hbox{ for } \abs{x} > \brk{t}^{\frac{1}{2}}, \quad \xi \in \R^3, \label{eq:af-extra2-0}\tag{AF${}_{0}$-e2}
\end{align*}
where $r:= \abs{x} = \sqrt{\sum_{1\leq \mu\leq 3} (x^\mu)^2}$.
\end{enumerate}

\begin{enumerate}[label=$(\mathrm{ILED}_{0})$]
\item (Integrated local energy decay)\label{hyp:iled-0} For every $t_{0} \geq 0$, $T > 0$, $R > 0$ and $N = 0, 1, \ldots, M_{c}$, the following holds:
For every $F:[t_{0}, t_{0} + T] \times \bbR^{3} \to \bbC$ such that
\begin{equation*}
 \supp F \subseteq \set{(t, x) \in [t_0, t_{0}+T] \times \bbR^{3}: \abs{x} < R}, \quad
 \nrm{\rd_{x}^{(\leq N+s_{c})} F}_{L^{2} L^{2}([t_0, t_0 + T] \times \set{\abs{x} < R})} < + \infty,
\end{equation*}
there exists a unique forward (i.e., $\psi(t_{0}) = 0$) solution $\psi$ to $(i\p_t - \calH) \psi = F$ in $[t_0, t_{0} + T] \times \bbR^{3}$ in $C_{t}([t_{0}, t_{0} + T]; H^{N})$. Moreover, the solution $\psi$ satisfies
\begin{equation*}
	\nrm{\rd_{x}^{(\leq N)} \psi}_{L^{2} L^{2}([t_0, t_0 + T] \times \set{\abs{x} < R})} \leq C_{R, N} \nrm{\rd_{x}^{(\leq N+s_{c})} F}_{L^{2} L^{2}([t_0, t_0 + T] \times \set{\abs{x} < R})},
\end{equation*}
where $C_{R, N}$ is independent of $t_0, T$.
\end{enumerate}

\begin{enumerate}[label=$(\mathrm{Th}_{0})$]
\item \label{hyp:se-0} (No threshold eigenvalue/resonance) For every $t_{0} \geq 0$,
$\calH \rst_{t=t_{0}}$ is invertible as a map $\dot{H}^{1}(\bbR^{3}) \to \dot{H}^{-1}(\bbR^{3})$ and the operator norm of the inverse has a uniform bound independent of $t_0$.
\end{enumerate}

We view assumptions \ref{hyp:iled-0} and \ref{hyp:se-0} as a \emph{black-box}, and we emphasize that the results below do not depend on the specific technique used to prove them. We also emphasize that our approach does \emph{not} require the self-adjointness of $\calH$, nor any spectral theory, distorted Fourier transform, or resolvent estimates. (Nevertheless, these concepts can be used, and in fact are often useful, for establishing the hypotheses \ref{hyp:iled-0} and \ref{hyp:se-0}.) A further discussion of these assumptions will be given in Remark~\ref{rem:linear} below. 

In its simplest form, our main result concerning the linear equation \eqref{eq:sch-gen-simple} is as follows.

\begin{maintheorem}[The linear case, $f = 0$, simplified assumptions] \label{thm:linear-f=0}
Assume that \ref{hyp:af-0}, \ref{hyp:iled-0}, and \ref{hyp:se-0} hold true with
\begin{equation*}
	\eta > 0, \quad s_c \in \bbZ_{\geq 0}, \quad M_{c} \in \bbZ_{\geq 0}.
\end{equation*}
Then there exist $c_{0} > 0$ and $s_c' \geq s_c$ such that the following holds.

Assume that $M_c$ and $M_0$ satisfy the bounds \[
    M_c \geq M_0, \quad M_0 \geq N + c_0 s_c'
\]
for some $N \in \bbZ_{\geq 0}$. Then for any solution $u$ to \eqref{eq:sch-gen-simple} with Schwartz initial data $u_0$ satisfying the bound \begin{equation}\label{eq:D-u}
\tag{IDB}
    \|\brk{x}^2 \p_x^{(\leq M_0)} u_0\|_{L^2} \leq D_0
\end{equation}
and the a-priori (unweighted) energy bound
\begin{equation} \label{eq:E-u} \tag{EB}
	\sup_{t \in [0, \infty)} \nrm{\rd_{x}^{(\leq M_{0}+2)} u(t, \cdot)}_{L^{2}} \leq A_{0},
\end{equation}
we have the following:
\begin{enumerate}[leftmargin=*, label=\arabic*. ]
\item (Sharp time decay of the uniform norm) Denote $D:= D_0 + A_0$. We have
\begin{equation} \label{eq:linear-decay}
	\sup_{x \in \br{\calM^{3}}} \abs{\rd_{x}^{(\leq N)} u(t, x)} \aleq D \brk{t}^{-\frac{3}{2}}.
\end{equation}

\item (Sharp global asymptotics)
Denote $\eta_\sharp:= \min\{\eta, \frac14 - \}$.
There exists $\gmm_\infty(v) \in L^\infty$ such that for any multi-index $\alpha$ with $|\alpha| \leq N$, any $x \in \br{\calM^{3}}$, 
\begin{equation} \label{eq:linear-asymp}
    \p_x^\alpha u(t, x) = t^{-\frac{3}{2}} e^{i \frac{\abs{x}^{2}}{4 t}} \sum_{ \bt_1 + \bt_2 = \alp} \binom{\alp}{\bt_1} (\tfrac{ix}{2t})^{\beta_1} \gmm_\infty(\tfrac{x}{t}) (\p_x^{\beta_2}\varphi_{0})(t, x) + O_{L^\infty}(D t^{-\frac{3}{2} - \frac14\eta_\sharp}),
\end{equation}
where $\varphi_{0}$ is the solution to $\calH \varphi_{0} = 0$ with $\varphi_{0} \to 1$ as $x \to \infty$ (see Lemma~\ref{lem:varphi0} and its Corollary for its existence, uniqueness and regularity), and
\begin{equation*}
	\sup_{v \in \bbR^{3}} \abs{\brk{v}^N\gmm_\infty(v)} \aleq D.
\end{equation*}
We refer to $\gmm_\infty = \gmm_\infty(v)$ as the \emph{radiation field} associated with $u$.
\end{enumerate}
\end{maintheorem}
\begin{proof}
    This theorem is an immediate corollary of Main Theorem~\ref{thm:linear}.
\end{proof}

In the constant-coefficient case $\calH = - \lap$, \eqref{eq:linear-asymp} is precisely the \emph{Fraunhofer formula} (see for instance \cite[Lemma 4.12]{KV13}), where $\varphi_{0} = 1$ and $\gmm_\infty(v)$ is given in terms of the Fourier transform of the initial data (more precisely, at frequency $\xi$ whose group velocity $\rd_{\xi} \abs{\xi}^{2} = 2 \xi$ equals $v$). In the variable-coefficient case, we refer to the recent work of Gell-Redman--Gomes--Hassell \cite{GRComHa2022, GRComHas2023} (for coefficients settling down to the constant-coefficient case) and Looi--Sussman \cite{LooiSussman} (in the time-independent case), in which refined global asymptotics were proved using microlocal-analysis tools. For further prior works on the variable-coefficient case (with different hypotheses on $\calH$), we refer to Section~\ref{subsubsec-refs-var-coeff}. 

As emphasized earlier, the main novelty of the present paper lies in the robustness of its argument, which is \emph{readily adaptable to nonlinear problems}; see already Main~Theorems~\ref{thm:nonlinear-1} and \ref{thm:nonlinear-2} below. At this point, we pause to make some remarks concerning Main~Theorem~\ref{thm:linear-f=0}.

\begin{remark} \label{rem:linear}
We comment on the assumptions and conclusions of Main~Theorem~\ref{thm:linear-f=0}.
 \begin{enumerate}[leftmargin=*]
\item Our setup based upon the hypotheses \ref{hyp:af-0}, \ref{hyp:iled-0}, and \ref{hyp:se-0} is directly inspired by what proved useful in the context of variable-coefficient dispersive equations with a view towards nonlinear problems. 

Identification of \emph{asymptotic flatness} such as \ref{hyp:af-0} and a (fairly weak) \emph{global $L^{2}$-type decay estimate} of the form \ref{hyp:iled-0} as the fundamental ingredients for obtaining stronger dispersive estimates goes back to the work of Tataru \cite{Tataru}. The main result of \cite{Tataru} was global-in-time dispersive estimates (as opposed to local-in-time in the prior works \cite{StaTat2002, RobZui2005, HasTaoWun2005}) for an outgoing wave-packet parametrix for the variable-coefficient Schr\"odinger equation under a weak asymptotic flatness assumption, which was then combined with the global-in-time local smoothing estimate (a strong version of \ref{hyp:iled-0}, see Item~2 below) to establish global-in-time Strichartz estimates. 

Another major influence on our work is the literature on variable-coefficient (massless) wave equations, where this paradigm is more extensively developed. In addition to the construction of outgoing parametrices and global-in-time Strichartz estimates in \cite{MetTat2012, MetTat2009}, we highlight the $r^{p}$ method of Dafermos--Rodnianski \cite{DR10}, which provides a physical-space proof of sharp uniform time decay; see Moschidis \cite{Mos2016} for a further generalization. Additionally, we refer to the local pointwise decay results of Tataru \cite{Tat2013} and Metcalfe--Tataru--Tohaneanu \cite{MetTatToh2012}, as well as the vector field methods developed by Oliver \cite{Oli2016} and Oliver--Sterbenz \cite{OliSte2020}.

Our adoption of \ref{hyp:se-0} as a quantitative, time-dependent alternative to the standard assumption of regularity (i.e., no eigenvalue or resonance) at the threshold energy in the time-independent self-adjoint case originates from Metcalfe--Tataru--Tohaneanu \cite{MetTatToh2012, MetTatToh2017} (first as what is referred to as stationary local energy decay, and second more explicitly), and Metcalfe--Sterbenz--Tataru \cite[Definition 2.9]{MST20}.

\item \ref{hyp:iled-0} may be considered a relaxation of the \emph{global-in-time local smoothing estimate} for the Schr\"odinger equation (see also Proposition~\ref{prop:kato-small}). For the constant-coefficient Schr\"odinger equation, this estimate goes back to \cite{ConSau1987, Sjo1987, Veg1988}; in the variable-coefficient case, it was established in \cite{Tataru} for small asymptotically flat perturbations of the constant-coefficient case, and in \cite{RodTao2007} for the Laplace--Beltrami operator for a metric with a possibly large compactly supported perturbation with a non-trapping assumption (in particular, $c = 0$). For systematic studies of the global-in-time local smoothing estimate for variable-coefficient Schr\"odinger equations, we refer the reader to Marzuola--Metcalfe--Tataru \cite{MMT08}, as well as the work of Rodnianski--Tao \cite{RT15}.

While our assumption \ref{hyp:iled-0} also \emph{asserts $L^2$-integrability of local energy}, it differs from the global-in-time local smoothing estimate in that \emph{it does not assert smoothing of the solution} (compared to the data); in fact, a derivative loss of any order is allowed. For this reason, we call this estimate \emph{integrated local energy decay}, not local smoothing. Our setup makes it straightforward to incorporate the presence of trapped bicharacteristics, in which case some derivative loss (compared to local smoothing) is necessary; see the original paper of Doi \cite{Doi2000} and the textbook treatment in \cite[Theorem 7.1]{DZ19}. We refer to Examples~\ref{ex:lap-small-pert} and \ref{ex:non-trapping} below (as well as Example~\ref{ex:obstacle} for obstacle problems) for examples of operators for which this estimate could be verified, as well as the references cited above.

The analysis of resolvents is often a useful tool for establishing such estimates; see, for instance, \cite{MMT08, RT15, DatVas2012-gluing} in general settings, as well as the specific references in Examples~\ref{ex:lap-small-pert}--\ref{ex:obstacle} below. In particular, for high-frequency resolvent estimates in the presence of suitably unstable (i.e., hyperbolic) trapped bicharacteristics, as well as the more refined study of the associated scattering resonances, we refer to the original works of G\'erard--Sj\"ostrand \cite{GerSjo1988}, Wunsch--Zworski \cite{WunZwo2011}, Dyatlov \cite{Dya2015, Dya2016} and Nonnenmacher--Zworski \cite{NonZwo2015}, and also the textbook treatment in \cite[Chapter 6-7]{DZ19} and the references therein. 

\item In \ref{hyp:af-0}, the first three assumptions on $\bfh$, $b$, and $c$ are sharp, in the sense that the asymptotics \eqref{eq:linear-asymp} may fail if $\eta = 0$: Take, for instance, $\calH = -\Dlt - (d-3)r^{-1}\p_r$, which corresponds to the radial Schr\"odinger equation in a different dimension $d$. We remark already that the application of the Ifrim--Tataru testing-by-wave-packets method is crucial for handling all $\eta > 0$; see Section~\ref{subsec-ideas}.

\item The number of weighted derivatives assumed in \ref{hyp:af-0} is primarily motivated by Proposition~\ref{prop:good-comm} (and weighted elliptic regularity), where, however, \eqref{eq:af-extra2-0} is not used in its proof. On the other hand, \eqref{eq:af-extra2-0} is only assumed at the level of zeroth order and used in the proof of Proposition~\ref{prop:dotgmm_decay}.

\item The two additional assumptions in \ref{hyp:af-0} can be motivated as follows. Consider the Hamiltonian $H(x, \xi) = \bfh^{\mu\nu}(x)\xi_\mu \xi_\nu$.
The first assumption is to require $\psi = r$ to solve the associated Hamilton--Jacobi equation (eikonal equation) $0 = H(x,\nabla \psi) = h^{\mu\nu}\p_\mu\psi\p_\nu\psi$ up to sufficient precision. The second assumption is a refined characterization of the difference between $X_H r$ and zero, where the Hamilton vector field $X_H$ is given by \[
    X_H = 2h^{\mu\nu}\xi_\mu \p_{x^\nu} - (\p_{x^\lambda}h^{\mu\nu})(x) \xi_\mu \xi_\nu \p_{\xi_\lambda}.
\]
In addition, when $a^{\mu\nu}$ is written in the scattering form
\begin{equation*}
    a = a^{rr}\p_r \otimes \p_r + a^{r\tht^C} (\p_r \otimes \p_{\tht^C} + \p_{\tht^C} \otimes \p_r) + a^{AB} \p_{\tht^{A}} \otimes \p_{\tht^{B}},
\end{equation*}
the extra assumptions \eqref{eq:af-extra1-0} and \eqref{eq:af-extra2-0} may also be stated easily as \begin{align*}
    |\p_x^{(\leq M_c)} (\brk{x}\p_x)^{(\leq 2)} (a^{rr} - 1)| \aleq A_c t r^{-2-2\eta}, \quad   |a^{r\theta^C}\xi_C| \aleq A_c |\xi| t^{\frac12} r^{-1-2\eta}, \hbox{ for } r > t^{\frac12}.
\end{align*}

\item We formulate the theorem under the global bound \eqref{eq:E-u}, which, in most cases of interest, can be derived from the assumption \eqref{eq:D-u} with the constant $M_0$ replaced by $M_0 + 2$.
\end{enumerate}

\end{remark}

\subsubsection{The main theorem in the linear case, general version (Main Theorem~\ref{thm:linear})}
We now consider the general linear case,
\begin{equation}\label{eq:sch-gen}
    \begin{cases}
        (i\p_t - \calH)u = f, \quad &\hbox{ in } \calM^{1+3}, \\
        u \rst_{\set{t = 0}} = u_0 &\hbox{ on } \bbR^{3} \setminus \calK, \\
	    u = 0 &\hbox{ on } [0, \infty) \times \rd \calK.
    \end{cases}
\end{equation}

When an obstacle $\calK$ is present, we need to modify the existing assumptions. More specifically, it is natural to commute with $\partial_{t}$ instead of $\partial_{x}$, since only the former preserves the boundary condition on $\rd \calK$ in general. Given parameters $\eta > 0$, $M_{c} \in \bbZ_{\geq 0}$, and $A_{c} > 0$, we make the following assumptions in a way that unifies all cases:
\begin{enumerate}[label=$(\mathrm{AF})$]
\item (Asymptotic flatness)\label{hyp:af}
When $\calK = \0$, we assume \ref{hyp:af-0}. When $\calK \neq \0$, assume that
\begin{equation*}
    \rd \calK \hbox{ is a } C^{M_{c}+2} \hbox{ boundary},
\end{equation*}
that $a^{\mu \nu}$ is real and positive definite, and that the following holds true for all $(t, x) \in \br{\calM}$:
\begin{equation}
    \begin{array}{ll}
    \abs{\rd_{t}^{(\leq \lceil{\frac{M_c}{2}\rceil})} \rd_{x}^{(\leq M_c)} (\brk{x} \rd_{x})^{(\leq 3)}\bfh^{\mu \nu} (t, x)}
	+ \abs{\rd_{t}^{(\leq \lceil\frac{M_{c}}{2}\rceil)} \rd_{x}^{(\leq M_c)}\calS  \bfh^{\mu \nu} (t, x)} & \leq A_{c} \brk{x}^{-2 \eta}, \\
	\abs{\rd_{t}^{(\leq \lceil{\frac{M_c}{2}\rceil})} \rd_{x}^{(\leq M_c)} (\brk{x} \rd_{x})^{(\leq 2)} b^{\mu}(t, x)}
	+ \abs{\rd_{t}^{(\leq \lceil\frac{M_{c}}{2}\rceil)} \rd_{x}^{(\leq M_c)}\calS b^{\mu}(t, x)}& \leq A_{c}\brk{x}^{-1 - 2 \eta}, \\
	\abs{\rd_{t}^{(\leq \lceil{\frac{M_c}{2}\rceil})} \rd_{x}^{(\leq M_c)}(\brk{x} \rd_{x})^{(\leq 1)}c(t, x)} + \abs{\rd_{t}^{(\leq \lceil\frac{M_{c}}{2}\rceil)} \rd_{x}^{(\leq M_c)}\calS c(t, x)} &\leq A_{c} \brk{x}^{-2 - 2 \eta},
    \end{array}
\end{equation}
as well as the following extra assumptions for $r$:
\begin{align*}
	\bb|\rd_{t}^{(\leq \lceil\frac{M_{c}}{2}\rceil)} \p_x^{(\leq M_c)}(\brk{x}\p_x)^{(\leq 2)}(\bfh^{\mu \nu} (t, x) \rd_{\mu} r \rd_{\nu} r) \bb| \leq A_{c} t \brk{x}^{-2-2\eta} \quad \hbox{ for } \abs{x} > \brk{t}^{\frac{1}{2}}, \label{eq:af-extra1}\tag{AF-ex1}\\
    \bb| \bfh^{\mu \nu} (t, x) \rd_{\mu} r \xi_\nu \bb| \leq A_{c}|\xi| t^{\frac12} \brk{x}^{-1-2\eta} \quad \hbox{ for } \abs{x} > \brk{t}^{\frac{1}{2}}, \quad \xi \in \R^3.\label{eq:af-extra2}\tag{AF-ex2}
\end{align*}
\end{enumerate}

\begin{enumerate}[label=$(\mathrm{Th})$]
\item \label{hyp:se} (No threshold eigenvalue/resonance)
When $\calK = \0$, we assume \ref{hyp:se-0}. When $\calK \neq \0$, we assume instead that for every $t_{0} \geq 0$, $\calH \rst_{t=t_{0}}$ is invertible as a map $\dot{H}^{1}(\bbR^{3} \setminus \calK) \to \dot{H}^{-1}(\bbR^{3} \setminus \calK)$ and the operator norm of the inverse is uniformly bounded by $A_{c}$ independent of $t_0$.
Here, $\dot{H}^{1}(\bbR^{3} \setminus \calK)$ is the closure of $\set{v = \td{v} |_{\bbR^{3} \setminus \calK}: \td{v} \in \scrS(\bbR^{3}), \, \td{v} = 0 \hbox{ on } \p\calK}$ with respect to $\nrm{v}_{\dot{H}^{1}} = \nrm{\rd_{x} v}_{L^{2}(\bbR^{3} \setminus \calK)}$, and $\dot{H}^{-1}(\bbR^{3} \setminus \calK)$ is the dual of $\dot{H}^{1}(\bbR^{3} \setminus \calK)$.
\end{enumerate}

Next, we make an assumption on the forcing term $f$, using $D_{f} > 0$ as a quantifier for the size of the source:
\begin{enumerate}
[label=$(\mathrm{F})$]
    \item (Estimate on the source term)\label{hyp:f}
    For some $\dlt_f > 0$, for all $t > 0$, we assume the following for all $0 \leq N \leq M_c$.  
    When $\calK = \0$, we assume \[
        \|\brk{x}^2\p_x^{(\leq N+1)} f\|_{L^1L^2[0, t]} + t^{-(\frac14-\eta)}\|\p_x^{(\leq N)}\calL f\|_{\brk{x}^{-1}L^2L^2+L^1L^2[0, t]} + t^{\frac74 + \delta_f}\|\p_{x}^{(\leq N)} f(t, \cdot)\|_{L^2_x}
        \leq D_f.
    \]

    When $\calK \neq \0$, we assume \[
        \begin{aligned}
            &\sum_{2i+j \leq N+1}\|\brk{x}^2\p_t^{(\leq i)} \p_x^{(\leq j)} f\|_{L^1L^2[0, t]} + t^{-(\frac14-\eta)}\sum_{2i + j \leq N}\left(\|\p_t^{(\leq i)}\p_x^{(\leq j)}\calL f\|_{\brk{x}^{-1}L^2L^2+L^1L^2[0, t]} \right. \\
            &\qquad\qquad\qquad\qquad\qquad\qquad\qquad\qquad\qquad\qquad\qquad\qquad \left.+ t^{\frac74 + \delta_f}\|\p_t^{(\leq i)}\p_x^{(\leq j)} f(t, \cdot)\|_{L^2_x}\right) \\
            &+ t^{-(\frac14-\eta)} \sum_{2i + j \leq N}\left(\|t^{-\frac12-\eta} \p_t^{(\leq i)} \p_x^{(\leq j)} \calL f\|_{L^2L^2[1, t]} +  \|t^{-1 - \eta} \|\p_t^{(\leq i)}\p_x^{(\leq j)}\calL f\|_{L^2_x}\|_{L^1[1, t]} \right.\\
            &\qquad\qquad\qquad\qquad\qquad\qquad\qquad\qquad\qquad\qquad\qquad\qquad \left.+ \|t^{\frac54 - \eta}\|\p_t^{(\leq i)}\p_x^{(\leq j)} f\|_{L^\infty_x}\|_{\ell^1 L^2_t[1, t]}\right)
        \leq D_f,
        \end{aligned}
    \]
    and $f|_{[0,\infty) \times \p\calK} = 0$, where $\calL$ is defined in \eqref{eq:good-comm} and $\ell^1 L^2_t$ is defined in a similar fashion as the spatial norm in Section~\ref{subsubsec-Sobolev}.
\end{enumerate}

Finally, instead of \ref{hyp:iled-0}, we assume a variant that is suited to the general case, and is directly used in our proof (see Remark~\ref{rem:iled-iled*}.(2) below). We say that $L = i\p_t - \calH$ satisfies the \emph{dual integrated local energy decay estimates} up to order $M_{c}^*$ with derivative loss $s_c^*$ if the following holds:
\begin{enumerate}[label=$(\mathrm{ILED}^*)$]
\item ((Dual) integrated local energy decay)\label{hyp:iled*}
For every $t_0 \geq 0$, $T > 0$ and $N = 0, 1, \ldots, M_{c}^*$, the following holds for some $\bt_{c} \geq 0$:
\begin{itemize}[leftmargin=*]
\item When $\calK = \0$, for every $\psi_{0}: \bbR^{3} \to \bbC$ and $f: [t_{0}, t_{0}+T] \times \bbR^{3} \to \bbC$ such that
\begin{equation*}
\|\p_x^{(\leq N + s_c^*)} \psi_0\|_{L^2} + \nrm{\brk{x}^{\bt_{c} + 1} \rd_{x}^{(\leq N + s_c^*)} f}_{L^{2}[t_0, t_0 + T] L^{2}} < + \infty,
\end{equation*}
there exists a unique solution $\psi \in C_{t}([t_{0}, t_{0} + T]; H^{N})$ to $L\psi = f$ for $t \in [t_{0}, t_{0} + T]$ with $\psi(t_{0}) = \psi_{0}$. Moreover, the solution $\psi$ satisfies
\begin{equation*}
	\nrm{\rd_{x}^{(\leq N)} \psi}_{L^{\infty}[t_0, t_0 + T] L^{2}}
	\leq C_N\big( \|\p_x^{(\leq N + s_c^*)} \psi_0\|_{L^2} + \nrm{\brk{x}^{\bt_{c} + 1} \rd_{x}^{(\leq N + s_c^*)} f}_{L^{2}[t_0, t_0 + T] L^{2}}\big),
\end{equation*}
where $C_{N}$ is independent of $t_0 \geq 0, T > 0$.

\item When $\calK \neq \0$, for every $\psi_{0}: \calM^{3} \to \bbC$ and $f: [t_{0}, t_{0}+T] \times \calM^{3} \to \bbC$ such that
\begin{equation*}
 \|\p_x^{(\leq N + s_c^*)} \psi_0\|_{L^2} +  \sum_{2i + j \leq N + s_c^*} \nrm{\brk{x}^{\bt_{c} + 1} \rd_{t}^{(\leq i)}\rd_{x}^{(\leq j)} f}_{L^{2}[t_0, t_0 + T] L^{2}} < + \infty,
\end{equation*}
there exists a unique solution $\psi \in C_{t}([t_{0}, t_{0} + T]; H^{N})$ to $L\psi = f$ for $t \in [t_{0}, t_{0} + T]$ with $\psi(t_{0}) = \psi_{0}$ and $\psi = 0$ on $[t_{0}, t_{0} + T] \times \rd \calK$. Moreover, the solution $\psi$ satisfies
\begin{equation*}
	\nrm{\rd_{x}^{(\leq N)} \psi}_{L^{\infty}[t_0, t_0 + T] L^{2}}
	\leq C_N\big( \|\p_x^{(\leq N + s_c^*)} \psi_0\|_{L^2} +  \sum_{2i + j \leq N + s_c^*} \nrm{\brk{x}^{\bt_{c} + 1} \rd_{t}^{(\leq i)}\rd_{x}^{(\leq j)} f}_{L^{2}[t_0, t_0 + T] L^{2}}\big),
\end{equation*}
where $C_{N}$ is independent of $t_0 \geq 0, T > 1$.
\end{itemize}
\end{enumerate}

\begin{remark} \label{rem:hyp-dt}
    Assumptions \ref{hyp:af}, \ref{hyp:iled*}, and \ref{hyp:se} are formulated with time derivatives in order to accommodate the case of an exterior domain, where commuting $\p_t$ with the equation and using elliptic estimates is a natural way to control higher order derivatives of the solution. 
\end{remark}

\begin{remark}\label{rem:iled-iled*}
    We comment on the two additional assumptions \ref{hyp:f} and \ref{hyp:iled*} in Main~Theorem~\ref{thm:linear}:
    \begin{enumerate}[leftmargin=*]
        \item The first three bounds in \ref{hyp:f} are mainly used in the proof of Lemma~\ref{lem:base_iteration}, Proposition~\ref{prop:diff-u-gmmvarphi0-wo-1/tdv-involved} and Proposition~\ref{prop:dotgmm_decay}, respectively.
        \item In our construction, the boundary condition in \ref{hyp:f} is used crucially to ensure that $\calL u$ satisfies the Dirichlet boundary condition; see Remark~\ref{rem:Lu-BC}. See also Remark~\ref{rem:Lu-BC-other} for the sketch of an alternative (more technical) approach that avoids this feature.
        \item To connect \ref{hyp:iled-0} with \ref{hyp:iled*}, we refer to Proposition~\ref{prop:iled2iled-star}. The form we actually require to establish the theorem is formulated as Proposition~\ref{prop:iled*-L1L2+LE*}, which is derived as a corollary of \ref{hyp:iled*}.
        \item If \ref{hyp:iled-0} (or more precisely, its generalization \ref{hyp:iled} in the obstacle case stated in Appendix~\ref{sec:smoothing}) holds, then \ref{hyp:iled*} holds with the sharp exponent $\bt_{c} = 0$ according to Proposition~\ref{prop:iled2iled-star}. Roughly speaking, the idea is to use the sharp \emph{global-in-time local smoothing estimate} (Proposition~\ref{prop:kato-small}) to treat the large-$r$ region, which is applicable since $\calH$ is a small perturbation of $-\lap$ there. 
        
        We remark that \ref{hyp:iled*} with $\beta_c = 0$ is one of the two key ingredients for our ability to handle \ref{hyp:af} with the optimal (up to endpoint) range of $\eta > 0$, which is new in such generality. (The other key ingredient is the aforementioned application of the Ifrim--Tataru testing-by-wave-packet method \cite{IT15, IT24}.) We refer to Sections~\ref{subsubsec-iteration} and \ref{subsubsec-IT-wp} below for more discussion.
    \end{enumerate}
\end{remark}

The general version of our main theorem in the linear case is as follows.

\begin{maintheorem}[The linear case, general version] \label{thm:linear}
    Assume that \ref{hyp:af}, \ref{hyp:iled*} and \ref{hyp:se} hold true for $\calH$ with
\begin{equation*}
	\eta > 0, \quad \beta_c \geq 0, \quad s_c^* \in \bbZ_{\geq 0}, \quad M_c =  M_{c}^* \in \bbZ_{\geq 0}.
\end{equation*}
Let $s_{c}^{\prime} = s_{c}^{\ast} + \delta_s + 1$, where \[\delta_s:=
        \begin{cases}
             1, \quad \bfh^{\mu\nu} \not\equiv 0, \\
             0, \quad \bfh^{\mu\nu} \equiv 0.
        \end{cases}
    \]
Then there exists $c_0 > 0$ such that the following holds.
Let $N \in \bbZ_{\geq 0}$, and let $u: \calM^{1+3} \to \bbC$ be a solution to~\eqref{eq:sch-gen} satisfying the initial data bound \eqref{eq:D-u} and the a-priori (unweighted) energy bound \eqref{eq:E-u}, where the forcing term $f$ satisfies \ref{hyp:f}. Assume that $M_{c}$ and $M_{0}$ satisfy the bounds
\begin{equation*}
    \begin{cases}
        M_c \geq M_0, \hbox{ when } \calK = \0, \\
        M_c \geq M_0 + 2, \hbox{ when } \calK \neq \0,
    \end{cases}
    \quad
	M_{0} \geq N + c_{0} s_{c}^{\prime}.
\end{equation*}
Then the same conclusions as those in Main~Theorem~\ref{thm:linear-f=0} (namely the sharp time decay of the uniform norm and the sharp global asymptotics) hold for $u$, where $D$ is replaced by $D + D_f$ and the decay rate of the error in \eqref{eq:linear-asymp} is modified to $O_{L^\infty}(D t^{-\frac{3}{2} - \min\{\frac14\eta_\sharp, \delta_f\}})$.
\end{maintheorem}
\begin{remark}\label{rem:eta-sharp}
    For its proof, we refer to Section~\ref{sec-main-proof}. It is clear from its proof that the result holds for any $c_0$ such that $c_0 \geq \lceil\frac{8}{\eta_\sharp}\rceil$, where $\eta_\sharp:= \min\{\eta, \frac14-\}$. See also the convention for $\eta$ for future sections in Section~\ref{subsubsec-convention}.
\end{remark}

\subsubsection{Examples}
Before moving to the nonlinear setting, we provide some examples of $\calH$ that satisfy the assumptions of Main Theorem~\ref{thm:linear}, thereby demonstrating the scope of our approach.

\begin{example} \label{ex:lap-small-pert}
When $\calH$ is a small perturbation of $-\lap$ in the sense that $\calK = \0$ and \ref{hyp:af-0} holds with a sufficiently small $A_{c} < \ll 1$ (in particular, $a^{\mu \nu}$, $b^{\mu}$, and $c$ may be time-dependent, and the latter two may be complex-valued), then \ref{hyp:iled} and \ref{hyp:se} hold true. In particular, \ref{hyp:iled} is established in Proposition~\ref{prop:kato-small}, while \ref{hyp:se} can be derived taking advantage of Hardy's inequality. 
\end{example}

\begin{example} \label{ex:non-trapping}
    \ref{hyp:iled} and \ref{hyp:iled*} can be dropped under the following assumptions:
    \begin{enumerate}[leftmargin=*, label=\arabic*. ]
    \item $\calH$ satisfies \ref{hyp:af-0} and \ref{hyp:se-0} with $\calK = \0$;
    \item $a^{\mu \nu}$, $b^{\mu}$, and $c$ are time-independent and real-valued (or equivalently, $\calH$ is staionary and symmetric);
    \item the Hamiltonian vector field $H_a$ of $a(t,x,\xi) = a^{\mu\nu}(t,x) \xi_\mu \xi_\nu$ permits no trapped geodesics \cite[Definitions 1.10, 1.16]{MMT08};
    \item $\calH$ has no discrete spectrum.
    \end{enumerate}
    Indeed, \cite[Theorem 1.19]{MMT08} establishes \ref{hyp:iled*} with $\bt_{c} = 0$ under these assumptions.
\end{example}

\begin{example} [Obstacle problem] \label{ex:obstacle}
    When $\calH = -\lap$ with $\calK \neq \0$ (obstacle problem), \ref{hyp:af} and \ref{hyp:se} hold automatically. Moreover, it is well known that \ref{hyp:iled} and \ref{hyp:iled*} hold in the following cases:
    \begin{enumerate}[leftmargin=*, label=\arabic*. ]
    \item $\calK$ is star-shaped;
    \item $\calK$ is the union of one or two disjoint strictly convex obstacles; or
    \item $\calK = \bigcup_{i = 1}^{N} \calK_{i} \subseteq \bbR^{3}$ is the union of $N$ strictly convex obstacles, $\calK_{i}$, with $N > 2$ satisfying the following:
    \begin{itemize}[leftmargin=*]
        \item For any $1 \leq i, j, k \leq N$, pairwise distinct,  the convex hull of $\calK_{i} \cup \calK_{j}$ is disjoint from $\calK_{k}$.
        \item Denote by $\kappa$ the infimum of the principal curvatures of the boundaries of the obstacles $\calK_{i}$, and denote by $L$ the infimum of the distances between two obstacles. Then $\kappa L > N$.
    \end{itemize}
    \end{enumerate}
    Case~1 is essentially due to Lax--Morawetz--Phillips \cite{LaxMorPhi1963}, and Cases~2 and 3 (in particular, for two or more obstacles) are essentially due to Ikawa \cite{Ika1988}, both in the context of the wave equation; see also \cite{Ika1982, Ika1983, GerSjo1987, Ger1988}. To give more direct references, in Case~1 and Case~2 with one strictly convex obstacle, observe that there are no trapped geodesics, hence \ref{hyp:iled} follows directly from \cite[Proposition 2.7]{BGT04} (as is well known, in Case~1, multiplying the equation by a suitable radial vector field directly yields \ref{hyp:iled} \cite{LaxMorPhi1963}). In Case~2 with two strictly convex obstacles and Case~3, \ref{hyp:iled} follows from \cite[(4.5)]{Burq04}. In these cases, there exist trapped geodesics, but they are suitably unstable (or hyperbolic), thus \ref{hyp:iled} holds with a loss of derivatives.
\end{example}
 
\subsubsection{The main theorems in the nonlinear setting}\label{subsubsec-nonlinear-main-thm}
To illustrate the robustness of our physical-space approach, we apply it to derive the following global existence and asymptotics results for nonlinear problems.

Consider the Schr\"odinger equation with cubic gauge-covariant nonlinearity:
\begin{equation} \label{eq:nls}
	(i \rd_{t} - \calH) u = \calN(u)
\end{equation}
in the sense that $\calN(u)$ satisfies the following assumption:
\begin{enumerate}
[label=$(\mathrm{NL})$]
    \item (Types of nonlinearity)\label{hyp:nl}
    We assume that $\calN(u)$ satisfies phase rotation symmetry in the sense that \[
        \calN(\psi e^{i\tht}) = \calN(\psi) e^{i\tht}, \hbox{ for any } \tht \in \R, \quad \psi: \R^3 \to \bbC.
    \]
\end{enumerate}
This assumption guarantees that the Leibniz rule extends to $L_\mu$ (see \eqref{eq:L_mu}) in the presence of such nonlinearities. To illustrate the method, we consider two representative scenarios.

The first application concerns the quasilinear setting. For simplicity, we present it in the case of a cubic nonlinearity.

\begin{enumerate}
[label=$(\mathrm{NL-1})$]
    \item (Cubic nonlinearity with Hamiltonian quasilinearity)\label{hyp:nl-1}
    We assume that $\calN(u)$ satisfies \ref{hyp:nl} and that it has the following form \begin{align*}
    \calN(u):= D_{\mu} (G^{\mu \nu} |u|^2 D_{\nu} u) - G^{\mu\nu}u D_\mu \bar u D_\nu u + F(D \br{u}, D u, \br{u}, u).
    \end{align*}
    We impose the following assumptions on the $G^{\mu\nu}, F$:
    \begin{itemize}[leftmargin=*]
        \item We require that $G^{\mu \nu}$ is a constant symmetric real-valued matrix.
        \item Let $F: \bbC^{3} \times \bbC^{3} \times \bbC \times \bbC  \to \bbC$, $F = F(w, z, q, p)$ satisfy the following:
        \begin{itemize}[leftmargin=*]
        \item $F(D \br{u}, D u, \br{u}, u)$ is a linear combination of cubic gauge-covariant monomials.
        \item In the case $\calK \neq \0$, we assume that each summand contains at least one copy of $u$ or $\br{u}$ so as to ensure the boundary condition; see Remark~\ref{rem:Lu-BC}.
        \item Define
        \[
          \qquad\qquad B^{\mu}(w, z, q, p):= \rd_{z^{\mu}} F(w, z, q, p), \quad \ul B^\mu(w, z, q, p):= \rd_{w^{\mu}} F(w, z, q, p).
        \]
        Both expressions are linear combinations of quadratic gauge-invariant monomials. Moreover, we require that $B^\mu(D \br{u}, D u, \br{u}, u)$ be real-valued.
        \end{itemize}
    \end{itemize}
\end{enumerate}

When $\calH$ is symmetric and $F = 0$, note that $(i \rd_{t} - \calH) u = \calN(u)$ is precisely the Hamiltonian flow associated with the energy functional
\begin{equation*}
    E(u) = \frac{1}{2} \int \Re(\br{u} \calH u - G^{\mu \nu} \abs{u}^{2} D_{\mu} \br{u} D_{\nu} u) \, \ud x
\end{equation*}
with respect to the standard symplectic form. Hence, in case $G^{\mu \nu} \neq 0$, we say that $\calN(u)$ has \emph{Hamiltonian quasilinearity}. 

\begin{remark}[On Hamiltonian quasilinearity]
    The advantage of this assumption is that the nonlinearity possesses a favorable energy structure. We refer to \cite{FI22, Ia24, HM22} for the study of local well-posedness of such equations via energy methods in various settings.
\end{remark}

We also make the following assumptions to avoid the issue of derivative losses:
\begin{enumerate}[label=$(\mathrm{US})$]
\item(Ultimate stationarity) \label{hyp:ult-stat} In addition to \ref{hyp:af}, let $a^{\mu \nu}$, $b^{\mu}$, and $c$ satisfy the following time derivative bounds:
\begin{align*}
\nrm{\rd_{x}^{(\leq M_{c} + 1)} \rd_{t}^{(1 \leq \cdot \leq \lceil \frac{M_c}{2} \rceil)} a^{\mu \nu}}_{L^{1} L^{\infty} [0, \infty)} &+ \nrm{\rd_{x}^{(\leq M_{c} + 1)} \rd_{t}^{(1 \leq \cdot \leq \lceil \frac{M_c}{2} \rceil)} b^{\mu}}_{L^{1} L^{\infty}[0, \infty)} \\
&+ \nrm{\rd_{x}^{(\leq M_{c})} \rd_{t}^{(1 \leq \cdot \leq \lceil \frac{M_c}{2} \rceil)} c}_{L^{1} L^{\infty}[0, \infty)} \leq A_{c}.
\end{align*}
\end{enumerate}

\begin{enumerate}[label=$(\mathrm{E})$]
\item(Energy estimate) \label{hyp:en} Let $\td{A}^{\mu \nu} : [t_{0}, t_{1}] \times (\bbR^{3} \setminus \calK) \to \bbR$ be symmetric,  $\td{B}^{\mu}: [t_{0}, t_{1}] \times (\bbR^{3} \setminus \calK) \to \bbR$, and $\td{\ul{B}}^{\mu}: [t_{0}, t_{1}] \times (\bbR^{3} \setminus \calK) \to \bbC$ satisfy
\begin{gather*}
    \nrm{\rd_{\mu} \td{B}^{\mu}}_{L^{1} L^{\infty}[t_{0}, t_{1}]}
    + \nrm{\rd_{\mu} \ul{\td{B}}^{\mu}}_{L^{1} L^{\infty}[t_{0}, t_{1}]} \leq A_{\td{c}},
\end{gather*}
for some $A_{\td{c}} \geq 0$. Then for every $u \in \scrS(\bbR^{1+3})$ with $u = 0$ on $[0, \infty)_{t} \times \rd \calK$ (if $\calK \neq \0$), we have
\begin{equation*}
	\sup_{t \in [t_{0}, t_{1}]} \nrm{u(t)}_{L^{2}} \leq C_{A_{\td{c}}} \bb[\nrm{u(t_{0})}_{L^{2}} + \nrm{f}_{L^{1} L^{2}[t_{0}, t_{1}]}
    \bb],
\end{equation*}
where
\begin{equation*}
    f = (i \rd_{t} - \calH)u - D_{\mu} (\td{A}^{\mu \nu} D_{\nu} u) + \td{B}^{\mu} D_{\mu} u + \ul{\td{B}}^{\mu} D_{\mu} \br{u}.
\end{equation*}
\end{enumerate}

Note that \ref{hyp:ult-stat} is automatically satisfied if $a^{\mu \nu}$, $b^{\mu}$, and $c$ are time-independent; in particular, it holds for Examples~\ref{ex:non-trapping} and \ref{ex:obstacle}. In \ref{hyp:en}, we have allowed for the coefficients $\td{A}^{\mu \nu}$,  $\td{B}^{\mu}$, and $\ul{\td{B}}^{\mu}$ in anticipation of nonlinear applications. The following result shows that \ref{hyp:en} also holds for Examples~\ref{ex:non-trapping} and \ref{ex:obstacle}:
\begin{proposition} \label{prop:E-implied-by-stationarity}
Assume that \ref{hyp:af} holds true. Assume also that $a^{\mu \nu}$, $b^{\mu}$, and $c$ are time-independent and real-valued. Then \ref{hyp:en} holds.
\end{proposition}
\begin{proof}
    Under the assumptions, $\calH + D_\mu \td{A}^{\mu\nu} D_\nu$ is self-adjoint with respect to the complex $L^2$-inner product. We write \begin{align*}
        \p_t \int |u|^2 &= \brk{\p_t u, u} + \brk{u, \p_t u} = \brk{i\p_t u, iu} + \brk{iu, i\p_t u} \\
        &= i(\brk{\td{B}^\mu D_\mu u + \ul{\td{B}}^{\mu} D_{\mu} \br{u} - f, u} - \brk{u, \td{B}^\mu D_\mu u + \ul{\td{B}}^{\mu} D_{\mu} \br{u} - f}) \\
        &= \int\td{B}^\mu \p_\mu |u|^2 + \frac12\int \ul{\td{B}}^{\mu} \p_\mu \br{u}^2 + \frac12 \int \br{\ul{\td{B}}}^{\mu} \p_\mu u^2 + 2\Im \int f \br{u}
    \end{align*}
    and hence the result follows.
\end{proof}

Our first main result concerning a broad class of \eqref{eq:nls} is as follows.
\begin{maintheorem}  \label{thm:nonlinear-1}
Consider the initial(-boundary) value problem for \eqref{eq:nls}.
Assume that $\calH$ satisfies \ref{hyp:af}, \ref{hyp:iled*}, \ref{hyp:se}, \ref{hyp:ult-stat} and \ref{hyp:en} with
\begin{equation*}
	\eta > 0, \quad \beta_c \geq 0, \quad s_c^* \in \bbZ_{\geq 0}, \quad M_c =  M_{c}^* \in \bbZ_{\geq 0},
\end{equation*}
where $M_c$ is some sufficiently large number. Let $\calN(u)$ satisfy \ref{hyp:nl-1}.
Let $s_{c}^{\prime}$ be defined as in Main~Theorem~\ref{thm:linear}. Then there exists $c_0 > 0$ such that the following holds. Let $N \in \bbZ_{\geq 0}$, and assume that $M_{c}$ and $M_{0} \in 2\bbZ_{\geq 0}$ satisfy
\begin{equation*}
	M_{c} - 2 \geq M_{0} \geq N + 3 + c_{0} s_{c}^{\prime},
\end{equation*}
and that $\eps > 0$ is sufficiently small. Let $u_0$ satisfy the initial data bound \eqref{eq:D-u} with $D_0 \leq \eps$, and $M_0$ replaced by $M_0 + 2$ in the assumption \eqref{eq:D-u}. Then there exists a unique global solution $u: \calM^{1+3} \to \bbC$ to~\eqref{eq:nls} satisfying $u(0, x) = u_{0}(x)$. Moreover, the solution $u$ obeys the same sharp decay and asymptotics as in the conclusions of Main~Theorem~\ref{thm:linear} up to order $N$.
\end{maintheorem}

\begin{remark} \label{rem:nonlinear1}
This theorem provides new global existence results for the Schr\"odinger equation with cubic nonlinearity in the presence of variable coefficients and obstacles, particularly for nonlinearities involving derivatives, even including quasilinearities. The existing literature in 3D can be broadly categorized as follows:
\begin{itemize} [leftmargin=*]
\item \textbf{Constant-coefficient setting:} The study of cubic power-type nonlinearities has a long history, dating back to the foundational work of Klainerman \cite{Kla1982}, Shatah \cite{Sha1982}, Klainerman--Ponce \cite{KlaPon1983}, and later Constantin \cite{Const90} using the vector field method. We furthermore mention a new approach introduced by Ifrim--Tataru, initiated in \cite{IfrTat2023, IfrTat2022, IfrTat2025, IfrTat2023-survey}, which in particular yields sharp results for cubic quasilinear equations (without the Hamiltonian condition) with small \emph{non-localized data} \cite{IfrTat2024-higher}; such a result is beyond the scope of our method.

In the constant-coefficient case, global existence is known for some quadratic nonlinearities, which are also outside the scope of our method; for more discussion, see Section\ref{subsubsec.small-data-results}.

\item \textbf{Variable-coefficient setting:} Results are more limited, and are somewhat disjoint from our theorem. Notably, L\'eger \cite{Leger21} and Pusateri--Soffer \cite{PS24} treated \emph{quadratic} power-type nonlinearities in the presence of a potential using spacetime resonances and the distorted Fourier transform. To the best of our knowledge, there are no prior results for derivative or quasilinear nonlinearities in this setting, even for the cubic case.
    
\item \textbf{Obstacle scattering:} Previous global results for Schr\"odinger-type equations in exterior domains have been restricted to power-type nonlinearities; see, e.g., \cite{BGT04, Anton08, Iv10, BSS12, LSZ12, KVZ16, Laf18}. The case of (massless) wave equations is better studied, and it is known that small data global existence holds for quasilinear wave equations with null structure for the quadratic nonlinearity; see, e.g., \cite{MetSog2007,MetSog2006-starshaped,MetSog2006-highdim,MetSog2005-multispeed,MetSog2005,DafHolRodTay2022, DafHolRodTay2024}. 
\end{itemize}

Compared to these earlier works for Schr\"odinger-type equations, the main advantage of our method lies in its robustness, handling variable coefficients, obstacles, and derivative nonlinearities simultaneously.
\end{remark}

\begin{remark}\label{rem:comm_pt}
    Note that establishing high-regularity well-posedness for the cubic nonlinearity requires particular care in the choice of commuting vector fields, in contrast to earlier low-regularity results on exterior domains. In this setting, one must commute with the special derivative $\p_t$ in order to preserve the Dirichlet boundary condition at higher orders when carrying out energy estimates. 
    A similar predicament was recently encountered in the context of dispersive equations with rough nonlinearities \cite{PT25}, where higher regularity turns out to be more delicate than lower regularity, and where the time derivative also plays a crucial role.
\end{remark}

Our second application is the Hartree equation:
\begin{enumerate}
[label=$(\mathrm{NL-2})$]
    \item (Hartree-type nonlinearity)\label{hyp:nl-2}
    We may also allow for a Hartree-type cubic interaction: \begin{align*}
        \calN(u):= (|\cdot|^{-1} * |u(t, \cdot)|^2)(x) u(t, x).
    \end{align*}
\end{enumerate}
As in the constant-coefficient case, the global asymptotics of the solutions exhibit a logarithmic phase modification:

\begin{maintheorem}  \label{thm:nonlinear-2}
Consider the initial value problem for \eqref{eq:nls} with $\calK = \0$ and $\calN(u)$ satisfying \ref{hyp:nl-2}.
Assume that $\calH = -\Delta + c$ satisfies \ref{hyp:af}, \ref{hyp:se} and the estimate in Proposition~\ref{prop:iled*-L1L2+LE*} with $s_c' = 0$ and $M_c = 0$, where $c$ is a real potential.
Let $u_0$ satisfy the initial data bound \eqref{eq:D-u} with $D_0 \leq \eps$ and $M_0 = 2$. Then for $\eps$ sufficiently small, there exists a unique global solution $u$ to the above initial value problem such that $u \in L^\infty L^2(\calM)$ and $\calL u \in L^\infty L^2(\calM)$ (where $\calL$ is defined in \eqref{eq:good-comm}), which obeys the same sharp decay as in the conclusions of Main~Theorem~\ref{thm:linear} at its zeroth order. Moreover, it admits the following modified scattering asymptotics: \[
    u(t, x) = t^{-\frac32}e^{i\frac{|x|^2}{4t}} \gmm_\infty(\tfrac{x}{t}) \exp\left(-i\big(|\cdot|^{-1} * |\gmm_\infty(\cdot)|^2(\tfrac{x}{t})\big) \log t\right) \varphi_0(t, x) + O_{L^\infty}(\eps t^{-\frac32-\frac{\eta_\sharp}{12}}),
\]
where $\vphi_{0}$ is as in Main~Theorem~\ref{thm:linear-f=0}.
\end{maintheorem}

\begin{remark} \label{rem:nonlinear2}
    This theorem shows that our method applies even when the problem exhibits modified scattering. In the case of the constant-coefficient linear operator $\calH = -\lap$, this result goes back to Hayashi--Naumkin \cite{HN98}; see also \cite{KP11, TVH24} for other approaches. In particular, \cite{TVH24} implements the testing-by-wave-packets method \cite{IT15, IT24} in dimensions $2$ and $3$. We expect a similar result for more general $\calH$ and $\calN$, but we focus on the present case for simplicity.
\end{remark}

\subsection{Main ideas}\label{subsec-ideas}
The main ideas of our approach can already be demonstrated with the linear equation with a real-valued time-independent potential in the absence of obstacles, i.e., $\calH = -\Delta + c(x)$, $f = 0$ and $\calK = \0$ in \eqref{eq:sch-gen}. For simplicity, we assume that $c(x)$ satisfies \ref{hyp:af-0} with $M_c = 0$, i.e., 
\begin{equation} \label{eq:ideas:af}
\abs{c(x)} + \abs{\calS c(x)} \leq A_{c} \brk{x}^{-2 - 2 \eta},
\end{equation}
for $A_{c} > 0$ and a small exponent $\eta > 0$. We also assume \ref{hyp:se-0} (for which $c \geq 0$ is a sufficient condition). In addition, we assume that for any nice enough $w$, 
\begin{equation}\label{eq:main-idea-iled}
	   \nrm{w}_{L^{\infty}[t_0, t_1] L^{2}}
	\aleq \|w(t_0)\|_{L^2} +  \|(i\p_t - \calH)w\|_{(\brk{x}^{-1}L^2L^2 + L^1L^2)[t_{0}, t_{1}]},
\end{equation}
which can be derived as a corollary of \ref{hyp:iled*}; see Proposition~\ref{prop:iled*-L1L2+LE*} (with $s_c' = 0$, $\beta_c = 0$, $N = 0$). 

We now provide an overview of our entire approach for this simple case. The additional modifications needed to handle the more general case will be discussed in Section~\ref{subsubsec-mod-general-case}.

\subsubsection{Review of classical vector field approach}

We begin by reviewing why the classical vector field method of Klainerman \cite{Kl85} fails in this simple setting. The adaptation of Klainerman's vector field method to the Schr\"odinger setting to establish pointwise decay is well known; see, for instance, \cite{HT86, Const90}. In the case $c \equiv 0$, the method hinges on the following favorable commutation property: \[
    [i\p_t + \Delta, L_\mu] = 0,
\]
where $L_\mu:= x_\mu + 2it\p_{x^\mu}$. Furthermore, in view of the conjugation identity
\begin{equation} \label{eq:ideas:conj-id}
    e^{-i \mathring{\bfPhi}} L_{\mu} u = 2 i t \rd_{x^{\mu}} (e^{-i \mathring{\bfPhi}} u), 
\end{equation}
where $\mathring{\bfPhi}:= \frac{\abs{x}^{2}}{4 t}$ is the (real-valued) \emph{quadratic phase function} familiar from the asymptotics for the free Schr\"odinger equation, the sharp $t^{-\frac{3}{2}}$ pointwise decay for $\abs{u}$ follows from a rescaled Sobolev embedding for $U:= e^{-i \mathring{\bfPhi}} u$ (whose modulus equals that of $u$) on each ball $B_{t}(x)$ of radius $t$ (a step often referred to as the Klainerman--Sobolev embedding), and the boundedness of $\nrm{L^{(k)} u(t, \cdot)}_{L^{2}} = 2^k\nrm{t^{k} \rd^{(k)} U(t, \cdot)}_{L^{2}}$ for $k=0, 1, 2$. 

However, this approach breaks down as soon as a potential is introduced (i.e., $c(x) \neq 0$). Since \[
    [i\p_t - \calH, L_\mu] = 2it\p_{x^\mu} c,
\]
the commutator grows linearly in $t$. This property seems to render $L_{\mu}$ useless as a commutator, since the coefficients of $L_{\mu}$ grow only linearly in $t$ as well. (In fact, the twice $L_{\mu}$-commuted energy necessarily grows in time generically; see \eqref{eq:growth-general-LmuLnu} below.)

\subsubsection{The good commutator $\calL$} \label{subsubsec-good-comm}
Our starting point is to use a (single) weighted second-order operator $\calL$ -- which we call a \emph{good commutator} -- that avoids the preceding obstruction. In this simple setting, we may take
\begin{equation} \label{eq:ideas:good-comm}
	\calL:= \dlt^{\mu\nu}L_\mu L_\nu + 4 t^2 c.
\end{equation}
(For the systematic construction of $\calL$ in the general case, see Section~\ref{subsubsec-mod-general-case}.) The operator $\calL$ has two key features across different spatial scales:
\begin{enumerate}[leftmargin=*]
\item {\bf Far-field behavior.} As $\abs{x} \to +\infty$, we have the formal convergence of $\calL$ to the sum of squares of the classical operators $L_{\mu}$'s:
\begin{equation*}
    \calL \to \calL_{0}:= \dlt^{\mu \nu} L_{\mu} L_{\nu} \quad \hbox{ as } \abs{x} \to +\infty.
\end{equation*}
This ensures that $\calL$ retains the ``control'' properties of the free case in the region where the variable coefficients decay.

\item {\bf Near-field behavior.} On the other hand, for $\abs{x}$ small (say $\abs{x} < t^{1-\eta_{0}}$ for some $\eta_{0} > 0$), $L_{\mu} = 2 i t \rd_{x^{\mu}} + o(t)$ and thus 
\begin{equation*}
    \calL = 4 t^{2} (-\lap + c) + o(t^{2}) = 4 t^{2} \calH + o(t^{2})
\end{equation*} 
as $t \to +\infty$, ignoring derivatives and only focusing on the size of coefficients.
\end{enumerate}

The precise computation that reflects the improved commutation property is the following commutator identity:
\begin{equation} \label{eq:ideas:good-comm-id}
     [\calL, i\p_t - \calH] = -4it(\calS c + 2c),
\end{equation}
where we recall that, by our hypothesis \eqref{eq:ideas:af}, $\calS c$ obeys the same type of spatial decay bounds as $c$. In comparison to the largest weight $t^{2}$ in $\calL$, this commutator only has the weight $t$, i.e., it exhibits one order of cancellation (in terms of $t$-weights)!

We also provide a heuristic argument for the improved commutation property of $\calL$ compared to $L_{\mu} L_{\nu}$. For this purpose, it is instructive to begin with an {\it a-posteriori} explanation using the global asymptotics \eqref{eq:linear-asymp} that will be established\footnote{As remarked before, global asymptotics \eqref{eq:linear-asymp} have been known in some linear variable-coefficient problems using different methods; see, e.g., \cite{GRComHa2022, GRComHas2023, LooiSussman}.}. Indeed, if \eqref{eq:linear-asymp} holds, then one sees  that $L_\mu L_\nu u$ is growing from the following computation: in $\set{\abs{x} < 1}$, ignoring all terms with a better $t$-decay, we have
\begin{align*}
	L_\mu L_\nu u &= t^{-\frac32} e^{i\tfrac{|x|^2}{4t}} \gmm_\infty(\tfrac{x}{t}) ((2it)^2\rd_{x^{\mu}} \rd_{x^{\nu}}) \vphi_{0}(x) + \cdots,
\end{align*}
and thus under the generic assumption (which we do not justify in this heuristic argument) that $\rd_{x^{j}} \rd_{x^{k}} \vphi_{0}(x) \neq 0$ and $\gmm_\infty(\frac{x}{t}) \neq 0$  in $\set{\abs{x} < 1}$ for large $t$,  we have \begin{equation}\label{eq:growth-general-LmuLnu}
    \|L_\mu L_\nu u(t, \cdot)\|_{L^2_x} \ageq t^{\frac12}.
\end{equation}
However, we also see that $\calL$ behaves better, since the leading order term in \eqref{eq:linear-asymp} will get canceled upon application of $\calL$, thanks to $\calH \varphi_{0} = 0$. 

\begin{remark} \label{rem:good-comm-history}
The history of the operator $\calL$ spans several distinct contexts. In the case $\calH = -\lap + c$, a version of \eqref{eq:ideas:good-comm} appeared in \cite{HO89} to establish local smoothing, and later in \cite{CGV14} to prove time decay for 1D nonlinear equations with power-type nonlinearities. Notably, \cite{CGV14} utilized the distorted Fourier transform to analyze $\calL$, whereas our approach remains entirely in physical space. For Laplace--Beltrami operators on asymptotically conic manifolds, a similar operator was introduced in \cite{RT15} (see Remark~\ref{rem:RT15} for further comparison).

Beyond Schr\"odinger equations, the modification of commuting vector fields has roots in the study of the (massless) wave equation, particularly the pioneering works of \cite{ChrKla1993, Ali2003, Ali2005}. A second-order commuting operator was also utilized by Andersson--Blue \cite{AB15} to prove Morawetz estimates on Kerr black hole exteriors; see also \cite{Giorgi24}. For the broader development of ``good commutators'' in nonlinear dispersive equations, we refer the reader to \cite{IfrTat2019, IfrKocTat2023}.
\end{remark}

\subsubsection{\ref{hyp:se-0} and two-scale elliptic estimates}\label{subsubsec-two-scale}
Having established the favorable commutation relation \eqref{eq:ideas:good-comm-id}, we must now understand what control $\nrm{\calL u(t, \cdot)}_{L^{2}}$ confers on $u$. For this purpose, observe that, by \eqref{eq:ideas:conj-id} and \eqref{eq:ideas:good-comm}, we have the conjugation identity
\begin{equation} \label{eq:ideas:good-comm-conj}
    t^{\frac{3}{2}} e^{- i \mathring{\bfPhi}} \calL u = 4t^2\calH (t^{\frac{3}{2}} e^{- i \mathring{\bfPhi}} u).
\end{equation}

This identity suggests that estimates for $u$ can be recovered by applying elliptic theory for $\calH$. Matching the ``near-field'' and ``far-field'' features discussed in Section~\ref{subsubsec-good-comm}, we summarize our guiding principles for the elliptic analysis of $\calL$ as follows:
\begin{enumerate}[leftmargin=*]
    \item {\bf Inner scale ($r \aleq t^{\bt_{e}}$).} In a ball of time-dependent radius where $\calH$ may deviate significantly from $-\lap$, we apply a smoothly cutoff global elliptic estimate:
    \begin{equation} \label{eq:ideas:inner}
    \begin{aligned}
        	&t^{-2} \nrm{r^{\frac{1}{2}} L_\mu L_\nu u}_{\ell^{\infty} L^{2}(r < t^{\beta_{e}})} + t^{-1} \nrm{r^{-\frac{1}{2}} L_\mu u}_{\ell^{\infty} L^{2}(r < t^{\beta_{e}})}
	+ \nrm{r^{-\frac{3}{2}} u}_{\ell^{\infty} L^{2}(r < t^{\beta_{e}})} \\
    &\aleq t^{-2} \nrm{r^{\frac{1}{2}} \calL u}_{\ell^{1} L^{2}(r < 2 t^{\beta_{e}})} + \nrm{r^{-\frac{3}{2}} u}_{L^{2}(\frac12 t^{\beta_{e}} < r < 2 t^{\beta_{e}})},
    \end{aligned}    
    \end{equation}
    which relies on the assumption \ref{hyp:se-0} (see Section~\ref{subsec:ell-calH} for the notation $\ell^{r} L^2$ and more details).
    
    \item {\bf Outer scale ($r \ageq t^{\bt_{e}}$).} Outside this ball\footnote{In the analysis, it is convenient to further distinguish the region $r \geq t$ -- which we call the \emph{far-away zone} -- where $\nrm{L_{\mu} L_{\nu} u}_{L^{2}}$ is controlled by $\nrm{\calL u}_{L^2}$ and the standard energy $\nrm{u}_{L^2}$. See \eqref{eq:ell-r>t-U} and Remark~\ref{rem:zones}.}, we treat the difference $\calH - (-\lap)$ perturbatively and apply local elliptic regularity. This yields:
    \begin{equation} \label{eq:ideas:outer}
    \begin{aligned}
        &t^{-2+\frac{1}{2}\bt_{e}}\nrm{L_\mu L_\nu u}_{L^{2}(r > t^{\beta_{e}})} + t^{-1+\frac{1}{2}\bt_{e}}\|r^{-1}L_\mu u\|_{L^2(r > t^{\beta_{e}})} + t^{\frac{1}{2}\bt_{e}} \|r^{-2}u\|_{L^2(r > t^{\beta_{e}})} \\
	&\aleq t^{-2+\frac{1}{2}\bt_{e}}\nrm{\calL u}_{L^{2}(r > \frac12 t^{\beta_{e}})} + t^{\frac{1}{2}\bt_{e}} \nrm{r^{-2} u}_{L^{2}(r > \frac12 t^{\beta_{e}})}, 
    \end{aligned}   
    \end{equation}
    which does \emph{not} require \ref{hyp:se-0}.
\end{enumerate}
We normalize \eqref{eq:ideas:inner} so that the left-hand side controls $\nrm{u}_{L^{\infty}}$ via a rescaled Sobolev embedding on dyadic annuli, and we ensure that the left-hand sides of \eqref{eq:ideas:inner} and \eqref{eq:ideas:outer} match on the overlapping region $r \approx t^{\bt_{e}}$. We call this combined approach \textbf{two-scale elliptic analysis}. The specific exponent $\bt_{e}$ -- which partitions the two regimes -- will be determined after the discussion of the Ifrim--Tataru testing-by-wave-packet method in Section~\ref{subsubsec-IT-wp}.

\begin{remark}[Two-scale elliptic analysis without \eqref{eq:ideas:good-comm-conj}]
    While the identity \eqref{eq:ideas:good-comm-conj} already provides a simple relationship between $\calL$ and $\calH$ in the current setting, our two-scale elliptic analysis is designed for broader applicability. In many dispersive problems, such as the variable-coefficient Klein--Gordon equation \cite{OPT25+}, a simple conjugation identity may be unavailable, yet the underlying elliptic structure of the good commutator at different scales remains robust.
\end{remark}

\subsubsection{Application of Ifrim and Tataru's method of testing by wave packets}\label{subsubsec-IT-wp}

To bridge the gap between the control of $\nrm{\calL u(t, \cdot)}_{L^{2}}$ and the $L^\infty$ decay of $u$, and to subsequently obtain global asymptotics, we apply Ifrim and Tataru's \emph{method of testing by wave packets} \cite{IT15, IT24}. A \emph{test wave packet} traveling at velocity $v$ is defined as \[
    \Psi_v:= \chi^v e^{i\mathring{\bfPhi}}, \quad \chi^v:= \chi(\frac{x-tv}{\sqrt{t}})
\]
where $\chi$ is a non-negative bump function with normalization $\|\chi\|_{L^1} = 1$ and support property $\supp\chi \subset \{|z| \leq 1\}$. Observe that $\Psi_{v}$ is a wave packet for the free Schr\"odinger flow $i \rd_{t} + \lap$ with spatial localization $t^{\frac{1}{2}}$, and thus it remains coherent over a time scale $t$ (see also Remark~\ref{rem:wp-test-free}.(1) below). We then define the \emph{approximate profile} $\gamma(t, v)$ by testing $u(t, \cdot)$ against $\Psi_{v}$:
\[
\gmm(t, v):= \brk{u, \Psi_v},
\]
where we use bracket notation to denote the $L^2(\calM^3; \bbC)$-pairing. 

The key insight of \cite{IT15, IT24} is that testing against a wave packet with the time-dependent localization scale $t^{\frac{1}{2}}$ optimally balances two tasks: (1)~approximating $u(t, t v)$ via the approximate profile $\gamma(t, v)$, and (2)~analyzing the time evolution of $\gmm(t, v)$ itself. In our specific case, we establish the following \emph{approximation estimate} 
\begin{equation} \label{eq:ideas:wp-approx}
    \left\|t^{\frac{3}{2}} u(t, tv) -   \gmm(t, v) e^{i \mathring{\bfPhi}(t, tv)} \varphi_{0}(t, tv) \right\|_{L^\infty_v} 
    \aleq t^{-\frac14+\frac12\eta} \|\calL u(t, \cdot)\|_{L^2_x} + t^{\frac32-\frac{\eta}{4}} \|u(t, \cdot)\|_{L^\infty_x} + t^{-\frac14}\|u(t, \cdot)\|_{L^2_x},
\end{equation}
where the \emph{spatial profile} $\vphi_{0}$ solves $\calH \vphi_{0} = 0$ with $\vphi_{0} \to 1$ as $\abs{x} \to +\infty$ (see Remark~\ref{rem:wp-test-free}.(2) below), alongside the \emph{evolution estimate}
\begin{equation} \label{eq:ideas:wp-evol}
   t \|\tfrac{d}{dt}\gmm(t, v)\|_{L^\infty_v}\lesssim t^{-\frac14 + \frac{\eta}2} \|\calL u(t, \cdot)\|_{L^2} + t^{\frac32-\frac{\eta}{4}}\|u(t, \cdot)\|_{L^\infty}.
\end{equation}
Observe that the same powers occur in both estimates, reflecting the aforementioned optimization.

The method of testing by wave packets balances\footnote{In this strict sense, the Ifrim--Tataru method should be thought of as a \emph{phase-space} approach instead of physical- or Fourier-space. Our abuse of terminology in calling our approach ``physical-space'' refers to the useful fact that, as we sketched, it can nevertheless be implemented without reference to the Fourier transform.} the two prior approaches for computing the leading order asymptotics of $u$: ODE integration in physical space \cite{LS06} on the one hand, and tracking Fourier coefficients (cf.~the appearance of the Fourier transform in the Fraunhofer formula; see, e.g., \cite{KP11}) on the other. We further note that these two prior approaches seem unsuitable for our purpose: the former requires a pointwise bound on $\calL u$ (while we only prove an $L^{2}$ bound), and the latter requires a suitable variable-coefficient version of the Fourier transform, which we aim to avoid altogether. We also remark that the key spatial localization scale $t^{\frac{1}{2}}$ first appeared in the global outgoing parametrix construction in \cite{Tataru}. 

Notably, this wave-packet analysis determines the scale $t^{\beta_{e}}$ for our two-scale elliptic analysis. If we temporarily ignore the $L^\infty$ terms on the right-hand sides, \eqref{eq:ideas:wp-approx} and \eqref{eq:ideas:wp-evol} imply that $t^{3/2} \|u\|_{L^\infty}$ remains bounded provided that 
\begin{equation} \label{eq:Lu-grow}
    \nrm{\calL u(t, \cdot)}_{L^{2}} \aleq t^{\frac{1}{4}-\eta}. 
\end{equation}
Substituting this into the inner- and outer-scale elliptic estimates \eqref{eq:ideas:inner} and \eqref{eq:ideas:outer} reveals that we may take $\bt_{e} = 1/2$ (up to an $\eta$-loss, which we ignore) to retain the $t^{-\frac{3}{2}}$ decay of the left-hand sides. Thus, the spatial scale for the elliptic analysis coincides with the wave-packet localization scale $t^{1/2}$ (see also Remark~\ref{rem:1/2} below).

\begin{remark} \label{rem:wp-test-free}
We conclude this discussion with two technical observations regarding our implementation:
\begin{enumerate}[leftmargin=*]
    \item {\bf Free test wave packets and separation of scales.}
    Remarkably, although we study a variable-coefficient flow, we test against wave packets adapted to the \emph{free} Schr\"odinger flow ($i \rd_{t} + \lap$). This is made possible by a \emph{separation of scales}: the variable coefficients are localized at a time-independent scale $O(1)$, whereas the wave packet $\Psi_{v}$ spreads and relaxes according to the time-dependent scale $O(t^{1/2})$. The resulting interaction errors in the evolution estimate \eqref{eq:ideas:wp-evol} decay in time, allowing the free wave packet to remain a valid proxy.

    \item {\bf Refined two-scale elliptic estimates.}  
    Implementation of the testing-by-wave-packets method requires high-precision inputs from our elliptic analysis. In the outer scales ($r \gtrsim t^{1/2}$), we utilize a variant of Morrey's inequality (Lemma~\ref{lem:Morrey-out-scale}). In the inner scales ($r \lesssim t^{1/2}$), we apply a sharp elliptic asymptotics result that explicitly recovers the spatial profile $\vphi_{0}$ (see Proposition~\ref{prop:ell-asymp} and Lemma~\ref{lem:near-zone-asymp-zero-order}).
\end{enumerate}
    We mention that a similar implementation of the testing-by-wave-packets method in the presence of a potential in one dimension has appeared in \cite{Ste24} and \cite{KSW25}, where the flat wave packet was used in the outer scales.
\end{remark}

\subsubsection{\ref{hyp:iled*} and the main iteration scheme} \label{subsubsec-iteration}
The final component of the proof is an iteration argument that improves the decay rates step by step. We begin with the initial assumptions:
\begin{equation} \label{eq:ideas:iter-assume}
    \|\calL u\|_{L^2} \leq A t^{\frac14 - \eta + \alpha}, \quad \|u\|_{L^\infty} \leq A t^{-\frac32 + \alpha}.
\end{equation}
While crude energy estimates allow us to establish these bounds for a relatively large $\alpha$, our goal is to show that we can iteratively reduce $\alpha$ until it reaches zero. In each step, the constants improve in such a way that $A \leadsto A'$ (depending on $A$, $A_{c}$, and the data) and $\alpha \leadsto \alpha':= \max\{0, \alpha - \frac{\eta}{4}\}$. Since $\eta > 0$, this process terminates in finitely many steps, yielding the optimal $t^{-3/2}$ decay and sharp asymptotics (the latter is via testing by wave packets).

The iteration proceeds as follows. First, combining the wave-packet approximation \eqref{eq:ideas:wp-approx} and evolution estimate \eqref{eq:ideas:wp-evol} yields the desired improvement for $\|u\|_{L^\infty}$. Next, we use the ILED estimate \eqref{eq:main-idea-iled} to improve the bound on $\calL u$. Setting $w = \calL u$ in \eqref{eq:main-idea-iled} and using the commutator identity \eqref{eq:ideas:good-comm-id}, we have:
\begin{equation*}
    \nrm{\calL u}_{L^{\infty}[0, T] L^{2}} \aleq \|\calL u (0, \cdot)\|_{L^2} + \|t (\calS c + 2 c) u \|_{(\brk{x}^{-1}L^2L^2 + L^1L^2)[0, T]}.
\end{equation*}
To estimate the potential term, we split the domain at the spatial scale $t^{1/2}$ and apply the decay hypothesis \eqref{eq:ideas:af}:
\begin{align*}
    &\nrm{t (\calS c + 2 c) u}_{(\brk{x}^{-1}L^2L^2 + L^1L^2)[0, T]} \\
    &\aleq \nrm{t \brk{r}^{-2-2\eta} u}_{\brk{x}^{-1}L^2L^2 \set{0 < t < T, \, \abs{x} < t^{\frac{1}{2}}}}
    + \nrm{t \brk{r}^{-2-2\eta} u}_{L^1L^2 \set{0 < t < T, \, \abs{x} > t^{\frac{1}{2}}}} \\
    &\aleq t^{\frac{1}{4}-\eta + \alpha'} (t^{\frac{3}{2}-\alpha'} \nrm{u}_{L^{\infty}})
\end{align*}
These estimates provide the desired improvement of $\nrm{\calL u(t, \cdot)}_{L^{2}}$, completing the iteration step.

\begin{remark} [Spatial scale $t^{\frac{1}{2}}$] \label{rem:1/2}
The spatial scale $t^{1/2}$ appears in three different capacities in our analysis. It is: (1)~the localization scale for the test wave packets; (2)~the transition point between inner and outer elliptic estimates; and (3)~the boundary for splitting ILED and energy estimates.
\end{remark}

\begin{remark} [Comparison with \cite{RT15}] \label{rem:RT15}
The iteration scheme is primarily inspired by that in \cite[Section~16]{RT15}. Compared to \cite{RT15}, a key difference (other than generality) is our application of the Ifrim--Tataru testing-by-wave-packet method, which offers three main advantages. First, we obtain global asymptotics rather than just (almost optimal) time decay. Second, by allowing $\nrm{\calL u}_{L^{2}}$ to grow slightly as in \eqref{eq:Lu-grow}, we can handle optimal decay assumptions ($\eta > 0$ vs.~$\eta > 1/4$ in \cite{RT15}). Finally, our iteration terminates in finitely many steps because each step yields a uniform improvement in the decay rate.
\end{remark}

\subsubsection{Modifications in the general case}\label{subsubsec-mod-general-case}
We conclude this overview by noting the necessary adjustments for the more general settings considered in our Main Theorems.

\begin{enumerate}[leftmargin=*, label=\arabic*. ]
    \item For operators of the general form \eqref{eq:calH}, the good commutator $\calL$ is defined via the conjugation relation \eqref{eq:ideas:good-comm-conj}. This leads to a more involved commutator identity, which can be found in Proposition~\ref{prop:good-comm} below. We also remark that, in the case of $\bfh^{\mu\nu} \neq 0$ and $b^\mu \neq 0$, two-scale elliptic estimates need to be used in the estimate for $\nrm{\calL u(t, \cdot)}_{L^{2}}$, whereas in the above sketch, they are only used implicitly in the application of the testing-by-wave-packets method.
    
    \item When $\calK \neq \emptyset$, we modify the phase $\mathring{\bfPhi}$ to a cut-off version $\bfPhi^\ob$ supported away from the obstacle; see \eqref{eq:bfPhi-ob}. This ensures that $\calL u$ satisfies the Dirichlet boundary condition on $\rd \calK$, allowing us to recover the required energy estimates. While this modification is crucial for the proof, it is lower-order in $t$ and does not affect the leading-order asymptotics.

    \item In cases where the assumptions \ref{hyp:iled} or \ref{hyp:iled*} involve derivative losses, we implement an iteration scheme where each step may lose a fixed number of derivatives; see, in particular, Lemma~\ref{lem:main_iteration}. Because the iteration terminates in finitely many steps, we can afford these losses provided $M_c$ and $M_{0}$ (regularity of the coefficients and data, respectively) are sufficiently large compared to $s_{c}$ or $s_{c}^{\ast}$ (derivative loss in \ref{hyp:iled} or \ref{hyp:iled*}, respectively). The commutation procedure developed for this purpose also enables us to derive asymptotics for higher-order derivatives.

    \item Thanks to its robustness, the approach extends naturally to nonlinear problems. We proceed via a standard bootstrap argument in the case of Main Theorem~\ref{thm:nonlinear-1}, and we follow the roadmap in \cite{IT15, IT24} in the case of Main Theorem~\ref{thm:nonlinear-2}, where modified scattering occurs. We refer to Section~\ref{sec-nonlinear} below for details.
\end{enumerate}

\subsection{Other related works} \label{subsec-refs}
We now discuss the related works in the literature that have not been highlighted so far.

\subsubsection{Dispersive estimates for the linear variable-coefficient Schr\"odinger equation} \label{subsubsec-refs-var-coeff}

There is a huge literature on the time decay of solutions to the linear variable-coefficient Schr\"odinger equation. In addition to the references already cited above, we note the development in the case of potential perturbations (i.e., $a - \dlt, b = 0$) \cite{JSS91, GS04, ES04, RS04, Yaj05, G06, BG12, BK25}, and refer to the survey \cite{Sch2007}. For a textbook treatment, we cite \cite{KK12}, and for metric perturbations, we refer to \cite{SSS10-1, SSS10-2, BoucletBurq}. Finally, we also note that for suitably ``hyperbolic'' trapped geodesics, global Strichartz estimates \emph{without} loss of derivatives have been proved, which stands in contrast to the aforementioned global-in-time local smoothing estimates; see \cite{BGH10, MMTT10, HSTZ26}.

\subsubsection{Testing by wave packets}
As discussed earlier, the method of testing by wave packets was first introduced in the context of the (constant-coefficient) NLS \cite{IT15, IT24}, and has since been applied in a variety of equations; see, for instance, \cite{IT16, HGIfrTat2017, HM18, L19, Clo20, AA24, HIS25}. In addition, we point out its application to the one-dimensional derivative NLS by  \cite{By25-1, Shan25, By25-2}, as well as the works \cite{Ste24, KSW25}, which combine the distorted Fourier transform (discussed further below) with testing by wave packets in the presence of a potential for the one-dimensional Schr\"odinger equation.

\subsubsection{Small data scattering results} \label{subsubsec.small-data-results}
In the constant-coefficient setting, the study of small-data scattering for NLS has a rich history, notably through the foundational series of works by Hayashi--Naumkin (e.g., \cite{HN98, HayNau2000}) and the development of the \emph{spacetime resonance method} by Germain--Masmoudi--Shatah \cite{GerMasSha2009} (see also \cite{GerMasSha2012,GerMasSha2012-2D-NLS}). The latter is a powerful Fourier-based framework designed to capture cancellations arising from phase oscillations (normal forms) and group velocity differences; notably, these techniques have been successfully applied to handle various quadratic nonlinearities, in particular $\br{u}^{2}$ and $u^2$ (see \cite{GerMasSha2012}; see also \cite{HayNau2000, HayMizNau2003, Kaw2005} for a different approach).

It remains an interesting open question how to marry our physical-space approach with tools for capturing such nonlinear cancellations to treat quadratic nonlinearities in the variable-coefficient setting. We note that existing results in this direction typically rely on the distorted Fourier transform to extend the spacetime resonance method directly (see \cite{PS24}, discussed below). More recently, there has been a push toward obtaining more detailed information about the solution's behavior; we refer to \cite{JS26} for a study of higher-order asymptotics that further refines the description of the global behavior.

\subsubsection{Distorted Fourier transform (dFT), spectral methods, and spacetime resonance method}
The \emph{distorted Fourier transform} (dFT) represents a powerful alternative to the physical-space methods pursued in the present paper for studying variable-coefficient nonlinear problems. Much of the progress in dFT has focused on potential perturbations in one space dimension; see \cite{CGV14, Del16, Nam16, GPR18, CP22, Chen23, CP24} and the references therein. In the $(3+1)$-dimensional setting, a combination of dFT and the spacetime resonance method has been successfully employed to establish small-data scattering for quadratic nonlinearities such as $\bar{u}^{2}$ \cite{GHW15} and $u^{2}$ \cite{PS24} (cf.~ \cite{HayNau2000, HayMizNau2003, Kaw2005, GerMasSha2009} in the constant-coefficient case). The latter extends the earlier works \cite{Leger20, Leger21}, in which the spacetime resonance method was adapted for small potentials using a different spectral-theoretic approach. 

For other applications of spectral-theoretic techniques to nonlinear problems with variable coefficients, we refer to \cite{Schl21, SW25}. Finally, we mention related results concerning soliton scattering and stability in the presence of potentials, such as \cite{FGJS04, HZ08, CollotGermain23, LL24}.

\subsubsection{NLS outside an obstacle}
Strichartz estimates were established and used to prove local well-posedness; see, for instance, \cite{BGT04, Iv10, BSS12}.
Global well-posedness and scattering for the Schr\"odinger equation in the exteriors of balls were obtained via explicit representations in \cite{LSZ12}. Later, such a result for quintic NLS outside general domains was obtained in \cite{KVZ16}. See Remark~\ref{rem:comm_pt} for a comparison with our result.

\subsection{Organization of paper}
The remainder of this paper is structured as follows. In Section~\ref{sec:prelim}, we collect our notation and conventions and develop the basic elliptic theory for the operator $\calH$. In Section~\ref{sec-main-proof}, we introduce the fundamental components of our approach -- including the good commutator, test wave packets, and the approximate profile -- and provide a high-level proof of the linear case (Main Theorem~\ref{thm:linear}). The technical heart of the paper, which provides the rigorous foundations for the proof in Section~\ref{sec-main-proof}, is contained in the subsequent four sections:
\begin{itemize}
\item In Section~\ref{sec:good-comm}, we derive the good commutator identity.
\item In Section~\ref{sec:two-scale}, we develop the two-scale elliptic analysis.
\item In Section~\ref{sec:en-est-calLu}, we establish the energy estimates for $\calL u$.
\item In Section~\ref{sec:test-by-wp}, we implement the Ifrim--Tataru testing-by-wave-packets method.
\end{itemize}
Finally, in Section~\ref{sec-nonlinear}, we extend these techniques to the nonlinear setting to prove Main Theorems~\ref{thm:nonlinear-1} and \ref{thm:nonlinear-2}.

\stoptoc
\subsection*{Acknowledgments}
The authors would like to thank Ovidiu-Neculai Avadanei, Allison Byars, Mihaela Ifrim, Ben Pineau, Igor Rodnianski, Ely Sandine, Sohrab Shahshahani, Gavin Stewart, Zhongkai Tao, Daniel Tataru, and Maciej Zworski for insightful discussions.  

F.P.~acknowledges support from the Bergman Fellowship of the American Mathematical Society. 
S.-J.O.~was partially supported by a National Science Foundation
CAREER Grant under NSF-DMS-1945615, a National Science Foundation Grant under DMS-2452760, and the Miller Research Professorship.\resumetoc

\section{Preliminaries}\label{sec:prelim}

\subsection{Notation and conventions} \label{subsec:notation}
In this subsection, we collect pieces of notation and various conventions used throughout the paper.

\subsubsection{Notation}\label{subsubsec-notation}
\begin{itemize}[leftmargin=*, label= $\circ$]
\item Given two normed vector spaces $X$ and $Y$, we define the space $X \cap Y$ by the norm
\begin{equation*}
	\nrm{u}_{X \cap Y}:= \nrm{u}_{X} + \nrm{u}_{Y}.
\end{equation*}
\item If two normed vector spaces $X$ and $Y$ embed into a common vector space (which is always the case in practice when they are function spaces), we define the space $X + Y$ by the set of all sums $u_{1} + u_{2}$ with $u_{1} \in X$ and $u_{2} \in Y$, equipped with the norm
\begin{equation*}
	\nrm{u}_{X+Y}:= \inf\set{\nrm{u_{1}}_{X} + \nrm{u_{2}}_{Y}: u_{1} \in X, \, u_{2} \in Y, \, u = u_{1}+u_{2}}.
\end{equation*}
\item Given a vector space $X$ for functions on the space $\calM^{3}$ or spacetime $\calM^{1+3}$ and a positive function $w$, we define the space $w X$ to be the space of all elements $ w u$ for $u \in X$ equipped with the norm $\nrm{v}_{w X}:= \nrm{w^{-1} v}_{X}$. 
\item Given a norm $X$ for functions on space $\calM^{3}$, we write $L^{p} [t_{0}, t_{1}] X$ for the mixed norm  $\nrm{f}_{L^{p} [t_{0}, t_{1}] X}:= \nrm{\nrm{f(t)}_{X}}_{L^{p}_{t}[t_{0}, t_{1}]}$. We omit the time interval when it is $[0, \infty)$, or when it is clear from the context.
\item When we consider the restriction of a norm $X$ for a function $f$ to a subdomain $U$, we write $\nrm{f}_{X(U)}$. When $U = \set{t \in [t_{0}, t_{1}]}$, we simply write $X[t_{0}, t_{1}]$. (Note that, under this convention, $\nrm{f}_{L^{p} L^{q}[t_{0}, t_{1}]}$ coincides with  $\nrm{f}_{L^{p}[t_{0}, t_{1}] L^{q}}$.) 
\item We use the convention that lowercase Greek indices $\mu, \nu, \lmb, \ldots$ run over $\set{1, 2, 3}$ and the convention of {\it raising and lowering lowercase (resp.~uppercase) Greek indices} using the Kronecker delta symbol $\dlt_{\mu\nu}$  (resp.~$\dlt^{\mu\nu}$). For instance, $x_\mu:= x^\mu$ for $\mu = 1, 2, 3$.
We also adopt the usual convention of {\it summing up repeated lower and upper indices}.

\item We sometimes omit the subscript and use $L^{(k)}u$ to represent a chain of $k$ copies of $L_\mu$'s $(\mu = 1, 2, 3)$ applied to $u$. Similarly, we write $\rd_{x}^{(\leq k)}$ for a chain of at most $k$ copies of $\rd_{\mu}$'s ($\mu = 1, 2, 3$).

\item Abusing the notation slightly, we write $\rd_{x}^{k}$ for a chain of exactly $k$-copies of $\rd_{\mu}$'s, or equivalently, $\rd_{x}^{\alp}$ for a multi-index $\alp$ with $\abs{\alp} = k$; the specific multi-index being used shall be clear from the context.

\item We use the notation $r:= |x|$ and $r:= \brk{x}$ interchangeably when writing various weighted norms.

\item We write $A \aleq B$ for $A \leq CB$ for some positive constant $C$, which may vary from line to line. We specify the dependencies of $C$ using subscripts.

\item We use the following pieces of asymptotic notation: $g = O_x^{M}(\brk{x}^{-\gmm})$ and $g = O_{t,x}^{M}(\brk{x}^{-\gmm})$ mean $|\p_x^{(\leq M)} g| \lesssim \brk{x}^{-\gmm}$ and $|\p_t^{(\leq \lceil M/2\rceil)} \p_x^{(\leq M)} g| \lesssim \brk{x}^{-\gmm}$, respectively.

\item We use $\brk{\cdot, \cdot}$ to denote the $L^2(\R^3\setminus\calK; \bbC)$ pairing, i.e., \[
    \brk{f, g} = \int_{\R^3\setminus\calK} f \bar g\, \ud x.
\]
\end{itemize}

\subsubsection{Global conventions}\label{subsubsec-convention}
\begin{itemize}[leftmargin=*, label= $\circ$]
\item Without loss of generality, we assume in the following sections (unless otherwise specified) that the decay rate $\eta$ in \ref{hyp:af} is less than $\tfrac14$. In other words, we replace $\eta$ by $\min\{\eta, \tfrac14-\}$.

\item Without loss of generality, we assume that $0 \in \calK$ when $\calK \neq \0$.

\item We use $R_0$ (resp. $T_0$) to denote a sufficiently large radius (resp. time) in various results.
\end{itemize}

\subsubsection{An index of notation and parameters}
We list the notation used frequently in this paper and point out their first (few) appearances:
\begin{itemize}[leftmargin=*, label= $\circ$]
    \item $\calM$, $\calM^{3}$, $\calK$, $\calH, D_\mu$: Section~\ref{sec-intro};
    \item $a^{\mu\nu}, b^\mu, c, \bfh^{\mu\nu}$: \ref{hyp:af-0}, \ref{hyp:af};
    \item $\calN(u)$: \ref{hyp:nl};
    \item $G^{\mu\nu}$, $F$: \ref{hyp:nl-1};
    \item $\td A^{\mu\nu}, \td B^\mu, \td{\ul B}^\mu$: \ref{hyp:en};
    \item $\eta$: \ref{hyp:af-0}, \ref{hyp:af}, Section~\ref{subsubsec-convention};
    \item $\eta_\sharp$: Main~Theorem~\ref{thm:linear-f=0}, Remark~\ref{rem:eta-sharp};
    \item $s_c, s_c^{\ast}$, $\beta_c$: \ref{hyp:iled}, \ref{hyp:iled*};
    \item $s_c^{\prime}, \delta_s$: Main Theorem~\ref{thm:linear-f=0}, Main Theorem~\ref{thm:linear},  Proposition~\ref{prop:iled*-L1L2+LE*};
    \item $A_c$: \ref{hyp:af-0}, \ref{hyp:af}, \ref{hyp:ult-stat}, Proposition~\ref{prop:good-comm};
    \item $M_c, M_c^*$: \ref{hyp:af-0}, \ref{hyp:af},  \ref{hyp:iled-0}, \ref{hyp:iled*};
    \item $D_0$: \eqref{eq:D-u};
    \item $A_0$: \eqref{eq:E-u};
    \item $D_f$, $\delta_f$: \ref{hyp:f};
    \item $M_0, c_0$: Main Theorem~\ref{thm:linear-f=0}, Main Theorem~\ref{thm:linear};
    \item $D$: Main Theorem~\ref{thm:linear-f=0};
    \item $\varphi_0$: Main Theorem~\ref{thm:linear-f=0}, Main Theorem~\ref{thm:linear}, Lemma~\ref{lem:varphi0};
    \item $\tilde\varphi_0, \eta_0$: Lemma~\ref{lem:varphi0};
    \item $\gmm_\infty$: Main Theorem~\ref{thm:linear-f=0}, Main Theorem~\ref{thm:linear}, Section~\ref{subsec-main-iter};
    \item $\gmm$, $\gmm[\p^\alpha u]$: Definition~\ref{defn_wp};
    \item $\mathring{\bfPhi}, \bfPhi^\ob, \chi^\ob$: Section~\ref{subsec-good-comm-wp};
    \item $u, U$: \eqref{eq:relation-uU};
    \item $U^{(\alpha)}, U^{(k)}$: \eqref{defn:U-higher-order};
    \item $L_\mu$: \eqref{eq:L_mu};
    \item $\calS$: \ref{hyp:af-0}, Proof of Proposition~\ref{prop:good-comm};
    \item $\calL$: Definition~\ref{def:good-comm};
    \item $X(x), X_\mu(x)$: \eqref{def:X_mu};
    \item $\Psi_v$: Definition~\ref{defn_wp};
    \item $\chi^v$: Definition~\ref{defn_wp};
    \item $\tilde\chi^v$: Case 2 for $R_{5, k}$ in Section~\ref{sec-proof_prop:dotgmm_decay}.
\end{itemize}

\subsection{Elliptic theory for $\calH$} \label{subsec:ell-calH}
The purpose of this subsection is to review, record, and establish various results for the elliptic operator $\calH$ (for each fixed time). Specifically, we record standard elliptic estimates for $\calH$; establish the existence of the spatial profile $\vphi_0$ appearing in \eqref{eq:linear-asymp}; and prove a new asymptotic formula (Proposition~\ref{prop:ell-asymp}), which plays a key role in our implementation of the testing by wave packets method. 

Throughout this subsection, the following conventions are in effect:
\begin{itemize}[leftmargin=*]
\item {\bf In case $\calK \neq \0$, we assume that $\rd \calK$ is a $C^{M_{c}+2}$ boundary}.

\item {\bf We assume that $\calH$ satisfies \ref{hyp:af-0}}.

\item {\bf We fix a time $t_{0}$ and (abusing notation) write $\calH$ for $\calH |_{t = t_{0}}$}.
\end{itemize}
We note that all the results stated here concerning $\calH$ are independent of $t_{0}$ under \ref{hyp:af-0}.

\subsubsection{Local elliptic regularity for $\calH$}

We start with the standard $L^{2}$ local elliptic regularity theorem applied to $\calH$:
\begin{proposition}[Local elliptic regularity] \label{prop:dyadic-ell-0}
Consider $\calH$ satisfying \ref{hyp:af-0}. Fix $a > 1$, and in case $\calK \neq \0$, let $\calK \subseteq \set{r < a^{-1} R}$. Then for every $w \in \scrS(\bbR^{3})$, we have
\begin{equation} \label{eq:dyadic-ell-0}
	\sum_{\alp: \abs{\alp} \leq 2} R^{\abs{\alp}-2} \nrm{\rd_{x}^{\alp} w}_{L^{2}(R < r < 2 R)}
	\aleq_{a} \nrm{\calH w}_{L^{2}(a^{-1} R < r < 2 a R)} + R^{-2} \nrm{w}_{L^{2}(a^{-1} R < r < 2 a R)}.
\end{equation}
\end{proposition}
This result follows by rescaling the problem to the unit annulus $R = 1$ and applying the standard interior elliptic regularity result (in, say, \cite[Section~6.3.1]{Evans}). We omit the details.

\subsubsection{Weighted Sobolev spaces and global elliptic estimates for $\calH$}\label{subsubsec-Sobolev}
For $R \geq 1$, we introduce the norms
\begin{align*}
\nrm{u}_{\ell^{r} L^2(r < R)}
&= \left(\sum_{j > 0} \nrm{u}_{L^2(\set{2^{j-1} < r < 2^{j+1}} \cap \set{r < R} )}^{r} + \nrm{u}_{L^{2}(r < 1)}^{r} \right)^{\frac{1}{r}}, \\
\nrm{u}_{\ell^{r} H^{k, \dlt}(r < R)}
&= \sum_{\alp: \abs{\alp} \leq k} \left( \sum_{j > 0} 2^{(\abs{\alp} + \dlt) j} \nrm{\rd^{\alp} u}_{L^2(\set{2^{j-1} < r < 2^{j+1}} \cap \set{r < R})}^{r} + \nrm{\rd^{\alp} u}_{L^{2}(r < 1)}^{r} \right)^{\frac{1}{r}},
\end{align*}
with the usual modification when $r = \infty$. We omit $\ell^{r}$ when $r = 2$ (in which case the above norms coincide with the standard $L^{2}$ and weighted Sobolev spaces). We also omit $(r < R)$ in the whole space case (i.e., $R = + \infty$).

We start with the following standard result for the Laplacian $-\lap$ (see, for instance,  \cite[Proposition 2.10]{MST20}, as well as \cite[Lemma~7.7]{LukOh24} for $w_{(\geq 1)}$):
\begin{lemma} \label{lem:w-ell-lap}
Let $w \in \scrS(\bbR^{3})$. For any $\frac{1}{2} < \dlt < \frac{3}{2}$, we have
\begin{equation} \label{eq:w-ell-lap}
	\nrm{w}_{H^{2, \dlt-2}(\R^3)}
	\aleq_{\dlt} \nrm{(-\lap) w}_{H^{0, \dlt}(\R^3)}.
\end{equation}
In the endpoint case $\dlt = \frac{1}{2}$, we have
\begin{equation} \label{eq:w-ell-0-end-lap}
	\nrm{w}_{\ell^{\infty} H^{2, - \frac{3}{2}}(\R^3)}
	\aleq \nrm{(-\lap) w}_{\ell^{1}  H^{0, \frac{1}{2}}(\R^3)}.
\end{equation}
Denote by $w_{(0)}(x):= \frac{1}{\abs{\bbS^{2}}}\int_{\bbS^{2}} w(\abs{x} \tht) \, \ud \sgm_{\bbS^{2}}$ the spherical average of $w$, and $w_{(\geq 1)}:= w - w_{(0)}$. Then for any $-\frac{1}{2} < \dlt < \frac{5}{2}$, we have
\begin{equation} \label{eq:w-ell-1-lap}
	\nrm{w_{(\geq 1)}}_{H^{2, \dlt-2}(\R^3)}
	\aleq \nrm{(-\lap) w_{(\geq 1)}}_{H^{0, \dlt}(\R^3)}.
\end{equation}
In the endpoint case $\dlt = -\frac{1}{2}$, we have
\begin{equation} \label{eq:w-ell-1-end-lap}
	\nrm{w_{(\geq 1)}}_{\ell^{\infty} H^{2, -\frac{5}{2}}(\R^3)}
	\aleq \nrm{(-\lap) w_{(\geq 1)}}_{\ell^{1} H^{0, -\frac{1}{2}}(\R^3)}.
\end{equation}
\end{lemma}

From Lemma~\ref{lem:w-ell-lap} and the assumptions \ref{hyp:af} and \ref{hyp:se}, we obtain the following elliptic estimate for $\calH$ in weighted Sobolev spaces:
\begin{proposition}[Global elliptic estimate in weighted Sobolev space] \label{prop:w-ell-0}
Consider $\calH$ satisfying \ref{hyp:af-0} and \ref{hyp:se}, then for every $w \in \scrS(\bbR^{3})$ with $w|_{\p\calK} = 0$, we have
\begin{equation} \label{eq:w-ell-0}
	\nrm{w}_{\ell^{\infty} H^{2, - \frac{3}{2}}}
	\aleq \nrm{\calH w}_{\ell^{1}  H^{0, \frac{1}{2}}}.
\end{equation}
\end{proposition}
\begin{proof}
    Thanks to \ref{hyp:se}, $\calH$ admits an inverse operator $R_0: \dot H^{-1} \to \dot H^1$. It suffices to show that $\calH$ admits a right inverse $\td R_0: \ell^{1} H^{0, \frac{1}{2}} \to \ell^{\infty} H^{2, -\frac{3}{2}}$ which agrees with $R_0$ on $\ell^{1} H^{0, \frac{1}{2}} \cap \dot H^{-1}$.

    \noindent {\it Step 1: Construction of $\td R_0$ as a right inverse. }
    Given any $h \in \ell^{1} H^{0, \frac{1}{2}}$, we take $R$ sufficiently large and note that $\calH_{\mathrm{out}}$ (as defined in \eqref{eq:defn-Hout}) satisfies the desired weighted Sobolev estimate. Indeed, by noting that $\calH_{\mathrm{out}}$ is a small perturbation of $-\Delta$, we suppose $\calH_{\mathrm{out}} w_1 = h$, \begin{align*}
        \|w_1\|_{\ell^\infty H^{2, -\frac32}} \aleq \|h\|_{\ell^1 H^{0, \frac12}} + \|(\calH_{\mathrm{out}} + \Delta) w_1\|_{\ell^1 H^{0, \frac12}} \aleq \|h\|_{\ell^1 H^{0, \frac12}} + R^{-\eta}\|w_1\|_{\ell^\infty H^{2, -\frac32}},
    \end{align*}
    where we require here the weighted bounds in \ref{hyp:af-0}. Absorbing
    the small error to the left, it allows us to extend $\calH_{\mathrm{out}}^{-1}$ (well-defined between non-endpoint weighted spaces by duality) to $\ell^1 H^{0, \frac12} \to \ell^\infty H^{2, -\frac32}$. With a slight abuse of notation, we define $w_1:= \calH_{\mathrm{out}}^{-1} h$. Then \[
        \|w_1\|_{\ell^\infty H^{2, -\frac32}} \aleq \|h\|_{\ell^1 H^{0, \frac12}}.
    \]
    Since $-(\calH - \calH_{\mathrm{out}})w_1$ is compactly supported, we know that it is in $L^2_{comp} \subset L^{6/5} \subset \dot H^{-1}$. Define $w_2:= R_0(-(\calH - \calH_{\mathrm{out}})w_1) \in \dot H^1$ and then \begin{equation}\label{eq:calH-w2}
        \calH w_2 = -(\calH - \calH_{\mathrm{out}})w_1.
    \end{equation}
    Now we apply \ref{hyp:se} and obtain \begin{align*}
        \|w_2\|_{\ell^\infty H^{1, -\frac32}} \aleq \|w_2\|_{\dot H^1} \aleq \|\calH w_2\|_{\dot H^{-1}} \aleq \|h\|_{\ell^1 H^{0, \frac12}}.
    \end{align*}
    Then we could obtain from the standard elliptic regularity result that \begin{align*}
        \|w_2\|_{\ell^\infty H^{2, -\frac32}} \aleq \|\calH w_2\|_{\ell^1 H^{0, \frac12}} + \|w_2\|_{H^{1, \dlt'}} \aleq \|h\|_{\ell^1 H^{0, \frac12}},
    \end{align*}
    where the first inequality holds for any $\dlt' \in \R$ and we choose $\dlt' = -\frac32-$ to make the second one work.

    Now we are ready to introduce the following definition: \[
        \td R_0 h:= w_1 + w_2 = \calH_{\mathrm{out}}^{-1}h + R_0 (h - \calH \calH_{\mathrm{out}}^{-1} h), \quad \td R_0: \ell^1 H^{0, \frac12} \to \ell^\infty H^{2, -\frac32},
    \]
    and the discussion above reveals that \[
        \|\td R_0 h\|_{\ell^\infty H^{2, -\frac32}} \aleq \|h\|_{\ell^1 H^{0, \frac12}}.
    \]
    It is then obvious that $\calH \td R_0 h = h$ for $h \in \ell^{1} H^{0, \frac{1}{2}}$.

    \noindent {\it Step 2: Showing that $\td R_0$ is also a left inverse and it agrees with $R_0$ on their common domain. }
    For any $\varphi \in C^\infty_c$, we observe that $\calH (\td R_0 \varphi - R_0 \varphi) = 0$. Thus, by \ref{hyp:se}, we have $\td R_0 \varphi = R_0 \varphi$. By density $C_c^\infty \subset \ell^1 H^{0, \frac12} \cap \dot H^{-1}$, $R_0 = \td R_0$ on their common domain. Therefore, for any $h \in \ell^{1} H^{2, -\frac{3}{2}}$ (note that this is not taken to be $\ell^\infty$ since $\calH h$ might not be in $\ell^1 H^{0,\frac12}$ otherwise), we take $C^\infty_c \ni h_n \to h$ and hence $\td R_0 \calH h_n = R_0 \calH h_n = h_n$. This implies that $\td R_0 \calH h = h$ for all $h \in \ell^{1} H^{2, -\frac{3}{2}}$ and hence $\td R_0$ is a left inverse on $\ell^1 H^{2, -\frac32}$.
\end{proof}

\begin{remark}
    A Fredholm alternative statement for $\calH$ between weighted Sobolev spaces can be obtained and will also be used in Lemma~\ref{lem:varphi0}. Note that it cannot hold for $\ell^1-\ell^\infty$-based spaces since as mentioned above in Step 2, the domain of $\td R_0$ as the left inverse is modified. Thus, we have actual invertibility and a Fredholm alternative when considering the non-endpoint elliptic estimates based on $\ell^2$ for both norms.
\end{remark}

\subsubsection{Existence of $\varphi_0$ and its properties}
We now extract a function $\varphi_0$ from the spectral information, which appears in the main theorems to describe the asymptotics.

\begin{lemma} \label{lem:varphi0}
Assume that $\calH$ satisfies \ref{hyp:af-0} and \ref{hyp:se}. Choose $\eta_0 \in C^\infty(\overline{\calM^3})$ such that \begin{equation}\label{eq:defn-eta_0}
    \begin{cases}
        \eta_0:\equiv 1, \quad \calK = \0, \\
        \eta_0:= \begin{cases}
            1, \quad |x| > 3\sup_{y \in \p\calK} |y| \\
            0, \quad |x| < 2\sup_{y \in \p\calK} |y|
        \end{cases}, \quad \calK \neq \0.
    \end{cases}
\end{equation}
There exists a unique solution $\varphi_{0}$ to $\calH \varphi_{0} = 0$ such that $\varphi_0|_{\p\calK} = 0$ when $\calK \neq \0$ and $\tilde\varphi_0:= \varphi_{0} - \eta_0 \in H^{2, -\frac{3}{2}+\dlt'}$ for all $0 < \dlt' < \min \set{1, 2 \eta}$. Moreover, for $\dlt'$ in this range, we have
\begin{equation} \label{eq:varphi0-est}
	\nrm{\varphi_{0} - \eta_0}_{H^{2, -\frac{3}{2}+\dlt'}} \aleq A_c.
\end{equation}
We remark that $\eta_0$ is not unique and $\varphi_0$ does not depend on the choice of $\eta_0$.
\end{lemma}
\begin{proof}
By Proposition~\ref{prop:w-ell-0} and Fredholm alternative, it follows that $\calH^{-1}: H^{0, \dlt} \to \mathrm{Dom}(\calH)$ is well-defined if $\frac{1}{2} < \dlt < \frac{3}{2}$, where $\mathrm{Dom}(\calH)$ consists of those functions in $H^{2, \dlt-2}$ with Dirichlet boundary condition.
Note that $\calH \eta_0 \in H^{0, \dlt}$ for $0 < \dlt - \frac{1}{2} < 2\eta$ thanks to \ref{hyp:af-0}; moreover,
\begin{equation*}
\varphi_{0} = \eta_0 + \td{\varphi}_{0}, \hbox{ where } \td{\varphi}_{0}:= -\calH^{-1} \left( \calH \eta_0 \right) \in H^{2, \dlt-2},
\end{equation*}
where $\calH \eta_0$ is viewed as an element in $H^{0, \dlt}$ with $\frac{1}{2} < \dlt < \min\set{\frac{3}{2}, \frac{1}{2}+2 \eta}$. The desired estimate for $\varphi_{0} - \eta_0 = \td{\varphi}_{0}$ follows from Proposition~\ref{prop:w-ell-0}.

The fact that $\varphi_0$ is independent of the choice of $\eta_0$ follows from a contradiction argument. 
\end{proof}

\begin{corollary}[Elliptic regularity, uniform decay, and time regularity]\label{cor:decay_varphi0}
    Under the assumptions of Lemma~\ref{lem:varphi0}, the following result holds for any $\dlt' \in (0, \min(2\eta, 1))$: \begin{align*}
        \nrm{\p_x^{(\leq M_c)}\tilde\varphi_{0}}_{H^{2, -\frac{3}{2}+\dlt'}} &\aleq A_c, &
        \nrm{\p_x^{(\leq M_c-1)} t \rd_{t} \tilde\varphi_{0}}_{H^{2, -\frac{3}{2}+\dlt'}} &\aleq A_c, \\
        |\p_x^{(\leq M_c)}\tilde\varphi_0| & \lesssim A_c \brk{x}^{-\delta'}, &
        \quad |\p_x^{(\leq M_c-1)} t \rd_{t} \tilde\varphi_0| &\lesssim A_c \brk{x}^{-\delta'}.
    \end{align*}
\end{corollary}
\begin{proof}
    It follows from \ref{hyp:af-0} that $\calH\eta_0$ actually satisfies an improved regularity assumption \[\p_x^{(\leq M_c)}(\calH \eta_0) \in H^{0, \delta}.\]
    Writing \[
        \calH(\p_x^k \tilde\varphi_0) = [\calH, \p_x^k]\tilde\varphi_0 - \p_x^k \calH \eta_0 \in H^{0, \delta},
    \]
    one can argue by induction on $0 \leq k \leq M_c$ to upgrade the regularity and hence conclude the first assertion.

    For the second bound, we begin by noting that
    \begin{equation*}
        \calH(t \rd_{t} \tilde{\vphi}_{0})
        = - D_{\mu} (t \rd_{t} a^{\mu \nu} D_{\nu} \vphi_{0}) 
        - t \rd_{t} b^{\mu} D_{\mu} \vphi_{0}
        - D_{\mu} (t \rd_{t} b^{\mu} \vphi_{0}) - t \rd_{t} c \vphi_{0} \in H^{0, \dlt},
    \end{equation*}
    where we used \ref{hyp:af-0} and $t \rd_{t} = \calS - x \cdot \rd_{x}$. It follows that $t \rd_{t} \tilde{\vphi}_{0} \in H^{2, -\frac{3}{2}+\delta}$. Then, by an induction argument similar to the one above, the second assertion follows. (In this case, we can only go up to order $M_{c}-1$ due to the factors $D_{\mu} (t \rd_{t} a^{\mu \nu})$ and $D_{\mu} (t \rd_{t} b^{\mu})$.)

    For the third bound, we choose $\beta > 0$ such that $\beta + \delta' < \min \set{1, 2 \eta}$. Then it follows from weighted Sobolev embeddings that \begin{align*}
        \|\brk{x}^{\delta'}\p_x^{(\leq M_c)}\tilde\varphi_0\|_{L^\infty} \lesssim \|\brk{x}^{\delta'}\p_x^{(\leq M_c)}\tilde\varphi_0\|_{H^{2, -\frac32+\beta}} \lesssim \|\p_x^{(\leq M_c)}\tilde\varphi_0\|_{H^{2,-\frac32+\delta'+\beta}} \lesssim A_c.
    \end{align*}
    The last bound is proved similarly. 
\end{proof}
\begin{remark}
    In the obstacle case, the  Dirichlet boundary condition does not hold for higher derivatives anymore. Hence, one cannot simply obtain higher derivative bounds simply from inverting the operator $\calH$.
\end{remark}

\subsubsection{A precise asymptotic formula}

The following asymptotics will be useful later in Section~\ref{sec:test-by-wp}:
\begin{proposition}[Asymptotic formula] \label{prop:ell-asymp}
Assume that $\calH$ satisfies \ref{hyp:af-0} and \ref{hyp:se}. Given $R > \max\{1, 2\sup_{y \in \p\calK} |y|\}$, let $f_{R} \in L^{2}$ satisfy $\supp f_{R} \subseteq \set{\frac{1}{2} R < r < 2 R}$. Define $U_{R}:= \calH^{-1} f_{R}$. For every $\lmb \geq 2$ such that $\lmb^{-1}R > 3\sup_{y \in \p\calK} |y|$, we have the following estimate with constant independent of $\lambda$ and $R$:
\begin{equation*}
	\nrm{U_{R} - \frkc_{(0)}[f_{R}] \varphi_{0}}_{\ell^{\infty} H^{2, -\frac{3}{2}}(r < \lmb^{-1} R)} \aleq (\lmb^{-1} + A (\lmb^{-1} R)^{-\eta} )\nrm{f_{R}}_{H^{0, \frac{1}{2}}},
\end{equation*}
where
\begin{equation*}
\frkc_{(0)}[f_{R}] = \frac{1}{4 \pi} \int \frac{f_{R}(x)}{\abs{x}} \ud x.
\end{equation*}
\end{proposition}

\begin{proof}
According to our convention for $\eta$ in Section~\ref{subsubsec-convention}, $0 < \eta < \frac{1}{2}$. Moreover, we fix a choice of $\eta_0$ from Lemma~\ref{lem:varphi0} and extend $\eta_0$ trivially in $\calK$ such that $\eta_0 \in C^\infty(\R^3)$. We start with the proof of the obstacle-free case in which $\eta_0 \equiv 1$. Then we sketch the proof for the obstacle case by highlighting the key difference compared to $\calK = \0$.

\smallskip
\noindent {\it Step~1: Structure of $U_{R}$.}
We introduce the shorthands $f_{(0)} = (f_{R})_{(0)}$ and $f_{(\geq 1)} = f_{R} - (f_{R})_{(0)}$, and define
$v_{(0)} = (-\lap)^{-1} (f_{R})_{(0)}$, $v_{(\geq 1)} = (-\lap)^{-1} (f_{R})_{(\geq 1)}$. In view of spherical harmonic decompositions, $v_{(0)} = ((-\lap)^{-1} (f_{R}))_{(0)}$ and hence is spherically symmetric. The basic idea is to use $v_{(0)} + v_{(\geq 1)}$ (which equals $(-\lap)^{-1} f_{R}$) as the first approximation for $U_{R}$. Since
\begin{equation*}
\calH (U_{R} - v_{(0)} - v_{(\geq 1)}) = ((-\lap) - \calH) (v_{(0)} + v_{(\geq 1)}),
\end{equation*}
we arrive at the key formula
\begin{equation} \label{eq:ell-asymp-u}
U_{R} = v_{(0)} + \calH^{-1} ((-\lap) - \calH) v_{(0)} + v_{(\geq 1)} + \calH^{-1} ((-\lap) - \calH) v_{(\geq 1)}.
\end{equation}
In the remainder of the proof, we study each term in \eqref{eq:ell-asymp-u}.

\smallskip
\noindent {\it Step~2: The term $v_{(0)}$.}
By Lemma~\ref{lem:w-ell-lap}, we have the trivial bound
\begin{equation} \label{eq:ell-asymp-v0-0}
	\nrm{v_{(0)}}_{\ell^{\infty} H^{2, -\frac{3}{2}}} \aleq \nrm{f_{(0)}}_{H^{0, \frac{1}{2}}}.
\end{equation}
However, this result does not give any information about decay and one needs to extract more precise information as follows.
Observe that $v_{(0)}$ is radial and harmonic in $\set{r < \lmb^{-1} R}$, which already means that it is constant there.
In fact, we have the following exact formula:
\begin{equation} \label{eq:ell-asymp-v0}
	v_{(0)}(r) = \int f_{(0)}(r') r' \, \ud r' = \frkc_{(0)}[f_{R}] \quad \hbox{ for } r < \lmb^{-1} R.
\end{equation}
To see this, recall that, by the variation of constants formula,
\begin{equation}\label{eq:var-const-v}
	v(r):= \frac{1}{r} \int_{0}^{r} f(r') (r')^{2} \, \ud r' + \int_{r}^{\infty} f(r') (r') \, \ud r'
\end{equation}
defines the unique regular and decaying solution to $-(\rd_{r}^{2} + 2 r^{-1} \rd_{r}) v = f$.
Considering the boundedness of $\|v_{(0)}\|_{H^{2,-\frac32+\delta'}}$ with $\delta' > 0$, one knows that $v_{(0)}$ is decaying. Thus, the desired formula for $v_{(0)}$ in $\set{r < \lmb^{-1} R}$ is an immediate consequence of \eqref{eq:var-const-v} and the support property of $f_R$ when $\lambda \geq 2$.

\smallskip
\noindent {\it Step~3: The term $\calH^{-1} ((-\lap) - \calH) v_{(0)}$.}
By \eqref{eq:ell-asymp-v0}, we obtain
\begin{align*}
& \calH^{-1} \left(((-\lap) - \calH) v_{(0)}\right) \\
&= \calH^{-1} \left(\chi_{< \lmb^{-1} R}((-\lap) - \calH) v_{(0)}\right) +  \calH^{-1} \left((1 - \chi_{< \lmb^{-1} R})((-\lap) - \calH) v_{(0)}\right) \\
&= - \frkc_{(0)}[f_{R}] \calH^{-1} \left(\chi_{< \lmb^{-1} R} \calH 1\right) +  \calH^{-1} \left((1 - \chi_{< \lmb^{-1} R})((-\lap) - \calH) v_{(0)}\right) \\
&= \frkc_{(0)}[f_{R}] \td{\varphi}_{0}
+ \frkc_{(0)}[f_{R}] \calH^{-1} \left((1 - \chi_{< \lmb^{-1} R}) \calH 1\right)
+ \calH^{-1} \left((1 - \chi_{< \lmb^{-1} R})((-\lap) - \calH) v_{(0)}\right).
\end{align*}
Observe that the last two terms should be better in view of the exterior cut-off $(1 - \chi_{< \lmb^{-1} R})$ and \ref{hyp:af-0}. Indeed, by \ref{hyp:af-0}, \eqref{eq:ell-asymp-v0-0}, and the obvious bound $\frkc_{(0)}[f_{R}] \aleq \nrm{f_{(0)}}_{H^{0, \frac{1}{2}}}$, we have
\begin{equation}\label{eq:ell-asymp-v0'-prelim}
	\frkc_{(0)}[f_{R}]\nrm{(1 - \chi_{< \lmb^{-1} R}) \calH 1}_{H^{0, \frac{1}{2}}}
	+ \nrm{(1 - \chi_{< \lmb^{-1} R})(\calH - (-\lap)) v_{(0)}}_{H^{0, \frac{1}{2}}}
	\aleq A (\lmb^{-1} R)^{-2 \eta} \nrm{f_{(0)}}_{H^{0, \frac{1}{2}}}.
\end{equation}
By Proposition~\ref{prop:w-ell-0}, we arrive at
\begin{equation} \label{eq:ell-asymp-v0'}
\nrm{\calH^{-1} \left(((-\lap) - \calH) v_{(0)}\right) - \frkc_{(0)}[f_{R}] \td{\varphi}_{0}}_{\ell^{\infty} H^{2, - \frac{3}{2}}(r < \lmb^{-1} R)} \aleq A \lmb^{2 \eta} R^{-2 \eta} \nrm{f_{(0)}}_{H^{0, \frac{1}{2}}}.
\end{equation}

\smallskip
\noindent {\it Step~4: The contribution of $v_{(\geq 1)}$.} According to Lemma~\ref{lem:w-ell-lap}, $v_{(\geq 1)}$ enjoys a larger range of weighted elliptic estimates. Combined with the support property of $f_R$, we have
\begin{equation}\label{eq:ell-asymp-v1-end}
    R \nrm{v_{(\geq 1)}}_{\ell^{\infty} H^{2, -\frac{3}{2}-1}} \aleq \nrm{f_{(\geq 1)}}_{H^{0, \frac{1}{2}}}.
\end{equation}
Restricting the left-hand side to $\set{r < \lmb^{-1} R}$, it follows that
\begin{equation} \label{eq:ell-asymp-v1}
	\nrm{v_{(\geq 1)}}_{\ell^{\infty} H^{2, -\frac{3}{2}}(r < \lmb^{-1} R)}  \aleq \lambda^{-1}R\nrm{v_{(\geq 1)}}_{\ell^{\infty} H^{2, -\frac{3}{2}-1}(r < \lmb^{-1} R)} \aleq \lmb^{-1} \nrm{f_{(\geq 1)}}_{H^{0, \frac{1}{2}}}.
\end{equation}
Moreover, we decompose
\begin{align*}
\calH^{-1} ((-\lap) - \calH) v_{(\geq 1)}
= \calH^{-1} \left(\chi_{< R} ((-\lap) - \calH) v_{(\geq 1)} \right) +  \calH^{-1} \left((1 - \chi_{< R}) ((-\lap) - \calH) v_{(\geq 1)} \right)
\end{align*}
For the first term -- which is more delicate since it involves the interior where $(-\lap) - \calH$ may be large -- observe that \eqref{eq:ell-asymp-v1-end} and \ref{hyp:af-0} imply
\begin{align*}
	\nrm{\chi_{< R} ((-\lap) - \calH) v_{(\geq 1)}}_{H^{0, \frac{1}{2}}}
	&\aleq A \sum_{j \geq 0, \, 2^{j} \aleq R} 2^{(1-2 \eta) j} \nrm{v_{(\geq 1)}}_{\ell^{\infty} H^{2, -\frac{5}{2}}}
	\aleq A R^{-2 \eta} \nrm{f_{(\geq 1)}}_{H^{0, \frac{1}{2}}},
\end{align*}
since we assumed $\eta < \frac12$. By Proposition~\ref{prop:w-ell-0}, it follows that
\begin{equation} \label{eq:ell-asymp-v1'-in}
	\nrm{\calH^{-1} \left(\chi_{< R} ((-\lap) - \calH) v_{(\geq 1)} \right)}_{H^{2, -\frac{3}{2}}}
	\aleq A R^{-2 \eta} \nrm{f_{(\geq 1)}}_{H^{0, \frac{1}{2}}}.
\end{equation}
On the other hand, proceeding as in the end of Step~2, we obtain
\begin{equation} \label{eq:ell-asymp-v1'-out}
\nrm{\calH^{-1} \left((1 - \chi_{< R}) ((-\lap) - \calH) v_{(\geq 1)} \right)}_{H^{2, -\frac{3}{2}}} \aleq A R^{-2 \eta} \nrm{f_{(\geq 1)}}_{H^{0, \frac{1}{2}}}.
\end{equation}

\smallskip
\noindent {\it Step~5: Completion of proof for the obstacle-free case.} Combining \eqref{eq:ell-asymp-v0}, \eqref{eq:ell-asymp-v0'}, \eqref{eq:ell-asymp-v1}, \eqref{eq:ell-asymp-v1'-in}, and \eqref{eq:ell-asymp-v1'-out} with the formula \eqref{eq:ell-asymp-u} (and using $\lmb \geq 1$), the desired conclusion follows.

\smallskip
We now turn to the obstacle case and distinguish the steps by using the prime notation.

\smallskip
\noindent{\it Step~1': Structure of $U_R$.}
We keep the shorthand notation $f_{(0)}, f_{(\geq 1)}$ defined in Step 1 and recall the following notation
$v_{(0)} = (-\lap_{\R^3})^{-1} (f_{R})_{(0)}$, $v_{(\geq 1)} = (-\lap_{\R^3})^{-1} (f_{R})_{(\geq 1)}$. This is to emphasize that we do not use the Dirichlet Laplacian on the exterior domain but we invert the standard Laplacian on $\R^3$ by inserting a subscript.
We write
\begin{equation*}
\calH (U_{R} - \eta_0 v_{(0)} - \eta_0 v_{(\geq 1)}) = \left((- \calH)(\eta_0 v_{(0)}) - \Delta_{\R^3} v_{(0)} \right) + \left((- \calH)(\eta_0 v_{(\geq 1)}) - \Delta_{\R^3} v_{(\geq 1)} \right).
\end{equation*}
Since $U_{R} - \eta_0 v_{(0)} - \eta_0 v_{(\geq 1)}$ satisfies the Dirichlet boundary condition, we arrive at the key formula
\begin{equation} \label{eq:ell-asymp-u-ob}
U_{R} = \eta_0 v_{(0)} + \calH^{-1} \left((- \calH)(\eta_0 v_{(0)}) - \Delta_{\R^3} v_{(0)} \right) + \eta_0 v_{(\geq 1)} + \calH^{-1} \left((- \calH)(\eta_0 v_{(\geq 1)}) - \Delta_{\R^3} v_{(\geq 1)} \right).
\end{equation}
In the remainder of the proof, we study each term in \eqref{eq:ell-asymp-u-ob}.

\noindent {\it Step~2': The terms $\eta_0 v_{(0)}$ and $\eta_0 v_{(\geq 1)}$.}
Since $v_{(0)}$ and $v_{(\geq 1)}$ stay the same as before, \eqref{eq:ell-asymp-v0} still holds and \eqref{eq:ell-asymp-v1} carries over as well:  \[
    \nrm{\eta_0 v_{(\geq 1)}}_{\ell^{\infty} H^{2, -\frac{3}{2}}(r < \lmb^{-1} R)}  \aleq \lmb^{-1} \nrm{f_{(\geq 1)}}_{H^{0, \frac{1}{2}}}.
\]

\smallskip
\noindent {\it Step~3': The term $\calH^{-1} \left((- \calH)(\eta_0 v_{(0)}) - \Delta_{\R^3} v_{(0)} \right)$.}
We write
\begin{align*}
& \calH^{-1} \left((- \calH)(\eta_0 v_{(0)}) - \Delta_{\R^3} v_{(0)} \right) \\
&= \calH^{-1} \left(\chi_{< \lmb^{-1} R}\left((- \calH)(\eta_0 v_{(0)}) - \Delta_{\R^3} v_{(0)} \right)\right) +  \calH^{-1} \left((1 - \chi_{< \lmb^{-1} R})\left((- \calH)(\eta_0 v_{(0)}) - \Delta_{\R^3} v_{(0)} \right)\right) \\
&= - \frkc_{(0)}[f_{R}] \calH^{-1} \left(\chi_{< \lmb^{-1} R} \calH \eta_0\right) +  \calH^{-1} \left((1 - \chi_{< \lmb^{-1} R})\left((- \calH)(\eta_0 v_{(0)}) - \Delta_{\R^3} v_{(0)} \right)\right) \\
&= \frkc_{(0)}[f_{R}] \td{\varphi}_{0}
+ \frkc_{(0)}[f_{R}] \calH^{-1} \left((1 - \chi_{< \lmb^{-1} R}) \calH \eta_0\right)
+ \calH^{-1} \left((1 - \chi_{< \lmb^{-1} R})\left((- \calH)(\eta_0 v_{(0)}) - \Delta_{\R^3} v_{(0)} \right)\right).
\end{align*}
Since $\eta_0 = 1$ when $|x| > \lmb^{-1}R > 3\sup_{y \in \p\calK} |y|$, we have
\[
(1 - \chi_{< \lmb^{-1} R})\left((- \calH)(\eta_0 v_{(0)}) - \Delta_{\R^3} v_{(0)} \right) = (1 - \chi_{< \lmb^{-1} R})\left((- \calH) - \Delta_{\R^3} \right) v_{(0)}
\]
Estimating in the same way as in \eqref{eq:ell-asymp-v0'-prelim}, we arrive at
\begin{equation*}
\nrm{\calH^{-1} \left((- \calH)(\eta_0 v_{(0)}) - \Delta_{\R^3} v_{(0)} \right) - \frkc_{(0)}[f_{R}] \td{\varphi}_{0}}_{\ell^{\infty} H^{2, - \frac{3}{2}}(r < \lmb^{-1} R)} \aleq A \lmb^{2 \eta} R^{-2 \eta} \nrm{f_{(0)}}_{H^{0, \frac{1}{2}}}.
\end{equation*}

\smallskip
\noindent {\it Step~4': The contribution of $\calH^{-1} \left((- \calH)(\eta_0 v_{(\geq 1)}) - \Delta_{\R^3} v_{(\geq 1)} \right)$.}
Following the same proof as that of \eqref{eq:ell-asymp-v1'-in}, \eqref{eq:ell-asymp-v1'-out}, we proceed with the same observation of the support condition of $\eta_0$ as above and obtain
\begin{equation*}
\nrm{\calH^{-1} \left((- \calH)(\eta_0 v_{(\geq 1)}) - \Delta_{\R^3} v_{(\geq 1)} \right)}_{H^{2, -\frac{3}{2}}} \aleq A R^{-2 \eta} \nrm{f_{(\geq 1)}}_{H^{0, \frac{1}{2}}}.
\end{equation*}

\smallskip
\noindent {\it Step~5': Completion of proof.} Combining all the previous steps, the desired conclusion follows. \qedhere
\end{proof}

\section{Proof of Main Theorem~\ref{thm:linear}: main objects and the main iteration argument}\label{sec-main-proof}

In Section~\ref{subsec-good-comm-wp}, we introduce the main objects of our approach, namely, the \emph{good commutator} $\calL$ and the \emph{approximate profile} $\gmm(t, v)$, and state their main properties, namely, Propositions \ref{prop:est-calLu}, \ref{prop:diff-u-gmmvarphi0-wo-1/tdv-involved}, and \ref{prop:dotgmm_decay}. Then in Section~\ref{subsec-main-iter}, we provide a high-level proof of Main Theorem~\ref{thm:linear} assuming these propositions, whose proofs will be given in subsequent sections.

\subsection{The good commutator and the approximate profile} \label{subsec-good-comm-wp}
We introduce the \emph{quadratic phase function}
\begin{equation} \label{eq:std-phase}
	\mathring{\bfPhi} = \frac{\abs{x}^{2}}{4t}.
\end{equation}
which is a solution to the Hamilton--Jacobi equation associated with $H (\xi) = \dlt^{\mu \nu} \xi_{\mu} \xi_{\nu}$:
\begin{equation*}
	\rd_{t} \mathring{\bfPhi} + \dlt^{\mu \nu} \rd_{\mu} \mathring{\bfPhi} \rd_{\nu} \mathring{\bfPhi} = 0.
\end{equation*}

To handle the possible presence of an obstacle $\calK$, we make the following modification. Let $\chi^\ob(x)$ be a cut-off vanishing near the obstacle $\calK$ such that it is constant one away from the obstacle. In the case where $\calK = \0$, we choose $\chi^\ob \equiv 1$. Then we introduce the modified phase: \begin{equation} \label{eq:bfPhi-ob}
    \bfPhi^\ob:= \chi^\ob \frac{|x|^2}{4t}.
\end{equation}

\subsubsection{Good commutator}
We define the following operator: \begin{equation}\label{eq:L_mu}
    L_\mu = \chi^\ob x_\mu + \frac12(\p_\mu \chi^\ob)|x|^2 + 2it\p_\mu,
\end{equation}
which is adapted to the phase conjugation: \[
    t^{\frac32} e^{-i\bfPhi^\ob} \circ L_\mu \circ t^{-\frac32} e^{i\bfPhi^\ob} = 2it \p_{x^\mu} = -2t D_{x^\mu}.
\]
We will frequently switch between the notation $u = u(t,x)$ (the original solution in \eqref{eq:sch-gen}) and $U$, where a new variable $U = U(t, x)$ in terms of \begin{equation}\label{eq:relation-uU}
    u(t, x) = t^{-\frac32} e^{i\bfPhi^\ob} U(t, x).
\end{equation}
We note that this new variable was also employed effectively in \cite{LS06}.

Then the (modified) good commutator we will use throughout the paper is given as follows.
\begin{definition}\label{def:good-comm}
The \emph{good commutator} $\calL$ is defined as
\begin{equation}  \label{eq:good-comm}
	\calL:= L_{\mu} a^{\mu \nu} L_{\nu} - 2 t (b^{\mu} L_{\mu} + L_{\mu} b^{\mu}) + 4 t^{2} c.
\end{equation}
\end{definition}

We record a useful conjugation identity:
\begin{equation}\label{eq:conj_calLu_calHU}
    t^{\frac32}e^{-i\bfPhi^\ob}\calL u = 4t^2\calH U.
\end{equation}
For additional properties of $\calL$, we refer to Section~\ref{sec:good-comm}.

\begin{remark}\label{rem:Lu-BC}
 In the obstacle case (i.e., $\calK \neq \0$), the main purpose of utilizing the cut-off phase $\bfPhi^{\ob}$ is to ensure that the Dirichlet boundary condition holds for $\calL u$. Indeed, assuming \ref{hyp:f}, we have: \[
    \calL u|_{\p\calK} = 4t^{\frac12} \calH U|_{\p\calK} = 4t^2 \calH u|_{\p\calK} = 4t^2(i\p_t u - f)|_{\p\calK} = -4t^2 f|_{\p\calK} = 0.
\]
    Therefore, in the obstacle case, one could still commute $\calL$ into the equation and make sense of $\calL u$ as a solution so that \ref{hyp:iled} or \ref{hyp:iled*} would apply.
\end{remark}

\begin{remark}\label{rem:Lu-BC-other}
Another viable approach in the obstacle case is to define directly
\begin{equation*}
\calL^{\ob}:= \chi^{\ob} \calL + (1 - \chi^{\ob}) 4 i t^{2} \rd_{t},
\end{equation*}
where $\calL$ is defined using the exact quadratic phase $\mathring{\bfPhi}$. Then the preservation of the Dirichlet boundary condition is obvious (cf.~Remark~\ref{rem:Lu-BC}). Moreover, in view of the identity
\begin{equation*}
	\calL u = \calL^{\ob} u - (1-\chi^{\ob}) (4t^2f - (\calL - 4 t^{2} \calH) u),
\end{equation*}
and the observation that $\calL \approx 4 t^{2} \calH$ on $\supp(1-\chi^{\ob})$ for $t$ large, we may control $\calL u$ via $\calL^{\ob} u$ under appropriate assumptions on $f$. While we will not carry out the details of this other approach, we point out that it has the advantage of not requiring $f |_{\rd \calK}$ to vanish, and hence a larger class of nonlinearities than \ref{hyp:nl-1} would be admissible.
\end{remark}

One fundamental property of $\calL$ is that the commutator $[i \rd_{t} - \calH, \calL]$ exhibits nontrivial cancellations; see Proposition~\ref{prop:good-comm}. Another important property is that, thanks to the weights in \eqref{eq:L_mu} and \eqref{eq:good-comm}, decay properties of $u$ follow from a control of $\calL u$; this is the purpose of the two-scale elliptic analysis in Section~\ref{sec:two-scale}. In this section, we only write down the final consequence of these properties, namely, a good \emph{energy estimate} for $\calL u$:

\begin{proposition}[Energy estimate with the good commutator] \label{prop:est-calLu}
    Assume that $\calH$ satisfies \ref{hyp:af}, \ref{hyp:iled*} (with $M_c = M_c^*$) and \ref{hyp:se}. Then there exists $T_0$ such that for all $t > T_0 + 1$, for all $0 \leq N + s_c' \leq M_c$, for any solution $u$ to \eqref{eq:sch-gen},
    \begin{itemize}[leftmargin=*]
        \item $\calK = \0$: \begin{equation}\label{eqn:est-calLu}
    \begin{aligned}
        \|\p_x^{(\leq N)} \calL u\|_{L^\infty L^2[T_0, t]} \aleq\, &\|\p_x^{(\leq N + s_c')} \calL u(T_0, \cdot)\|_{L^2}
        +  \|\p_x^{(\leq N + s_c^\prime)}\calL f\|_{\brk{x}^{-1}L^2L^2+L^1L^2[T_0, t]} \\
        &+ \|t^{-\frac12 - \eta}\p_x^{(\leq N + s_c^\prime)} \calL u\|_{L^2L^2[T_0, t]}
        + \|t^{-\frac14 - \eta}\|t^{\frac32}\p_x^{(\leq N + s_c^\prime)} u\|_{L^\infty_x}\|_{\ell^1L^2_t[T_0, t]} \\
        & + \|t^{-1 - \eta} \|\p_x^{(\leq N + s_c^\prime)}\calL u\|_{L^2_x}\|_{L^1[T_0, t]}
        + \|t^{-1 - 2\eta}\|\p_x^{(\leq N + s_c^\prime)} u\|_{L^2_x}\|_{L^1_t[T_0, t]},
    \end{aligned}
    \end{equation}
    \item $\calK \neq \0$: \begin{align*}
        \|\p_x^{(\leq N)} \calL u\|_{L^\infty L^2[T_0, t]} \aleq\,
        &\|\p_x^{(\leq N + s_c')} \calL u(T_0, \cdot)\|_{L^2} + \sum_{2i + j \leq N + s_c'} \|\p_t^{(\leq i)}\p_x^{(\leq j)}\calL f\|_{\brk{x}^{-1}L^2L^2+L^1L^2[T_0, t]} \\
        &+ \sum_{2i + j \leq N + s_c'}\bb(\|t^{-\frac12-\eta} \p_t^{(\leq i)} \p_x^{(\leq j)} \calL f\|_{L^2L^2[T_0, t]} +  \|t^{-1 - \eta} \|\p_t^{(\leq i)}\p_x^{(\leq j)}\calL f\|_{L^2_x}\|_{L^1[T_0, t]} \\
        & \relphantom{+ \sum_{2i + j \leq N + s_c'}\bb(} + \|t^{\frac54 - \eta}\|\p_t^{(\leq i)}\p_x^{(\leq j)} f\|_{L^\infty_x}\|_{\ell^1 L^2_t[T_0, t]} + \|t^{-1 - 2\eta}\|\p_t^{(\leq i)}\p_x^{(\leq j)} f\|_{L^2_x}\|_{L^1_t[T_0, t]} \bb) \\
       &  + \|t^{-\frac12 - \eta}\p_x^{(\leq N + s_c')} \calL u\|_{L^2L^2[T_0, t]}
         + \|t^{-\frac14 - \eta}\|t^{\frac32}\p_x^{(\leq N + s_c')} u\|_{L^\infty_x}\|_{\ell^1 L^2_t[T_0, t]} \\
       &+ \|t^{-1 - \eta} \|\p_x^{(\leq N + s_c')}\calL u\|_{L^2_x}\|_{L^1[T_0, t]}
         + \|t^{-1 - 2\eta}\|\p_x^{(\leq N + s_c')} u\|_{L^2_x}\|_{L^1_t[T_0, t]} .
    \end{align*}
    \end{itemize}
\end{proposition}

We remark that, in each case, the terms on the last two lines will be removed (after possibly losing more $\rd_{x}$-derivatives) via an iteration argument; see Section~\ref{subsec-main-iter}. Proposition~\ref{prop:est-calLu} will be proved in Section~\ref{sec:en-est-calLu}. 

\subsubsection{Wave packet and approximate profile}
In order to implement the testing-by-wave-packet method in Section~\ref{sec:test-by-wp}, we introduce:
\begin{definition}[Wave packet and approximate profile]\label{defn_wp}
We define the \emph{(test) wave packet traveling at velocity $v$} as \[
    \Psi_v:= \chi^v e^{i\bfPhi^\ob}, \quad \chi^v:= \chi(\frac{x-tv}{\sqrt{t}})
\]
where $\chi$ is a non-negative radial bump function with normalization $\|\chi\|_{L^1} = 1$ and support property $\supp\chi(z) \subset \{|z| \leq 1\}$.
As a measure of the decay of $u$ along $\{x = vt\}$, we introduce a function $\gmm[u] = \gmm[u](t, v)$, \emph{the approximate profile of $u$}, given by
\[
\gmm[u](t, v):= \brk{u, \Psi_v},
\]
where we use bracket notation to denote the $L^2(\calM^3; \bbC)$-pairing. Sometimes we will work with the alternative formula:
\begin{equation}\label{eq:defn_gmm}
        \gmm(t, v) = \int_{\calM^3} u(t, x)\overline{\Psi_v}\, \ud x = t^{-\frac{3}{2}}\int_{\calM^3} U(t, x)\chi\big(\frac{x-vt}{\sqrt{t}}\big)\, \ud x = t^{\frac32} \int_{\R^3 \setminus t^{-1}\calK} U(t, t y)\chi\big(\sqrt{t}(y-v)\big)\, \ud y.
    \end{equation}
\end{definition}

To deal with higher order derivatives, we also introduce its higher order counterpart \[
    \gmm[\p^\alpha u](t, v):= \brk{\p^\alpha u, \Psi_v}, \hbox{ for all multi-index } \alpha.
\]
    
The first main property of $\gmm$ is that it describes the precise asymptotics of $u$, given an appropriate control of $\calL u$:
\begin{proposition}[Approximation via $\gmm$] \label{prop:diff-u-gmmvarphi0-wo-1/tdv-involved}
    Assume that $\calH$ satisfies \ref{hyp:af-0} and \ref{hyp:se}. Then there exists $T_0$ such that for all $t > T_0$, given any $u$ satisfying \eqref{eq:E-u}, the following pointwise bound holds \begin{align*}
    &\left\|t^{\frac32}(\p_x^\alpha u)(t, tv) - \sum_{ \bt_1 + \bt_2 = \alp, N := |\alp|} \binom{\alp}{\bt_1}\gmm[\p_x^{\beta_1} u](t, v) e^{i\bfPhi^\ob(t, tv)} (\p_x^{\beta_2} \varphi_0)(t, tv) \right\|_{L^\infty_v(\{v \in \overline{\R^3 \setminus t^{-1} \calK}\})} \\
    \aleq\, &t^{-\frac14+\frac12\eta} \|\p_x^{(\leq N)}\calL u(t, \cdot)\|_{L^2_x} + t^{\frac32-\frac{\eta}{4}} \|\p_x^{(\leq N)} u(t, \cdot)\|_{L^\infty_x} + t^{-\frac14}\|\p_x^{(\leq N)}u(t, \cdot)\|_{L^2_x}, \hbox{ for all } 0 \leq N \leq M_0.
\end{align*}
\end{proposition}

Another key property of $\gmm$ is that it obeys a nice evolution equation, again given an appropriate control of $\calL u$:
\begin{proposition}[Evolution of $\gmm$] \label{prop:dotgmm_decay}
    Assume that $\calH$ satisfies \ref{hyp:af-0} and \ref{hyp:se}. Then there exists $T_0$ such that for all $t > T_0$, given any solution $u$ to \eqref{eq:sch-gen} satisfying \eqref{eq:E-u}, the time derivative of $\gmm$ satisfies the following bound uniformly in $v$ for any multi-index $\alp$ with $N:= \abs{\alp}$, $0 \leq N \leq M_c$: \begin{align*}
        \|\tfrac{d}{dt}\gmm[\p^{\alp} u](t, v) + i\int \p_x^{\alp} f \overline{\Psi_v}\, \ud x \|_{L^\infty_v(\{v \in \overline{\R^3 \setminus t^{-1} \calK}\})}\lesssim t^{-\frac54 + \frac{\eta}2} \|\p_x^{(\leq N)}\calL u(t, \cdot)\|_{L^2} + t^{\frac12-\frac{\eta}{4}}\|\p_x^{(\leq N)} u(t, \cdot)\|_{L^\infty} + t^{-\frac54}\|\p_x^{(\leq N)} u\|_{L^2}  .
    \end{align*}
\end{proposition}

Finally, we also state some auxiliary properties of $\gmm$ that we will need in the proof of Main Theorem~\ref{thm:linear}:

\begin{lemma} \label{lem:wp-aux}
For $k \geq 0$, we have
\begin{equation} \label{eq:dv-gmm}
        |(t^{-1}\p_v)^k \gmm(t, v)| \lesssim t^{-\frac{k}{2}}\|U(t, \cdot)\|_{L^\infty}.
\end{equation}
Moreover, we also have
\begin{equation}\label{eq:relation-der-0th-w-higher-order-gmm}
\begin{aligned}
    \left|(\tfrac1{t}\p_v)^k (\gmm(t, v)e^{i\bfPhi^\ob(t, tv)}) - \gmm[\p_x^k u](t, v) e^{i\bfPhi^\ob(t, tv)}\right| 
    \aleq t\|\p_x^{(\leq k)}u(t, \cdot)\|_{L^\infty_x}.
\end{aligned}
\end{equation}
\end{lemma}

Propositions~\ref{prop:diff-u-gmmvarphi0-wo-1/tdv-involved} and \ref{prop:dotgmm_decay} and Lemma~\ref{lem:wp-aux} are proved in Section~\ref{sec:test-by-wp}.

\subsection{Main iteration and proof of Main Theorem~\ref{thm:linear}}\label{subsec-main-iter}
In this subsection, we assume that the propositions in Section~\ref{subsec-good-comm-wp} hold and establish the main theorem. The proof relies on an iteration argument. To verify the base case, we record an energy estimate for perturbations with complex coefficients.

\begin{lemma}[An energy estimate with derivative loss]\label{lem:eng-est-w-loss}
    Assume that \ref{hyp:af} holds true for $\calH$. For any $u$ satisfying \eqref{eq:sch-gen}, we have
    \begin{itemize}[leftmargin=*]
        \item when $\calK = \0$, for $0 \leq N \leq M_c + 1$, \[
            \|\p_{x}^{(\leq N)}u(t, \cdot)\|_{L^2} \lesssim \|\p_{x}^{(\leq N)}u_0\|_{L^2} + \|\brk{x}^{-1-2\eta} \p_{x}^{(\leq N)} u\|_{L^1L^2[0, t]}^{\frac12}\|\p_x^{(\leq N+1)} u\|_{L^\infty L^2[0, t]}^{\frac12} + \|\p_{x}^{(\leq N)} f\|_{L^1L^2[0, t]},
        \]
        \item when $\calK \neq \0$, for $0 \leq N \leq M_c + 1$ with $N \in 2\bbZ_{\geq 0}$, \begin{align*}
        \|\p_x^{(\leq N)}u(t, \cdot)\|_{L^2} \aleq \|\p_x^{(\leq N)} u_0\|_{L^2} &+ \|r^{-1-2\eta} \p_{x}^{(\leq N)} u\|_{L^1L^2[0, t]}^{\frac12}\| \p_x^{(\leq N+1)} u\|_{L^\infty L^2[0, t]}^{\frac12} \\
        &+ \sum_{2i + j \leq N} \|\p_t^{(\leq i)}\p_x^{(\leq j)} f\|_{L^1L^2 \cap L^\infty L^2[0, t]},
        \end{align*}
        
        \item when $\calK \neq \0$, for $0 \leq N \leq M_c$ with $N \in 2\bbZ_{\geq 0} + 1$, \begin{align*}
            \|\p_x^{(\leq N)}u(t, \cdot)\|_{L^2} \aleq \|\p_x^{(\leq N+1)} u_0\|_{L^2} &+ \|r^{-1-2\eta} \p_{x}^{(\leq N+1)} u\|_{L^1L^2[0, t]}^{\frac12}\| \p_x^{(\leq N+2)} u\|_{L^\infty L^2[0, t]}^{\frac12} \\
            &+ \sum_{2i + j \leq N+1} \|\p_t^{(\leq i)}\p_x^{(\leq j)} f\|_{L^1L^2 \cap L^\infty L^2[0, t]}.
        \end{align*}
    \end{itemize}
\end{lemma}
\begin{proof}
    Expanding the equation as \[
        (i\p_t + \Delta)u = D_\mu \bfh^{\mu\nu} D_\nu u + b^\mu D_\mu u + D_\mu (b^\mu u) + c u + f.
    \]
    Since $\bfh^{\mu\nu}$ is real, we directly multiply it by $\bar u$, integrate by parts and take the imaginary part. The result for $N = 0$ follows readily. Compared with Proposition~\ref{prop:E-implied-by-stationarity}, the derivative loss occurs precisely when $b^{\mu}$ is not real, where integration by parts is no longer available to shift derivatives onto $b^\mu$. For $N > 0$ and $\calK = \0$, one simply commutes $\p_{x}^k$ into the equation and it follows from an induction on $k = N$.

    For the case $\calK \neq \0$, we apply Lemma~\ref{lem:unweight-ell} and \ref{hyp:af} with $N \in 2\bbZ_{\geq 0}$ iteratively and \begin{align*}
        &\quad \|\p_x^{(\leq N)}u(t \cdot)\|_{L^2_x} \aleq \|\p_t^{(\leq \frac{N}{2})} u(t, \cdot)\|_{L^2_x} + \sum_{2i + j\leq N} \|\p_t^{(\leq i)}\p_x^{(\leq j)}f(t, \cdot)\|_{L^2_x} \\
        &\aleq \|\p_x^{(\leq N)} u_0\|_{L^2} + \|r^{-1-2\eta} \p_{t}^{(\leq \frac{N}{2})} u\|_{L^1L^2[0, t]}^{\frac12}\|\p_x^{(\leq 1)} \p_t^{(\leq \frac{N}{2})} u\|_{L^\infty L^2[0, t]}^{\frac12} +  \sum_{2i + j\leq N} \|\p_t^{(\leq i)}\p_x^{(\leq j)}f\|_{L^1L^2 \cap L^\infty L^2[0, t]}\\
        &\aleq \|\p_x^{(\leq N)} u_0\|_{L^2} + \|r^{-1-2\eta} \p_{x}^{(\leq N)} u\|_{L^1L^2[0, t]}^{\frac12}\| \p_x^{(\leq N+1)} u\|_{L^\infty L^2[0, t]}^{\frac12} + \sum_{2i + j \leq N} \|\p_t^{(\leq i)}\p_x^{(\leq j)} f\|_{L^1L^2 \cap L^\infty L^2[0, t]}.
    \end{align*}
    To obtain the result for odd numbers of derivatives, one interpolates and loses one more derivative.
\end{proof}

Note that some results in Section~\ref{sec:good-comm} are only applicable when $t > T_0$. We introduce the following lemma, which is applicable for any $t > 0$, without referring to any of the results in Section~\ref{sec:good-comm}. This lemma allows for growth in time and serves as the base case for our iteration.

\begin{lemma}[Base case for the iteration]\label{lem:base_iteration}
    Assume that $\calH$ satisfies \ref{hyp:af}. Given a solution $u$ to \eqref{eq:sch-gen} with \eqref{eq:E-u}, \eqref{eq:D-u} and \ref{hyp:f} satisfying \[
    \begin{cases}
        M_c \geq M_0, \hbox{ when } \calK = \0 \\
        M_c \geq M_0 + 2, \hbox{ when } \calK \neq \0
    \end{cases},
    \] we have \[
        \|\p_{x}^{(\leq M_0)} u(t, \cdot)\|_{L^\infty} \lesssim A_0, \quad \|\p_{x}^{(\leq M_0)} \calL u(t, \cdot)\|_{L^2} \lesssim (D_0 + A_0 + D_f)\brk{t}^2, \hbox{  for all } t > 0.
    \]
\end{lemma}
\begin{proof}
    The first estimate follows from a simple Sobolev embedding \[
        \|\p_{x}^{(\leq N)}u\|_{L^\infty} \lesssim \|\p_{x}^{(\leq N+2)} u\|_{L^2} \lesssim A_0, \hbox{ for all } 0 \leq N \leq M_0.
    \]

    For simplicity, we prove the second estimate in the case $\calK = \0$, where the extra loss in the obstacle case is due to that in Lemma~\ref{lem:eng-est-w-loss}. We proceed as follows according to the structure of $\calL$.
    Since $[i\p_t + \Delta, x] = 2\nabla$, we commute $x$ into the equation and obtain \begin{align*}
       (i\p_t - \calH)x u = 2\nabla u +[x, D_\mu \bfh^{\mu\nu} D_\nu + b^\mu D_\mu + D_\mu b^\mu + c]u + x f.
    \end{align*}
    Since \[
    \|\p_{x}^{(\leq M_0+1)}[x, D_\mu \bfh^{\mu\nu} D_\nu + b^\mu D_\mu + D_\mu b^\mu + c]u\|_{L^1L^2[0,t]} \aleq t\|\p_{x}^{(\leq M_0+2)} u\|_{L^\infty L^2[0,t]} \aleq A_0 t,
    \]
    it follows from Lemma~\ref{lem:eng-est-w-loss} and \eqref{eq:E-u} that \begin{align*}
        \|\p_{x}^{(\leq M_0+1)}(xu)(t, \cdot)\|_{L^2} \aleq D_0 + A_0 \brk{t} + D_f.
    \end{align*}
    In addition, commuting $x$ once more, applying Lemma~\ref{lem:eng-est-w-loss} to \begin{align*}
       (i\p_t - \calH)x^2 u = 2\nabla(xu) + 2x\nabla u + [x^2, D_\mu \bfh^{\mu\nu} D_\nu + b^\mu D_\mu + D_\mu b^\mu + c]u + x^2 f,
    \end{align*}
    it follows that \begin{align*}
        \|\p_{x}^{(\leq M_0)}(x^2u)(t, \cdot)\|_{L^2} \aleq D_0 + A_0\brk{t} + D_f + (D_0 + A_0 t + D_f)t \aleq (D_0 + A_0 + D_f)\brk{t}^2.
    \end{align*}
    In particular, the second estimate in the lemma follows readily.
\end{proof}

\begin{lemma}[Main iteration]\label{lem:main_iteration}
    Under the assumptions of Main Theorem~\ref{thm:linear}, there exists $T_0$ such that for any $t > T_0$, if we fix an arbitrary $\alpha > 0$, and  additionally assume that there exists $0 \leq N \leq M_0 - s_c^\prime$ such that $u$ satisfies the following bounds: \begin{equation}\label{eq:iter_assump}
        \|\p_{x}^{(\leq N + s_c^\prime)}\calL u\|_{L^2} \leq A t^{\frac14 - \eta + \alpha}, \quad \|\p_{x}^{(\leq N + s_c^\prime)} u\|_{L^\infty} \leq A t^{-\frac32 + \alpha},
    \end{equation}
    where we have $\eta < \frac14$ according to our convention in Section~\ref{subsubsec-convention}.
    Then an improved bound with derivative loss can be obtained: \[
    \|\p_{x}^{(\leq N)}\calL u\|_{L^2} \leq A' t^{\frac14 - \eta + \alpha'}, \quad \|\p_{x}^{(\leq N + s_c^\prime)} u\|_{L^\infty} \leq A' t^{-\frac32 + \alpha'},
    \]
    where $\alpha':= \max\{\alpha - \frac{\eta}{4}, 0\}$. Here, $A' \aeq D_0 + A + A_0 + D_f$ with implicit constant depending on $A_c$.
\end{lemma}
\begin{proof}
    Fix a sufficiently large $T_0$ so that the results in Section~\ref{subsec-good-comm-wp} hold.
    It follows from Proposition~\ref{prop:diff-u-gmmvarphi0-wo-1/tdv-involved} and \ref{prop:dotgmm_decay} that \begin{align*}
       \|t^{\frac32}\p_x^{(\leq N + s_c^\prime)} u(t, tv)\|_{L^\infty_v} &\lesssim \|\gmm[\p_x^{(\leq N + s_c^\prime)} u](t, v)\|_{L^\infty_v} + (A + A_0) t^{\alpha - \frac{\eta}{4}} \\
       &\lesssim (A + A_0 + D_f)\max\{t^{\alpha - \frac{\eta}{4}}, 1\} + \|\gmm[\p_x^{(\leq N + s_c^{\prime})} u](T_0, v)\|_{L^\infty_v} \lesssim A' t^{\alpha'}.
    \end{align*}

    We then apply Proposition~\ref{prop:est-calLu} and obtain \begin{align*}
        \|\p_x^{(\leq N)} \calL u(t, \cdot)\|_{L^2} \aleq \|\p_x^{(\leq N+s_c^{\prime})} \calL u(T_0, \cdot)\|_{L^2} &+ At^{\max\{\frac14-2\eta+\alpha, 0\}} + At^{\max\{\alpha - \eta, 0\}}\\
        &+ At^{\max\{\frac14-2\eta +\alpha, 0\}} + A_0 T_0^{-\frac12-2\eta} + D_f t^{\frac14 - \eta}.
    \end{align*}
    Moreover, the terms at time $T_0$ could be just bounded by the base case provided that $T_0$ is just a fixed number.
\end{proof}

We are now ready to prove Main Theorem~\ref{thm:linear}.

\begin{proof}[Proof of  Main Theorem~\ref{thm:linear}]
    {\it Step 1: Sharp pointwise decay estimate.}
    Thanks to Lemma~\ref{lem:base_iteration}, the additional assumptions in Lemma~\ref{lem:main_iteration} hold with $\alpha = 2$ and $N = M_0 - s_c^\prime$.
    Running Lemma~\ref{lem:main_iteration}
   once would reach \[
        \alpha \leadsto \alpha':= \max\{\alpha - \frac{\eta}{4}, 0\}, \quad N \leadsto N':= N - s_c^\prime.
    \]
    Therefore, after running the iteration for $c_0:= \lceil\frac{8}{\eta}\rceil$ many times, one could reach the sharp pointwise decay for $\p_x^{(\leq N)} u$ with $0 \leq N \leq M_0 - c_0 s_c^\prime$.

    In particular, the first step, together with \eqref{eq:E-u}, \ref{hyp:f}, Propositions~\ref{prop:diff-u-gmmvarphi0-wo-1/tdv-involved} and \ref{prop:dotgmm_decay}, implies the following bounds for $0 \leq N \leq M_0 - c_0s_c'$, $t > T_0$: \begin{equation}\label{eq:sharp-est-in-proof-mainthm}
    \begin{aligned}
        \left\|t^{\frac32}(\p_x^\alpha u)(t, tv) - \sum_{ \bt_1 + \bt_2 = \alp, N := |\alp|}  \binom{\alp}{\bt_1} \gmm[\p_x^{\beta_1} u](t, v) e^{i\bfPhi^\ob(t, tv)} (\p_x^{\beta_2} \varphi_0)(t, tv) \right\|_{L^\infty_v} &\aleq D t^{-\frac{\eta}{4}}, \\
        \|\tfrac{d}{dt}\gmm[\p^N u](t, v)\|_{L^\infty_v} &\aleq (D + D_f) t^{-1-\min\{\frac{\eta}{4}, \delta_f\}}.
    \end{aligned}
    \end{equation}

    {\it Step 2: Leading order asymptotics for $N = 0$. }
    Taking $N = 0$ in \eqref{eq:sharp-est-in-proof-mainthm}, we know that \[
        |\gmm(t, v)| \lesssim |\gmm(T_0, v)| + \int_{T_0}^t |\dot\gmm(s,v)|\, \ud s \lesssim D + D_f,
    \]
    where $|\gmm(T_0, v)|$ satisfies the trivial bound by \eqref{eq:D-u} and the estimate is independent of $t$ and $v$. This in turn tells us that there exists $\gmm_\infty(v) \in L^\infty_v$ such that \[
        \gmm_\infty(v):= \lim_{t \to \infty}\gmm(t, v).
    \]
    Then it follows that \[
        |U(t, vt) - \gmm_\infty(v) \varphi_0(tv)| \lesssim |U(t, vt) - \gmm(t, v) \varphi_0(tv)| + |\gmm_\infty(v) - \gmm(t, v)||\varphi_0(tv)| \lesssim (D + D_f) t^{-\min\{\frac{\eta}{4}, \delta_f\}}.
    \]
    This completes the proof for the case $N = 0$.

    {\it Step 3: Leading order asymptotics for $0 < N \leq M_0 - c_0 s_c'$. }
    For any multi-index $\beta$, running the same analysis as in the $N = 0$ case, we know that there exists $\gmm_\infty[\p_x^\beta u](v) \in L^\infty_v$ such that $\gmm_\infty[\p_x^\beta u](v):= \lim_{t \to \infty} \gmm[\p_x^\beta u](t, v)$.
    Combining with \eqref{eq:relation-der-0th-w-higher-order-gmm}, we know that \begin{equation}
        \label{eq:main-proof-1}
        e^{-i\bfPhi^\ob(t, tv)}(\tfrac1{t}\p_v)^\alpha\big(\gmm(t, v) e^{i\bfPhi^\ob(t, tv)}\big) \to \gmm_\infty[\p_x^\alpha u](v).
    \end{equation}

    In addition, we fix an arbitrary $v$, and consider $(\tfrac1{t}\p_v)^\alpha\big(\gmm(t, v) e^{i\bfPhi^\ob(t, tv)}\big)$.
    We compute \begin{equation}\label{eq:main-proof-2}
        \begin{aligned}
        (\tfrac1{t}\p_v)^\alpha\big(\gmm(t, v) e^{i\bfPhi^\ob(t, tv)}\big) &= \gmm(t,v) (\tfrac1{t}\p_v)^\alpha\big(e^{i\bfPhi^\ob(t, tv)}\big) + \underbrace{\sum_{ \bt_1 + \bt_2 = \alp, |\beta_1| \geq 1} \big((\tfrac1{t}\p_v)^{\beta_1} \gmm (t, v) \big) \big((\tfrac1{t}\p_v)^{\beta_2} e^{i\bfPhi^\ob(t, tv)} \big)}_{\to 0, \hbox{ as } t \to \infty, \hbox{ due to \eqref{eq:dv-gmm}}}\\
        &= \big(\tfrac{iX(tv)}{2t}\big)^\alpha \gmm(t, v) e^{i\bfPhi^\ob(t, tv)}  + o(1),
    \end{aligned}
    \end{equation}
    where in the last step, we use the following claim: for any fixed $v$, \[
        (\tfrac1{t}\p_v)^{\beta_3}\big(\big(\tfrac{iX(tv)}{2t}\big)^{\beta_4}\big) \to 0, \hbox{ when } \beta_3 \neq 0.
    \]
    This can be easily shown by induction.
    Combining \eqref{eq:main-proof-1} with \eqref{eq:main-proof-2}, it yields that \[
        \gmm_\infty[\p_x^\alpha u](v) = \lim_{t \to \infty} \big(\tfrac{iX(tv)}{2t}\big)^\alpha \gmm(t, v) = \big(\tfrac{iv}{2}\big)^\alpha \gmm_\infty(v),
    \]
    where we used the shorthand $X_\mu(x):= \chi^\ob(x)x_\mu + \frac12(\p_\mu \chi^\ob) |x|^2$.
    As a side remark, the fact that this belongs to $L^\infty_v$ allows us to conclude decay properties of $\gmm_\infty$. Indeed, from this convergence, we deduce that $\brk{v}^N\gmm_\infty(v) \in L^\infty$.

    The proof concludes by plugging this into \eqref{eq:sharp-est-in-proof-mainthm}.
\end{proof}

\section{The good commutator identity}\label{sec:good-comm}
Our good commutator $\calL$ was defined in Definition~\ref{def:good-comm}. The goal of this section is to prove the following key commutation identity:
\begin{proposition}[Good commutator identity] \label{prop:good-comm}
We have
\begin{equation*}
	[\calL, i \rd_{t} - \calH] = \frac{1}{t} a_{\mathrm{gc}}^{\mu \nu} L_{\mu} L_{\nu} + b_{\mathrm{gc}}^{\mu} L_{\mu} + t c_{\mathrm{gc}},
\end{equation*}
where the coefficients satisfy the decay estimates
\begin{equation*}
    \begin{cases}
    a_{\mathrm{gc}}^{\mu \nu} = O^{M_c}_{x}(A_{c} \brk{r}^{-2 \eta}), \quad
	b_{\mathrm{gc}}^{\mu} = O^{M_c}_{x}(A_{c} \brk{r}^{-1-2 \eta}), \quad
	c_{\mathrm{gc}} = O^{M_c}_{x}(A_{c} \brk{r}^{-2-2 \eta}), \hbox{ when } \calH \hbox{ satisfies \ref{hyp:af-0};} \\
    a_{\mathrm{gc}}^{\mu \nu} = O^{M_c}_{t,x}(A_{c} \brk{r}^{-2 \eta}), \quad
	b_{\mathrm{gc}}^{\mu} = O^{M_c}_{t,x}(A_{c} \brk{r}^{-1-2 \eta}), \quad
	c_{\mathrm{gc}} = O^{M_c}_{t,x}(A_{c} \brk{r}^{-2-2 \eta}), \hbox{ when } \calH \hbox{ satisfies \ref{hyp:af}.}
    \end{cases}
\end{equation*}

\end{proposition}

\subsection{Conjugation by a modified phase function}
We start with the following conjugation identity for the operator $i \rd_{t} - \calH$, which is a straightforward computation:
\begin{lemma} \label{lem:phase-conj}
Given $\bfPhi: \bbR^{1+d} \to \bbC$, we have
\begin{equation} \label{eq:phase-conj}
\begin{aligned}
	&t^{\frac{d}{2}} e^{- i \bfPhi} (i \rd_{t} - \calH) t^{-\frac{d}{2}}e^{i \bfPhi} \\
	&= i \left( \rd_{t} - \tfrac{d}{2} t^{-1} + i \bfPhi_{t} \right) - (D_{\mu} + \bfPhi_{\mu}) a^{\mu \nu} (D_{\nu} + \bfPhi_{\nu}) - b^{\mu} (D_{\mu} + \bfPhi_{\mu}) - (D_{\mu} + \bfPhi_{\mu}) b^{\mu} - c \\
	&= i (\rd_{t} + 2 \dlt^{\mu \nu} \bfPhi_{\mu} \rd_{\nu}) - \calH - (\bfPhi_{t} +  \dlt^{\mu \nu} \bfPhi_{\mu}\bfPhi_{\nu}) + i \left( \dlt^{\mu \nu} \rd_{\mu} \bfPhi_{\nu}- \tfrac{d}{2} t^{-1} \right) \\
	&\peq -  \bfh^{\mu \nu} \bfPhi_{\nu} D_{\mu} -  D_{\mu} \bfh^{\mu \nu} \bfPhi_{\nu} 
	- \bfh^{\mu \nu} \bfPhi_{\mu} \bfPhi_{\nu} 
	- 2 b^{\mu} \bfPhi_{\mu},
\end{aligned}
\end{equation}
where we used the shorthands
\begin{equation*}
	\bfPhi_{t}:= \rd_{t} \bfPhi, \quad \bfPhi_{\mu}:= \rd_{\mu} \bfPhi.
\end{equation*}
\end{lemma}

The first identity is immediate, and the second identity follows by expanding and regrouping all the terms.

From this identity, we see the distinguished property of the \emph{quadratic phase} $\mathring{\bfPhi}$ given in \eqref{eq:std-phase}.
Indeed, with the choice \eqref{eq:std-phase}, we have
\begin{equation} \label{eq:std-phase-d}
	\mathring{\bfPhi}_{t} = - \frac{\abs{x}^{2}}{4 t^{2}}, \quad \mathring{\bfPhi}_{\mu} = \frac{x_{\mu}}{2 t},
\end{equation}
from which it follows that the following terms in \eqref{eq:phase-conj}, which are present even if $\calH = -\lap$, vanish:
\begin{equation*}
	\mathring{\bfPhi}_{t} + \dlt^{\mu\nu} \mathring{\bfPhi}_{\mu} \mathring{\bfPhi}_{\nu} = 0, \quad
	\dlt^{\mu \nu} \rd_{\mu} \mathring{\bfPhi}_{\nu} - \tfrac{d}{2} t^{-1} = 0.
\end{equation*}

To treat variable-coefficient and/or obstacle problems, using a suitably cutoff version of $\bfPhi = \bfPhi^\ob$ turns out to be useful. With this choice, we have
\begin{equation} \label{eq:ob-phase-d}
	\bfPhi^\ob_{t} = - \chi^\ob \frac{\abs{x}^{2}}{4 t^{2}}, \quad \bfPhi^\ob_{\mu} = \chi^\ob \frac{x_\mu}{2t} + (\p_\mu \chi^\ob) \frac{|x|^2}{4t}.
\end{equation}

To facilitate the computation in this section and Section~\ref{sec-proof_prop:dotgmm_decay}, we also record termwise conjugation identities
\begin{proposition}[Basic conjugation properties]\label{prop:conj_calH}
    In the obstacle-free case,
    \begin{align*}
	e^{- i \mathring{\bfPhi}} \circ D_\mu a^{\mu \nu} D_{\nu} \circ e^{i \mathring{\bfPhi}}
	&= D_\mu a^{\mu\nu} D_\nu + a^{\mu\nu} \frac{x_{\mu} x_{\nu}}{4 t^{2}} + a^{\mu \nu} \frac{x_{\mu}}{t} D_\nu - i \frac{1}{2 t} a^{\mu \nu} \dlt_{\mu \nu} + (D_\mu a^{\mu\nu}) \frac{x_\nu}{2t}, \\
	e^{- i \mathring{\bfPhi}} \circ (b^{\mu} D_{\mu} + D_\mu b^\mu) \circ e^{i \mathring{\bfPhi}}
	&= b^{\mu} D_{\mu} + D_\mu b^\mu + b^{\mu} \frac{ x_{\mu}}{t}, \\
	e^{- i \mathring{\bfPhi}} \circ c \circ  e^{i \mathring{\bfPhi}}
	&= c, \\
    t^{\frac{3}{2}} e^{- i \mathring{\bfPhi}} \circ \rd_{t} \circ t^{-\frac{3}{2}} e^{i \mathring{\bfPhi}}
    &= - i \frac{\abs{x}^{2}}{4 t^{2}} - \frac{3}{2} \frac{1}{t}+ \rd_{t}.
\end{align*}
    With the choice $\bfPhi = \bfPhi^\ob$, we record \begin{align*}
        b^{\nu}D_\nu \circ e^{i\bfPhi^\ob}
        &= e^{i\bfPhi^\ob} b^{\nu} \left(\chi^\ob \frac{x_\nu}{2t} + (\p_\nu \chi^\ob) \frac{|x|^2}{4t}\right) - i e^{i\bfPhi^\ob} b^{\nu}\p_\nu, \\
        e^{-i\bfPhi^\ob} \circ D_\mu a^{\mu\nu}D_\nu \circ e^{i\bfPhi^\ob} &= D_\mu a^{\mu\nu} D_\nu + a^{\mu\nu}\left(\chi^\ob \frac{x_\mu}{2t} + (\p_\mu \chi^\ob) \frac{|x|^2}{4t}\right)\left(\chi^\ob \frac{x_\nu}{2t} + (\p_\nu \chi^\ob) \frac{|x|^2}{4t}\right) \\
        &\quad + \left(\chi^\ob \frac{x_\nu}{2t} + (\p_\nu \chi^\ob) \frac{|x|^2}{4t}\right)(
        (D_\mu a^{\mu\nu}) + 2 a^{\mu\nu}D_\mu) - \frac{i}{2t} \chi^\ob a^{\mu\nu}\delta_{\mu\nu} \\
        &\quad - i a^{\mu\nu} (\p_\mu\chi^\ob) \frac{x_\nu}{t} - i a^{\mu\nu}(\p_{\mu\nu}^2 \chi^\ob) \frac{|x|^2}{4t}.
    \end{align*}
\end{proposition}
We omit the proof, which is a straightforward computation.

\subsection{Proof of Proposition~\ref{prop:good-comm}}\label{subsec-good-comm}
Now we are ready to prove Proposition~\ref{prop:good-comm}.
\begin{proof}[Proof of Proposition~\ref{prop:good-comm}]

{\it Step 1: The obstacle-free case with \ref{hyp:af-0}.}
    We start with the obstacle-free case. 
    Substituting \eqref{eq:std-phase-d} into Lemma~\ref{lem:phase-conj}, we have
    \[
    \begin{aligned}
        &\peq t^{\frac{3}{2}} e^{- i \mathring{\bfPhi}} \circ (i \rd_{t} - \calH) \circ t^{-\frac{3}{2}} e^{i \mathring{\bfPhi}} \\
        &= it^{-1}(t\p_t + x \cdot \nabla_x) - \calH - \bfh^{\mu\nu}\frac{x_\mu x_\nu}{4t^2} + i \bfh^{\mu\nu} \frac{x_\mu}{t}\p_{x^\nu} + i \frac1{2t}\bfh^{\mu\nu}\delta_{\mu\nu} - (D_\mu a^{\mu\nu} + 2b^\nu) \frac{x_\nu}{2t},
    \end{aligned}
    \]
    where $a^{\mu\nu} = \delta^{\mu\nu} + \bfh^{\mu\nu}$. Together with \eqref{eq:conj_calLu_calHU}, we arrive at \begin{equation}\label{eq:comm_id_in_proof}
        \begin{aligned}
    & \peq t^{\frac{3}{2}} e^{- i \mathring{\bfPhi}} \circ [\calL,  i \rd_{t} - \calH]
	\circ t^{-\frac{3}{2}} e^{i \mathring{\bfPhi}} \\
	&=  i t^{-1} [4 t^{2} \calH, \calS] - [\calH, \bfh^{\mu \nu} x_{\mu} x_{\nu}] + 4 i t [\calH, \bfh^{\mu \nu} x_{\mu} \rd_{x^{\nu}}] + 2 i t [\calH, \bfh^{\mu \nu} \dlt_{\mu \nu}] - 2t [\calH, (D_{\mu} a^{\mu\nu} + 2b^{\nu}) x_{\nu}] \\
	&=  4 i t ([\calH, \calS] - 2 \calH) - [\calH, \bfh^{\mu \nu} x_{\mu} x_{\nu}] + 4 i t [\calH, \bfh^{\mu \nu} x_{\mu} \rd_{x^{\nu}}] + 2 i t [\calH, \bfh^{\mu \nu} \dlt_{\mu \nu}] - 2t [\calH, (D_{\mu} a^{\mu\nu} + 2b^{\nu}) x_{\nu}],
        \end{aligned}
    \end{equation}
where $\calS:= t\p_t + x \cdot \nabla_x$. Note that the second order contribution solely comes from $[\calH, \calS] - 2\calH + [\calH, \bfh^{\mu\nu}x_\mu \p_{x^\nu}]$. Thanks to the identity $[\Delta, \calS] = 2\Delta$, the decay property of $a^{\mu\nu}_{\mathrm{gc}}$ is justified. Furthermore, one can use the three standard assumptions in \ref{hyp:af-0} to handle all the terms in \eqref{eq:comm_id_in_proof} except $[\calH, \bfh^{\mu \nu} x_{\mu} x_{\nu}]$. On the other hand, we treat this term by taking \eqref{eq:af-extra1-0} into account and hence $\nb^{j} (\bfh^{\mu \nu} x_{\mu} x_{\nu}) = O^{M_c}(t \brk{r}^{-j-2 \eta})$ for $j = 0, 1, 2$. This finishes the proof of the case $\calK = \0$.

{\it Step 2: The obstacle case with \ref{hyp:af-0}.}
Now we deal with the case $\calK \neq \0$. Since introducing $\p^{\leq 2} \chi^\ob$ would not change the spatial decay, in view of the compatible time decay of the coefficients in the conjugation identities above, the decay estimates for the coefficients under higher order spatial derivatives remain the same as in the previous case.

{\it Step 3: Under the hypothesis \ref{hyp:af}. }
Under the additional assumptions in \ref{hyp:af} compared to \ref{hyp:af-0}, the corresponding estimates for temporal derivatives follow immediately.
\end{proof}

\section{Two-scale elliptic analysis} \label{sec:two-scale}
We now adapt the standard elliptic estimates in Section~\ref{subsec:ell-calH} to study the good commutator $\calL$.
Recall that $\calL$ is time-dependent; in this section, unless otherwise stated, {\bf $\calL u$ is evaluated at time $t$}. (We remark that all constants in this section are independent of $t$)

\subsection{Zeroth-order statements}
We first use Proposition~\ref{prop:w-ell-0} to derive an estimate that is adapted to spatial scales that are fixed in $t$. We dub it the \emph{inner-scale} elliptic estimate:

\begin{lemma} [Inner-scale elliptic estimate] \label{lem:elliptic-near}
Assume that $\calH$ satisfies \ref{hyp:af-0} and \ref{hyp:se}. Then for every $R \geq 2 \max\{1, \sup_{y \in \p\calK} |y|\}$, for every $\mu, \nu = 1, 2, 3$, we have
\begin{equation*}
	t^{-2} \nrm{r^{\frac{1}{2}} L_\mu L_\nu u}_{\ell^{\infty} L^{2}(r < R)} + t^{-1} \nrm{r^{-\frac{1}{2}} L_\mu u}_{\ell^{\infty} L^{2}(r < R)}
	+ \nrm{r^{-\frac{3}{2}} u}_{\ell^{\infty} L^{2}(r < R)} \aleq t^{-2} \nrm{r^{\frac{1}{2}} \calL u}_{\ell^{1} L^{2}(r < 2R)} + \nrm{u}_{L^{\infty}(\frac12 R < r < 2 R)},
\end{equation*}
where the constant is independent of $R$. In addition, for any $0 < \veps < 1$, we have \[
    t^{-2} \nrm{r^{\frac{1}{2}+\veps} L_\mu L_\nu u}_{L^{2}(r < R)} + t^{-1} \nrm{r^{-\frac{1}{2}+\veps} L_\mu u}_{L^{2}(r < R)}
	+ \nrm{r^{-\frac{3}{2}+\veps} u}_{L^{2}(r < R)} \aleq t^{-2} \nrm{r^{\frac{1}{2} + \veps} \calL u}_{L^{2}(r < 2R)} + R^{\veps}\nrm{u}_{L^{\infty}(\frac12 R < r < 2 R)}.
\]
\end{lemma}
\begin{proof}
    We only give a detailed proof for the first estimate; the second estimate may be proved similarly.
    By switching to the notation $U$, it suffices to show that \begin{equation}\label{eq:unit-scale-ell-U}
        \nrm{U}_{\ell^{\infty} H^{2,-\frac32}(r < R)}
	 \aleq \nrm{r^{\frac{1}{2}} \calH U}_{\ell^{1} L^{2}(r < 2R)} + \nrm{U}_{L^{\infty}(\frac12 R < r < 2 R)}.
    \end{equation}
    From here on, our proof is a standard smooth cutoff argument.
    We introduce a cut-off function $\chi_R$ such that $\chi_R(r) \equiv 1$ in $\{r < R\}$ and $\supp \chi_R \subset \{r < \frac32R\}$. Applying Proposition~\ref{prop:w-ell-0} to $\chi_R U$, it yields that \begin{align*}
        \nrm{U}_{\ell^{\infty} H^{2,-\frac32}(r < R)}
	 &\aleq \nrm{r^{\frac{1}{2}} \calH U}_{\ell^{1} L^{2}(r < \frac32 R)} + R^{\frac12}\nrm{R^{-2}\chi_R''U + R^{-1}\chi_R'\p U}_{\ell^{1} L^{2}} \\
     &\lesssim \nrm{r^{\frac{1}{2}} \calH U}_{\ell^{1} L^{2}(r < \frac32 R)} + R^{-\frac32}\nrm{U}_{\ell^\infty L^{2}(R < r < \frac32 R)} + R^{-\frac12}\nrm{\nabla U}_{\ell^\infty L^{2}(R < r < \frac32 R)}.
    \end{align*}
    Applying Proposition~\ref{prop:dyadic-ell-0}, we obtain
    \[
        \|\nabla U\|_{L^2(R < r < \frac32 R)} \lesssim R\|\calH U\|_{L^2(\frac12 R < r < 2R)} + R^{-1}\|U\|_{L^2(\frac12 R < r < 2R)}.
    \]
    Thus, combining with the previous estimate, we obtain \[
        \nrm{U}_{\ell^{\infty} H^{2,-\frac32}(r < R)}
	 \lesssim \nrm{r^{\frac{1}{2}} \calH U}_{\ell^{1} L^{2}(r < 2R)} + R^{-\frac32}\nrm{U}_{ L^{2}(\frac12 R < r < 2R)}.
    \]
    Finally, estimating the last term in terms of $L^\infty$ on bounded region $\{r \aeq R\}$ concludes the proof.
\end{proof}

Next, we state an estimate that is adapted to larger spatial scales, which we refer to as the \emph{outer-scale} elliptic estimate. 

\begin{lemma}[Outer-scale elliptic estimate] \label{lem:elliptic-rad}
Assume that $\calH$ satisfies \ref{hyp:af-0}.
There exists some constant $R_0$ such that for every $R \geq R_0$, for every $\mu, \nu = 1, 2, 3$, we have
\begin{align*}
    t^{-2}\nrm{L_\mu L_\nu u}_{L^{2}(r > R)} + t^{-1}\|r^{-1}L_\mu u\|_{L^2(r > R)} + \|r^{-2}u\|_{L^2(r > R)}
	&\aleq t^{-2}\nrm{\calL u}_{L^{2}(r > \frac12 R)} + R^{-\frac12}\nrm{u}_{L^\infty(r > \frac12 R)}, 
\end{align*}
where the constant is independent of $R > R_0$.
Additionally, we have \begin{equation} \label{eq:ell-r>t-U}
    \|L_{\mu} L_{\nu} u\|_{L^2(r > R)} + \|L_\mu u\|_{L^2(r > R)} + \|u\|_{L^2(r > R)} \lesssim \|\calL u\|_{L^2(r > \frac12 R)} + (t^2R^{-2} + 1)\|u\|_{L^2(r > \frac12 R)}.
\end{equation}
\end{lemma}
\begin{proof}
    We start with the first estimate. Switching to the notation $U$ as before, it suffices to show
    \begin{equation} \label{eq:ell-rad-U-L2}
        \|\nb^{2} U\|_{L^2(r > R)} + \|r^{-1} \nabla U\|_{L^2(r > R)} + \|r^{-2}U\|_{L^2(r > R)} \lesssim \|\calH U\|_{L^2(r > a^{-1} R)} + \| r^{-2}U\|_{L^2(r > a^{-1} R)},
    \end{equation}
    for an arbitrary $a > 1$. 
    Indeed, from the preceding estimate with $a = 2$, the first estimate in Lemma~\ref{lem:elliptic-rad} will follow from H\"older's inequality,
    \[
    \|r^{-2}U\|_{L^2(r > \frac{1}{2} R)} \lesssim R^{-\frac12}\|U\|_{L^\infty(r > \frac{1}{2} R)}.
    \]
    By Proposition~\ref{prop:dyadic-ell-0}, for every $R' \geq R$, we have
    \begin{align*}
    &\|\nb^{2} U\|_{L^2(R' < r < 2 R')} + \|r^{-1} \nabla U\|_{L^2(R' < r < 2 R')} + \|r^{-2}U\|_{L^2(R' < r < 2 R')} \\
    \lesssim& \|\calH U\|_{L^2(a^{-1} R' < r < 2a R')} + \|r^{-2}U\|_{L^2(a^{-1} R' < r < 2a R')}.    
    \end{align*}
    Square summing this estimate for all $R' \in 2^{\bbZ_{\geq 0}} \cdot R$, the desired estimate for $U$ follows. 

    To prove the second inequality, observe that the control of $L_{\mu} L_{\nu} u$ follows from \eqref{eq:ell-rad-U-L2}, and that of $u$ is obvious. To control $L_\mu u$ in $L^{2}(r > R)$, after passing to the variable $U$, it is sufficient to establish the interpolation bound 
    \begin{equation*}
        \nrm{\nb U}_{L^{2}(r > R)} \aleq (\nrm{\nb^{2} U}_{L^{2}(r > R)} + R^{-2} \nrm{U}_{L^{2}(r > R)})^{\frac{1}{2}} \nrm{U}_{L^{2}(r > R)}^\frac{1}{2},
    \end{equation*}
    which follows from rescaling (to set $R = 1$), Sobolev extension from $H^{2}(r > 1)$ to $H^{2}(\bbR^{3})$ (similarly for $L^{2}$), and the standard interpolation on the whole space. \qedhere
\end{proof}

We record some useful corollaries of the preceding proof.
\begin{corollary} \label{cor:out-scale-ell-lap}
    Let $U_{\ageq R}:= \chi_{>R}(r) U$, where  $\chi_{>R}$ is a smooth cut-off function satisfying $\chi_{>R}(r) \equiv 1$ in $\{r > R\}$ and $\supp \chi_{>R} \subset \{r > \frac34 R\}$. We have
    \begin{equation}\label{eq:ell-rad-stronger-U-cut-off}
        \|\Delta U_{\gtrsim R}\|_{L^2} \lesssim \|\calH U\|_{L^2(r > \frac12 R)} + R^{-\frac12} \|U\|_{L^\infty(r > \frac12 R)}.
    \end{equation}
\end{corollary}
\begin{proof}
    We simply estimate the left-hand side using \eqref{eq:ell-rad-U-L2} with an appropriate choice of $R$ and $a$.
\end{proof}

\begin{corollary}\label{cor:out-scale-ell-loc}
    Under the same assumptions as in Lemma~\ref{lem:elliptic-rad}, for $R_2 > R_1 > R_0$, we have 
    \begin{align*}
    &\quad t^{-2}\nrm{L_\mu L_\nu u}_{L^{2}(R_1 < r < R_2)} + t^{-1}\|r^{-1}L_\mu u\|_{L^2(R_1 < r < R_2)} + \|r^{-2}u\|_{L^2(R_1 < r < R_2)} \\
	&\aleq t^{-2}\nrm{\calL u}_{L^{2}(R_1 \lesssim r \lesssim R_2)} + R_1^{-\frac12}\nrm{u}_{L^\infty(R_1 \lesssim r \lesssim R_2)}. 
\end{align*}
\end{corollary}
\begin{proof}
    In the proof of the first estimate in Lemma~\ref{lem:elliptic-rad}, we sum over finitely many dyadic annuli that cover $\set{R_{1} < r < R_{2}}$. We omit the details.
\end{proof}

To end this part, we include an unweighted elliptic estimate that will become useful in the obstacle case (see for instance Lemma~\ref{lem:eng-est-w-loss}) and nonlinear applications (see Section~\ref{sec-nonlinear}). 
\begin{lemma}\label{lem:unweight-ell}
    Given $u \in L^\infty H^k(\calM)$ with $u|_{[0, \infty) \times \p\calK} = 0$ and $\p\calK \in C^k$, we have \begin{align*}
        \|\p_x^{(\leq k)} u\|_{L^2} \aleq \|\p_x^{(\leq k - 2)} \calH u\|_{L^2} + \|u\|_{L^2}.
    \end{align*}
\end{lemma}
\begin{proof}
    We choose $R_0$ sufficiently large so that $\bfh^{\mu\nu}$ is perturbative. Writing $\Delta = -\calH + D_\mu \bfh^{\mu\nu} D_\nu + b^\mu D_\mu + D_\mu b^\mu + c$, the desired estimate for $u$ on $\set{r \ageq R_{0}}$ can be proved using the same argument as in the proof of Lemma~\ref{lem:elliptic-rad}.
    To deal with the inner region $\{r \aleq R_0\}$, one directly invokes the standard elliptic regularity theory up to the boundary. Combining the two regions concludes the proof.
\end{proof}

\subsubsection{Outer-scale Morrey's inequality}

We recall the standard Morrey's inequality on $\R^3$.
\begin{lemma}[Homogeneous Morrey's inequality on $\R^3$]
    For any function $g \in \scrS(\R^3)$, we have \[
        \frac{|g(x) - g(y)|}{|x - y|^{\frac12}} \lesssim \|\Delta g\|_{L^2}.
    \]
\end{lemma}
\begin{proof}
    One could find a proof in \cite{Evans} or \cite[Lemma 3.4]{Clo20}.
\end{proof}

\begin{lemma}[Outer-scale Morrey's inequality]\label{lem:Morrey-out-scale}
    There exists $R_0$ such that for all $|x_1|, |x_2| \geq R > R_0 $, we have \[
       t^{-\frac32}\frac{|U(t, x_1) - U(t, x_2)|}{|x_1 - x_2|^{\frac12}} \lesssim t^{-2}\|\calL u(t, \cdot)\|_{L^2(r > \frac12 R)} + R^{-\frac12} \|u\|_{L^\infty(r > \frac12 R)},
    \]
    where the relation between $u$ and $U$ is dictated by \eqref{eq:relation-uU}.
\end{lemma}
\begin{proof}
    Rewriting everything in the $U$-notation, it suffices to show \[
        \frac{|U(t, x_1) - U(t, x_2)|}{|x_1 - x_2|^{\frac12}} \lesssim \|\Delta U\|_{L^2(r > \frac12 R)} + R^{-\frac12} \|U\|_{L^\infty(r > \frac12 R)}.
    \]
    Using the same cut-off function $\chi_{\gtrsim R}$ and $U_{\gtrsim R}$ as in Corollary~\ref{cor:out-scale-ell-lap}, we write \[
       \frac{|U(t, x_1) - U(t, x_2)|}{|x_1 - x_2|^{\frac12}} = \frac{|U_{\gtrsim R}(t, x_1) - U_{\gtrsim R}(t, x_2)|}{|x_1 - x_2|^{\frac12}} \lesssim \|\Delta U_{\gtrsim R}\|_{L^2}
    \]
    and hence the result follows from \eqref{eq:ell-rad-stronger-U-cut-off}.
\end{proof}

\subsection{Higher-order statements}
Our next goal is to extend the inner- and outer-scale elliptic estimates (Lemmas~\ref{lem:elliptic-near} and \ref{lem:elliptic-rad}, respectively) to higher-order derivatives. We will also formulate the results in the way they are directly applied; we thus begin with a discussion of terminology.

\begin{remark}[Near, intermediate, and far-away zones] \label{rem:zones}
    It is useful to divide $\calM^{1+3}$ into three zones: 
    \begin{itemize}[leftmargin=*]
        \item the \emph{near zone} $\set{r \aleq t^\frac{1}{2}}$, where it is optimal to use the inner-scale elliptic estimate (assuming $\nrm{\calL u}_{L^{2}} \aleq t^{\frac{1}{4}-\eta}$ and $\nrm{u}_{L^{\infty}} \aleq t^{-\frac{3}{2}}$, which we retrieve via iteration); 
        \item the \emph{far-away zone} $\set{r \ageq t}$, where the outer-scale elliptic estimate allows us to simply control the $L^{2}$-norm of $L^{(2)} u$ with that of $\calL u$ and the standard energy $\nrm{u}_{L^2}$ (see \eqref{eq:ell-r>t-U}); and
        \item the \emph{intermediate zone} $\set{t^\frac{1}{2} \aleq r \aleq t}$, which sits in-between.
    \end{itemize}
\end{remark}

\begin{proposition}[Higher order near- \& intermediate-zone elliptic estimates with \ref{hyp:af-0}]\label{prop:elliptic-near}
    Assume that $\calH$ satisfies \ref{hyp:af-0}. Then for every $\delta \in [0, \frac12]$, we have \[\begin{aligned}
        &t^{-2} \nrm{r^{\frac{1}{2}} \p_{x}^{(\leq N)} L^{(2)} u}_{\ell^{\infty} L^{2}(r < t^{\frac12 + \delta})} + t^{-1} \nrm{r^{-\frac{1}{2}} \p_{x}^{(\leq N)} L^{(1)} u}_{\ell^{\infty} L^{2}(r < t^{\frac12 + \delta})} + \nrm{r^{-\frac{3}{2}} \p_{x}^{(\leq N)} u}_{\ell^{\infty} L^{2}(r < t^{\frac12 + \delta})} \\ \lesssim&\,  t^{-2}\|r^{\frac12}\p_x^{(\leq N)} \calL u\|_{\ell^1 L^2(r \lesssim t^{\frac12 + \delta})} + \| u\|_{L^\infty(r \aleq t^{\frac12 + \dlt})}, \hbox{ for all } 0 \leq N \leq M_c.
    \end{aligned}
    \]
\end{proposition}

\begin{corollary}[Higher order near- \& intermediate-zone elliptic estimates with \ref{hyp:af}]\label{cor:elliptic-near}
    Assume that $\calH$ satisfies \ref{hyp:af}. Then for every $\delta \in [0, \frac12]$, $2i + j \leq N$, for every solution $u$ to \eqref{eq:sch-gen}, we have \begin{align*}
        & t^{-2} \nrm{r^{\frac{1}{2}} \p_t^{(\leq i)} \p_{x}^{(\leq j)} L^{(2)} u}_{\ell^{\infty} L^{2}(r < t^{\frac12 + \delta})} + t^{-1} \nrm{r^{-\frac{1}{2}} \p_t^{(\leq i)} \p_{x}^{(\leq j)} L^{(1)} u}_{\ell^{\infty} L^{2}(r < t^{\frac12 + \delta})} + \nrm{r^{-\frac{3}{2}} \p_t^{(\leq i)} \p_{x}^{(\leq j)} u}_{\ell^{\infty} L^{2}(r < t^{\frac12 + \delta})} \\ \lesssim&\,  t^{-2} \|r^{\frac12}\p_x^{(\leq N)} \calL u\|_{\ell^1 L^2(r \lesssim t^{\frac12 + \delta})} + \|\p_t^{(\leq i)} u\|_{L^\infty(r \aleq t^{\frac12 + \dlt})} + t^{-2}\|r^{\frac12}\p_t^{(\leq i-1)}\p_x^{(\leq j)} \calL f\|_{\ell^1 L^2(r \aleq t^{\frac12+\dlt})}, \hbox{ for all } 0 \leq N \leq M_c.
    \end{align*}
\end{corollary}

\begin{proposition}[Higher order intermediate-zone elliptic estimates with \ref{hyp:af-0}]\label{prop:elliptic-int}
    Assume that $\calH$ satisfies \ref{hyp:af-0}. Then for every $\delta \in [0, \frac12]$, there exists $T_0$ such that for all $t > T_0$, we have \[\begin{aligned}
        &t^{-2} \nrm{\p_{x}^{(\leq N)} L^{(2)} u}_{L^{2}(t^{\frac12 + \delta} < r < t)} + t^{-1} \nrm{r^{-1}\p_{x}^{(\leq N)} L^{(1)} u}_{ L^{2}(t^{\frac12 + \delta} < r < t)} + \nrm{r^{-2}\p_{x}^{(\leq N)} u}_{ L^{2}(t^{\frac12 + \delta} < r < t)} \\  \lesssim&\,  t^{-2}\|\p_x^{(\leq N)} \calL u\|_{L^2( t^{\frac12 + \delta} \aleq r \aleq t)} + t^{-\frac14-\frac{\delta}{2}}\|\p_x^{(\leq N)} u\|_{L^\infty(t^{\frac12 + \delta} \aleq r \aleq t)}, \hbox{ for all } 0 \leq N \leq M_c.
    \end{aligned}
    \]
\end{proposition}

\begin{corollary}
    [Higher order intermediate-zone elliptic estimates with \ref{hyp:af}]\label{cor:elliptic-int}
    Assume that $\calH$ satisfies \ref{hyp:af}. Then for every $\delta \in [0, \frac12]$, $0 \leq 2i + j \leq N \leq M_c$,  there exists $T_0$ such that for all $t > T_0$, for every solution $u$ to \eqref{eq:sch-gen}, we have \[\begin{aligned}
        &t^{-2} \nrm{\p_t^{(\leq i)}\p_{x}^{(\leq j)} L^{(2)} u}_{L^{2}(t^{\frac12 + \delta} < r < t)} + t^{-1} \nrm{r^{-1}\p_t^{(\leq i)}\p_{x}^{(\leq j)} L^{(1)} u}_{ L^{2}(t^{\frac12 + \delta} < r < t)} + \nrm{r^{-2}\p_t^{(\leq i)}\p_{x}^{(\leq j)} u}_{ L^{2}(t^{\frac12 + \delta} < r < t)} \\  \lesssim&\,  t^{-2}\|\p_x^{(\leq N)} \calL u\|_{L^2(t^{\frac12 + \delta} \aleq r \aleq t)} + t^{-\frac14-\frac{\delta}{2}}\|\p_t^{(\leq i)}\p_x^{(\leq j)} u\|_{L^\infty(t^{\frac12 + \delta} \aleq r \aleq t)} + t^{-2} \|\p_t^{(\leq i-1)} \p_x^{(\leq j)} \calL f\|_{L^2(t^{\frac12+\dlt} \aleq r \aleq t)}.
    \end{aligned}
    \]
\end{corollary}

\begin{proposition}[Higher order far-away-zone elliptic estimates with \ref{hyp:af-0}]\label{prop:elliptic-rad}
    Assume that $\calH$ satisfies \ref{hyp:af-0}. Then for every $\delta \in [0, \frac12]$, there exists $T_0$ such that for all $t > T_0$, we have \[\begin{aligned}
        &t^{-2} \nrm{\p_{x}^{(\leq N)} L^{(2)} u}_{L^{2}(r > t^{\frac12 + \delta})} + t^{-2} \nrm{\p_{x}^{(\leq N)} L^{(1)} u}_{ L^{2}(r > t^{\frac12 + \delta})} \\  \lesssim&\,  t^{-2}\|\p_x^{(\leq N)} \calL u\|_{L^2(r \gtrsim t^{\frac12 + \delta})} + t^{-1-2\delta}\|\p_x^{(\leq N)} u\|_{L^2(r \gtrsim t^{\frac12 + \delta})}, \hbox{ for all } 0 \leq N \leq M_c.
    \end{aligned}
    \]
\end{proposition}

\begin{corollary}[Higher order far-away-zone elliptic estimates with \ref{hyp:af}]\label{cor:elliptic-rad}
    Assume that $\calH$ satisfies \ref{hyp:af}. Then for all $i, j, N$ such that $0 \leq 2i + j \leq N \leq M_c$, there exists $T_0$ such that for all $t > T_0$, for every solution $u$ to \eqref{eq:sch-gen}, we have \[\begin{aligned}
        &t^{-2} \nrm{\p_t^{(\leq i)}\p_{x}^{(\leq j)} L^{(2)} u}_{L^{2}(r > t)} + t^{-2} \nrm{\p_t^{(\leq i)}\p_{x}^{(\leq j)} L^{(1)} u}_{ L^{2}(r > t)} + t^{-2} \nrm{\p_t^{(\leq i)}\p_{x}^{(\leq j)} u}_{ L^{2}(r > t)} \\  \lesssim&\,  t^{-2}\|\p_x^{(\leq N)} \calL u\|_{L^2(r \gtrsim t)} + t^{-2}\|\p_x^{(\leq N)} u\|_{L^2(r \gtrsim t)} + t^{-2}\|\p_t^{(\leq i-1)}\p_x^{(\leq j)} \calL f\|_{L^2(r > t)}.
    \end{aligned}
    \]
\end{corollary}

\begin{proposition}[Higher order intermediate- \& far-away-zone elliptic estimates]\label{prop:elliptic-int-rad}
    Assume that $\calH$ satisfies \ref{hyp:af-0}. Then for every $\delta \in [0, \frac12]$, there exists $T_0$ such that for all $t > T_0$, we have \[\begin{aligned}
        &t^{-2} \nrm{L^{(2)} \p_{x}^{(\leq N)} u}_{L^{2}(r > t^{\frac12 + \delta})} + t^{-1} \nrm{r^{-1} L^{(1)} \p_{x}^{(\leq N)} u}_{ L^{2}(r > t^{\frac12 + \delta})} + \|r^{-2}\p_x^{(\leq N)}u\|_{L^2(r > t^{\frac12 + \delta})}  \\  \lesssim&\,  t^{-2}\|\p_x^{(\leq N)} \calL u\|_{L^2(r \gtrsim t^{\frac12 + \delta})} + t^{-\frac14-\frac{\delta}{2}}\|\p_x^{(\leq N)} u\|_{L^\infty} + t^{-2}\|\p_x^{(\leq N)} u\|_{L^2}, \hbox{ for all } 0 \leq N \leq M_c.
    \end{aligned}
    \]
\end{proposition}

\begin{remark} \label{rem:ell-est-higher-order}
Some comments concerning the above set of results are in order.
\begin{enumerate}[leftmargin=*]
  \item Propositions~\ref{prop:elliptic-near},~\ref{prop:elliptic-int},~\ref{prop:elliptic-rad},~\ref{prop:elliptic-int-rad} and their corollaries are higher order generalizations of Lemma~\ref{lem:elliptic-near} and \ref{lem:elliptic-rad}; their proofs are given in Section~\ref{subsec-two-scale-higher-order}. We choose to state them in a form with $\dlt \in [0, \frac12]$, which is convenient for future applications. 
    \item The ordering of coordinate derivatives relative to the operator $L$ is formulated differently in Propositions and Corollaries \ref{prop:elliptic-near}–\ref{cor:elliptic-rad}, compared to Proposition~\ref{prop:elliptic-int-rad}. The former formulation is adapted to the form of \ref{hyp:iled*}, which will be used to establish Proposition~\ref{prop:est-calLu}. On the other hand, the form of Proposition~\ref{prop:elliptic-int-rad} aligns with Morrey's inequality, and is therefore well suited for establishing Propositions~\ref{prop:diff-u-gmmvarphi0-wo-1/tdv-involved} and \ref{prop:dotgmm_decay}.
    \item We state Propositions and Corollaries~\ref{prop:elliptic-near}–\ref{prop:elliptic-int-rad} without assuming \ref{hyp:se}, in contrast to their corresponding zeroth-order estimates. For Propositions~\ref{prop:elliptic-int}–\ref{prop:elliptic-int-rad} and their corresponding zeroth order version, the assumption \ref{hyp:se} (or \ref{hyp:se-0}) is in fact automatically satisfied in the region of interest when $T_0$ is sufficiently large, since the operator there is effectively a small perturbation of the free Laplacian. On the other hand, the $N = 0$ case in the near zone estimate ( Proposition~\ref{prop:elliptic-near}) is strictly weaker than Lemma~\ref{lem:elliptic-near}, i.e., it is necessary to assume \ref{hyp:se-0} to ensure a localized error as in  Lemma~\ref{lem:elliptic-near}. On the other hand, the non-localized error phrased in Proposition~\ref{prop:elliptic-near} is chosen to depend only on $u$ itself rather than higher order derivatives of $u$, which is the main advantage here. The localization could, however, be recovered under the assumption \ref{hyp:se}, namely \begin{align*}
            &t^{-2} \nrm{r^{\frac{1}{2}} \p_{x}^{(\leq N)} L^{(2)} u}_{\ell^{\infty} L^{2}(r < t^{\frac12 + \delta})} + t^{-1} \nrm{r^{-\frac{1}{2}} \p_{x}^{(\leq N)} L^{(1)} u}_{\ell^{\infty} L^{2}(r < t^{\frac12 + \delta})} + \nrm{r^{-\frac{3}{2}} \p_{x}^{(\leq N)} u}_{\ell^{\infty} L^{2}(r < t^{\frac12 + \delta})} \\ \lesssim&\,  t^{-2}\|r^{\frac12}\p_x^{(\leq N)} \calL u\|_{\ell^1 L^2(r \lesssim t^{\frac12 + \delta})} + t^{-\frac32(\frac12 + \dlt)}\| \p_x^{(\leq N)}u\|_{L^2(r \aeq t^{\frac12+\dlt})}, \hbox{ for all } 0 \leq N \leq M_c.
        \end{align*}
        This proposition, either in the form of a localized $L^2$-error of $\p_x^{(\leq N)} u$ or a non-localized $L^\infty$-error of $u$, will be crucial in the proof of Corollary~\ref{cor:near-zone-asymp-higher-order}.
\end{enumerate}
\end{remark}

\subsection{Higher order two-scale elliptic estimates}\label{subsec-two-scale-higher-order}

\subsubsection{Proof of Proposition~\ref{prop:elliptic-near}}
Now we make an induction on $N \leq M_c$. 
Note that $k = 0$ case does not follow immediately from our previous zeroth order estimate, which takes \ref{hyp:se-0} as an input and has a localized error term on the right; see Remark~\ref{rem:ell-est-higher-order}. Here, the case $k = 0$ corresponds to the standard interior regularity, which can be simply obtained by summing up Proposition~\ref{prop:dyadic-ell-0}.

Assume that it holds for $N-1$ and we prove it for $1 \leq k = N$. It suffices to prove for the first and the third term since the second one follows from an interpolation. 

{\it Step 1: The first term.}
We claim that it suffices to show \begin{align}\label{eq:claim-proof-prop-ell-near}
    \|r^{\frac12} \p_x^{(\leq k)}\p_x^{(2)}U\|_{\ell^\infty L^2(r < R)} \aleq \|r^{\frac12} \p_x^{(\leq k)} \calH U\|_{\ell^1 L^2(r \aleq R)} + \|U\|_{L^\infty(r \aleq R)}.
\end{align}
Indeed, by assuming this, we select $R \aeq t^{\frac12 + \delta}$ and compute
\begin{align*}
    &t^{-2} \nrm{r^{\frac{1}{2}} \p_{x}^k L^{(2)} u}_{\ell^{\infty} L^{2}(r < R)}
    \lesssim t^{-2-k}\|r^{\frac12} L^{(k+2)}u\|_{\ell^\infty L^2(r < R)} + t^{-2}\|r^{\frac12}\p_x^{(\leq k-1)} L^{(2)}u\|_{\ell^\infty L^2(r < R)} \\
    \lesssim&\, (t^{-2-k}\|r^{\frac12} L^{(k)}\calL u\|_{\ell^1 L^2(r \lesssim R)} + \|u\|_{L^\infty(r \aleq R)}) + (t^{-2}\|r^{\frac12}\p_x^{(\leq k-1)}\calL u\|_{\ell^1 L^2(r \lesssim R)} + \|u\|_{L^\infty(r \aleq R)}) \\
    \lesssim&\, t^{-2}\|r^{\frac12}\p_x^{(\leq k)} \calL u\|_{\ell^1 L^2(r \lesssim R)} + \| u\|_{L^\infty(r \aleq R)},
\end{align*}
where the definition of $L$ is used in the first and last inequalities. Moreover, in the second inequality, we apply the claim \eqref{eq:claim-proof-prop-ell-near} to the first term after conjugation \eqref{eq:relation-uU} and we apply the induction hypothesis to the second term, respectively. This concludes the proof in Step 1.

Now we prove the claim \eqref{eq:claim-proof-prop-ell-near}. This is standard as Proposition~\ref{prop:dyadic-ell-0}, where one obtains from unweighted estimates (\cite[Theorem 6.3.2 on higher interior regularity]{Evans}) on each dyadic piece $\lmb \aleq R$: \begin{align*}
    \|r^{\frac12} \p_x^{(\leq k)}\p_x^{(2)}U\|_{L^2(\lmb < r < 2\lmb)} &\aleq \|r^{\frac12} \p_x^{(\leq k)} \calH U\|_{L^2(r \aeq \lmb) } + \lmb^{-\frac32}\|U\|_{L^2(r \aeq \lmb)} \\
   &\aleq \|r^{\frac12} \p_x^{(\leq k)} \calH U\|_{L^2(r \aeq \lmb) } + \|U\|_{L^\infty(r \aeq \lmb)}.
\end{align*}
More precisely, this handles the component $\|r^{\frac12} \p_x^{(k+2)} ((1 - \chi) U)\|_{\ell^\infty L^2}$ away from the boundary (when $\calK \neq \0$), where $\chi$ is an appropriate cutoff such that $\chi \equiv 1$ near the obstacle and $\supp \p\chi \subset \set{x: 0 < \dist(x, \p\calK) < 1,  |x| < 2R}$. The contribution near the boundary $\|r^{\frac12} \p_x^{(k+2)} (\chi U)\|_{\ell^\infty L^2(r < R)}$ is handled by the unweighted boundary regularity theorem \cite{Evans}.
Hence, from $\ell^1 \subset \ell^\infty$, we have \[
    \|r^{\frac12} \p_x^{(\leq k)} \p_x^{(2)} U\|_{\ell^\infty L^2(r < R)} \aleq \|r^{\frac12} \p_x^{(\leq k)} \calH U\|_{\ell^1 L^2 (r \aleq R)} + \|U\|_{L^\infty(r \aleq R)}.
\]

{\it Step 2: The third term. }
For the third term in the main proposition, we write \[
    \nrm{r^{-\frac{3}{2}} \p_{x}^k u}_{\ell^{\infty} L^{2}(r < R)}
    \lesssim t^{-1}\|r^{-\frac32}\p_{x}^{k-1}L^{(1)}u\|_{\ell^\infty L^2(r < R)} + \|r^{-\frac32}\p_x^{(\leq k-1)} u\|_{\ell^\infty L^2(r < R)},
\]
where we apply the induction hypothesis on both terms to conclude the proof.

\subsubsection{Proof of Proposition~\ref{prop:elliptic-int}}
The case $N = 0$ is established in Lemma~\ref{lem:elliptic-rad}. We assume that it holds for $N - 1$ and prove it for $1 \leq k = N$.
We prove by converting $\p_x$ to $L$ as in the proof of Proposition~\ref{prop:elliptic-near} before. To keep track of the growth rate generated by $x$ in $L$, we need to employ Corollary~\ref{cor:out-scale-ell-loc}.
We write
\begin{equation}\label{eq:proof-prop-ell-int}
    \begin{split}
        &t^{-2} \|\p_x^k L^{(2)} u\|_{L^2(t^{\frac12 + \delta} \leq r \leq t)} \lesssim t^{-2-k} \|L^{(k+2)} u\|_{L^2(t^{\frac12 + \delta} \leq r \leq t)} + t^{-2} \|\p_x^{(\leq k-1)}L^{(2)} u\|_{L^2(t^{\frac12 + \delta} \leq r \leq t)} \\
        \lesssim&\, t^{-2-k}\|\calL L^{(k)}u\|_{L^2(t^{\frac12 + \delta} \aleq r \aleq t)} + t^{-2} \|\p_x^{(\leq k-1)}\calL u\|_{L^2(t^{\frac12 + \delta} \aleq r \aleq t)} + t^{-\frac14-\frac{\delta}{2}}\|\p_x^{(\leq k)} u\|_{L^\infty(t^{\frac12 + \delta} \aleq r \aleq t)}.
    \end{split}
\end{equation}
Since $[L^{(1)}, L^{(1)}] = 0$, an explicit computation reveals that \[\begin{aligned}
    [\calL, L^{(k)}] &= L_\mu [a^{\mu\nu}, L^{(k)}]L_\nu - 2t([b^\mu, L^{(k)}]L_\mu + L_\mu [b^\mu, L^{(k)}]) + 4t^2 [c, L^{(k)}] \\
    &= (2it)^k \left(L_\mu [\bfh^{\mu\nu}, \p_x^{(k)}]L_\nu - 2t([b^\mu, \p_x^{(k)}]L_\mu + L_\mu [b^\mu, \p_x^{(k)}]) + 4t^2 [c, \p_x^{(k)}] \right) + \cdots,
\end{aligned}
\]
where we isolate the term which sees the highest order derivative of coefficients. To handle the commutator term, it is worth noting that the spatial decay needs to be exploited first so that $\ell^1$ can be replaced by $\ell^\infty$ before applying the induction hypothesis.
Applying \ref{hyp:af-0},
we obtain that \[\begin{aligned}
    t^{-2-k}\| [\calL, L^{(k)}] u\|_{\ell^1 L^2(t^{\frac12 + \delta} \aleq r \aleq t)} &\lesssim_{O^{M_c - k}(1)} t^{-2}\|  \p_x^{(\leq k-1)} L^{(2)} u\|_{\ell^\infty L^2(t^{\frac12 + \delta} \aleq r \aleq t)} \\
    &+ t^{-1}\|r^{-1} \p_x^{(\leq k-1)} L^{(1)} u\|_{\ell^\infty L^2(t^{\frac12 + \delta} \aleq r \aleq t)}
    + \|r^{-2} \p_x^{(\leq k-1)} u\|_{\ell^\infty L^2(t^{\frac12 + \dlt} \aleq r \aleq t)}.
\end{aligned}
\]
Applying the induction hypothesis and combining it with \eqref{eq:proof-prop-ell-int}, we conclude the proof for the first term. One can see from this estimate why we require $N \leq M_c$  and we do not keep track of it in the similar proofs below.

\subsubsection{Proof of Proposition~\ref{prop:elliptic-rad}}
The case $N = 0$ is established in Lemma~\ref{lem:elliptic-rad}. We assume that it holds for $N - 1$ and prove it for $1 \leq k = N$.
Thanks to \[
    [\p_\mu, L_\nu] = \p_\mu X_\nu = O^\infty(1), \quad [\p_x, L^{(2)}] = O^\infty(1) L^{(1)},
\]
we compute \[
    t^{-2} \|\p_x^k L^{(2)} u\|_{L^2(r > R)} \lesssim t^{-2} \|\p_x^{k-1} L^{(2)} \p_x u\|_{L^2(r > R)} + t^{-2}\|\p_x^{k-1} L^{(1)} u\|_{L^2(r > R)}.
\]
For $R \aeq t^{\frac12 + \delta}$, applying the induction hypothesis, we obtain \[\begin{aligned}
    t^{-2} \|\p_x^k L^{(2)} u\|_{L^2(r > R)} &\lesssim t^{-2} \|\p_x^{(\leq k-1)} \calL \p_x u\|_{L^2(r \gtrsim R)} + t^{-2}\|\p_x^{(\leq k-1)} \calL u\|_{L^2(r \gtrsim R)} + t^{-1-2\delta} \|\p_x^{(\leq k)} u\|_{L^2(r \gtrsim R)} \\
    &\lesssim t^{-2} \|\p_x^{(\leq k)} \calL u\|_{L^2(r \gtrsim R)} + t^{-2}\|\p_x^{(\leq k-1)} [\calL,\p_x] u\|_{L^2(r \gtrsim R)} + t^{-1-2\delta} \|\p_x^{(\leq k)} u\|_{L^2(r \gtrsim R)},
\end{aligned}
\]
where the commutator term can be estimated using the induction hypothesis as well. On the other hand, the term $t^{-2} \|\p_x^k L^{(1)} u\|_{L^2(r > R)}$ can be estimated in a similar fashion, and we omit the details.

\subsubsection{Proof of Corollaries~\ref{cor:elliptic-near},~\ref{cor:elliptic-int} and \ref{cor:elliptic-rad}}

We group the proofs of these three corollaries together. The proof of Corollary~\ref{cor:elliptic-int} proceeds along the same lines as that of Corollary~\ref{cor:elliptic-near}, where both results rely on the condition ${r \aleq t}$. Therefore, we only present the proof of Corollary~\ref{cor:elliptic-near} and \ref{cor:elliptic-rad} in detail in the following.

\begin{proof}[Proof of Corollary~\ref{cor:elliptic-near}]
    We first establish an auxiliary estimate with both temporal and spatial derivatives present on the right-hand side: \begin{align*}
        & t^{-2} \nrm{r^{\frac{1}{2}} \p_t^{(\leq i)} \p_{x}^{(\leq j)} L^{(2)} u}_{\ell^{\infty} L^{2}(r < t^{\frac12 + \delta})} + t^{-1} \nrm{r^{-\frac{1}{2}} \p_t^{(\leq i)} \p_{x}^{(\leq j)} L^{(1)} u}_{\ell^{\infty} L^{2}(r < t^{\frac12 + \delta})} + \nrm{r^{-\frac{3}{2}} \p_t^{(\leq i)} \p_{x}^{(\leq j)} u}_{\ell^{\infty} L^{2}(r < t^{\frac12 + \delta})} \\ \lesssim&\,  t^{-2} \|r^{\frac12}\p_t^{(\leq i)} \p_x^{(\leq j)} \calL u\|_{\ell^1 L^2(r \lesssim t^{\frac12 + \delta})} + \|\p_t^{(\leq i)} u\|_{L^\infty}, \hbox{ for all } 0 \leq 2i + j \leq N \leq M_c
    \end{align*}
    We prove by an induction on $i$. First, the case $i = 0$ was established in Proposition~\ref{prop:elliptic-near}. We assume that it holds true for $i = k-1$.

    Since $[\p_t^k, L^{(1)}] = 2ik\p_t^{k-1}\p_x$, \[
    [\p_t^k, L^{(2)}] = 2ik\p_t^{k-1}\p_x L^{(1)} + 2ik L^{(1)} \p_t^{k-1}\p_x
    = 4ik\p_t^{k-1}\p_x L^{(1)} + 4k(k-1) \p_t^{k-2}\p_x^2 - 2ik\p_t^{k-1}
    \]
    and the region of interest on the left satisfies $r \leq t$, one could easily see that the commutator bound follows from the induction hypothesis. It then suffices to establish \[\begin{aligned}
            & t^{-2} \nrm{r^{\frac{1}{2}} \p_{x}^{(\leq j)} L^{(2)} \p_t^{k} u}_{\ell^{\infty} L^{2}(r < t^{\frac12 + \delta})} + t^{-1} \nrm{r^{-\frac{1}{2}} \p_{x}^{(\leq j)} L^{(1)} \p_t^{ k}u}_{\ell^{\infty} L^{2}(r < t^{\frac12 + \delta})} + \nrm{r^{-\frac{3}{2}} \p_{x}^{(\leq j)} \p_t^{k}u }_{\ell^{\infty} L^{2}(r < t^{\frac12 + \delta})} \\ \lesssim&\,  t^{-2}\|r^{\frac12}\p_t^{(\leq k)}\p_x^{(\leq j)} \calL u\|_{\ell^1 L^2(r \lesssim t^{\frac12 + \delta})} + \| \p_t^{(\leq k)} u\|_{L^\infty}, \quad\forall 0 \leq j \leq N - 2k.
        \end{aligned}
    \]
    By Proposition~\ref{prop:elliptic-near}, the left-hand side can be controlled by \[
        t^{-2}\|r^{\frac12}\p_x^{(\leq j)} \calL \p_t^{k} u\|_{\ell^1 L^2(r \lesssim t^{\frac12 + \delta})} + \| \p_t^{k} u\|_{L^\infty},
    \]
    where the first term can be in turn bounded by \[
        t^{-2}\|r^{\frac12}\p_x^{(\leq j)}\p_t^{(\leq k)} \calL u\|_{\ell^1 L^2(r \lesssim t^{\frac12 + \delta})} + \| \p_t^{(\leq k)} u\|_{L^\infty}
    \]
    plus error terms with control via induction hypothesis. Note that no $\log t$-loss appears due to $\|r/t\|_{\ell^1(r \aleq t^{\frac12 + \dlt})} \aleq 1$. This finishes the proof of the auxiliary result.

    In view of this, it suffices to estimate $\|r^{\frac12} \p_t^{(\leq i)} \p_x^{(\leq j)} \calL u\|_{\ell^1 L^2(r \aleq t^{\frac12 + \dlt})}$.

    We make an induction on $i$ and assume that the result holds for $i \leq k-1$, $j \leq N - 2i$. The result holds for $i = 0$ thanks to Proposition~\ref{prop:elliptic-near}.

    Now we prove for $i = k$, $j \leq N - 2k$. We write \begin{align*}
        i\p_t^i \p_x^j\calL u &= \p_t^{i-1}\p_x^j\calH \calL u + \p_t^{i-1}\p_x^j [i\p_t - \calH, \calL]u + \p_t^{i-1}\p_x^j \calL f,
    \end{align*}
    where the first term features $2(i-1) + (j+2) \leq N$ and hence can be handled by the induction hypothesis. On the other hand, the commutator term can be handled using Proposition~\ref{prop:good-comm} together with the induction hypothesis.  
    Hence, the proof follows.
\end{proof}

\begin{proof}[Proof of Corollary~\ref{cor:elliptic-rad}]
    The strategy is roughly the same, where one first establishes an auxiliary result by commuting temporal derivatives with $L$'s. Then one formulates a suitable induction and it suffices to bound the commutator term by Proposition~\ref{prop:good-comm} \begin{align*}
        t^{-2}\|\p_t^{i-1}\p_x^j [i\p_t - \calH, \calL]u\|_{L^2(r > t)} \aleq t^{-3}\|r^{-2\eta}\p_t^{i-1}\p_x^j L^{(2)}u\|_{L^2(r > t)} &+ t^{-2}\|r^{-1-2\eta}\p_t^{i-1}\p_x^j L^{(1)}u\|_{L^2(r > t)} \\&+ t^{-1}\|r^{-2-2\eta}\p_t^{i-1}\p_x^j u\|_{L^2(r > t)}.
    \end{align*}
    Due to the localization $\{r > t\}$, the right-hand side can be bounded by the induction hypothesis.
\end{proof}

\subsubsection{Proof of Proposition~\ref{prop:elliptic-int-rad}}
First, this holds for $N = 0$ by Lemma~\ref{lem:elliptic-rad}. We make an induction as before. Applying the induction hypothesis for $k = N - 1$, the left-hand side of the estimate for $k = N$ can be bounded by \[\begin{aligned}
    &\lesssim t^{-2}\|\p_x^{(\leq k-1)}\calL\p_x u\|_{L^2(r \ageq t^{\frac12 + \delta})} + t^{-\frac14 - \frac{\delta}{2}}\|\p_x^{(\leq k)} u\|_{L^\infty} + t^{-2}\|\p_x^{(\leq k)} u\|_{L^2} \\
    &\lesssim t^{-2}\|\p_x^{(\leq k)}\calL u\|_{L^2(r \ageq t^{\frac12 + \delta})} + t^{-\frac14 - \frac{\delta}{2}}\|\p_x^{(\leq k)} u\|_{L^\infty} + t^{-2}\|\p_x^{(\leq k)} u\|_{L^2} + t^{-2}\|\p_x^{(\leq k-1)} L u\|_{L^2(r \ageq t^{\frac12 + \delta})},
\end{aligned}
\]
where we use Proposition~\ref{prop:elliptic-int} and Proposition~\ref{prop:elliptic-rad} to estimate the commutator in the last step. The last term can be estimated by decomposing the region into $\{t^{\frac12+\delta}\aleq r \aleq t\} \cup \{r \ageq t\}$ and applying these two propositions, respectively.

\section{Energy estimate for $\calL u$}\label{sec:en-est-calLu}
The goal of this section is to prove Proposition~\ref{prop:est-calLu}. It requires a stronger result than \ref{hyp:iled*}, in which the weight in the $L^2L^2$-based norm is relaxed and the energy norm is also included on the right-hand side.

{
\renewcommand\thetheorem{ILEDapp}
\begin{proposition}\label{prop:iled*-L1L2+LE*}
    Assume that $\calH$ satisfies \ref{hyp:af} and \ref{hyp:iled*} with $M_c^* = M_c$. Set $s_c' = s_c^* + \dlt_s + 1$. Suppose $u$ is a solution to \eqref{eq:sch-gen}. We have the following:
    \begin{itemize}[leftmargin=*]
        \item When $\calK = \0$, \begin{equation*}
    	   \nrm{\rd_{x}^{(\leq N)} u}_{L^{\infty}[t_0, t_0 + T] L^{2}} \aleq \|\p_x^{(\leq N+s_c')} u(t_0)\|_{L^2} + \|\p_x^{(\leq N + s_c')} f\|_{(\brk{x}^{-1}L^2L^2 + L^1L^2)[t_0, t_0 + T]}.
        \end{equation*}
        \item When $\calK \neq \0$, \begin{equation*}
    	   \nrm{\rd_{x}^{(\leq N)} u}_{L^{\infty}[t_0, t_0 + T] L^{2}} \aleq \|\p_x^{(\leq N+s_c')} u(t_0)\|_{L^2} +  \sum_{2i + j \leq N + s_c'}\|\p_t^{(\leq i)}\p_x^{(\leq j)} f\|_{(\brk{x}^{-1}L^2L^2 + L^1L^2)[t_0, t_0 + T]}.
        \end{equation*}
    \end{itemize}
    The implicit constants are uniform for $T > 1$ for any $0 \leq N \leq M_c^*$.
\end{proposition}
}
\begin{proof}
    The proof resembles that of Proposition~\ref{prop:iled2iled-star}. We only present the proof of the case that $\calK \neq \0$ and $t_0 = 0$. Define $v$ and $\eta_0$ as in the proof of Proposition~\ref{prop:iled2iled-star}. We know that $v$ satisfies the same estimate as in \eqref{eq:iled*-Hout}.

    On the other hand, for any decomposition $f = f_1 + f_2$, we modify the decomposition of $w = u - \eta_0 v = \td w_1 + \td w_2$, where \begin{align*}
        (i\p_t - \calH) \td w_1 &= (1 - \eta_0)f_1, \quad \td w_1(0) = (1 - \eta_0)u_0, \quad \td w_1|_{[0, \infty) \times \p \calK} = 0,\\
        (i\p_t - \calH) \td w_2 &= (1 - \eta_0)f_2 + \eta_0(\calH - \calH_{\mathrm{out}})v + [\calH, \eta_0] v, \quad \td w_2(0) = 0, \quad \td w_2|_{[0, \infty) \times \p \calK} = 0,
    \end{align*}
    where the existence and uniqueness of $\td w_1, \td w_2$ follow from \ref{hyp:iled*}.

    For the first equation, we write out the solution via the Duhamel principle, then it follows from \ref{hyp:iled*} that \begin{align*}
        \|\p_x^{(\leq N)} \td w_1\|_{L^\infty L^2} \aleq \|\p_x^{(\leq N + s_c^*)} u_0\|_{L^2} + \|\p_x^{(\leq N + s_c^*)} f_1\|_{L^1 L^2}.
    \end{align*}

    Regarding $\td w_2$, we apply \ref{hyp:iled*}, \eqref{eq:iled*-Hout} and the equation to convert temporal derivatives: for any $0 \leq N \leq M_c^*$, \begin{align*}
        &\quad\|\p_x^{(\leq N)} \td w_2\|_{L^\infty L^2} \aleq \|\p_x^{(\leq N + s_c^*)} u_0\|_{L^2} + \sum_{2i + j \leq N + s_c^*} \|\brk{x}^{\beta_c +1} \p_t^{(\leq i)}\p_x^{(\leq j)} \left((\calH - \calH_{\mathrm{out}}) v + [\calH, \eta_0]v + (1 - \eta_0)f_2\right)\|_{L^2 L^2} \\
        &\aleq \|\p_x^{(\leq N + s_c')} u_0\|_{L^2} + \sum_{2i + j \leq N + s_c'} \left(\|\p_t^{(\leq i)}\p_x^{(\leq j)} f\|_{\brk{x}^{-1}L^2L^2 + L^1L^2} + \|\p_t^{(\leq i)}\p_x^{(\leq j)} f_2\|_{\brk{x}^{-1}L^2L^2}\right),
    \end{align*}
    where the estimate is indifferent to $\beta_c \geq 0$ due to the compact support.

    We finish the proof by combining the estimates for $v, \td w_1$ and $\td w_2$ and taking $\inf_{f_1 + f_2 = f}$.
\end{proof}
\begin{remark}
    The proof follows from a simple Duhamel argument if one does not seek to strengthen the statement to $\bt_c = 0$.
\end{remark}

\smallskip
Now we are ready to prove Proposition~\ref{prop:est-calLu}.
Choose $T_0$ to be sufficiently large so that Propositions and Corollaries~\ref{prop:elliptic-int}~-~\ref{cor:elliptic-rad} hold for $t > T_0$.
On the time interval $[T_0, t]$, we apply Proposition~\ref{prop:iled*-L1L2+LE*} to \[
    (i\p_t - \calH) \calL u = -[\calL, i\p_t - \calH]u + \calL f.
\]
We discuss the two cases.

{\it Case 1: $\calK = \0$. }
Applying Proposition~\ref{prop:iled*-L1L2+LE*},
\[\begin{aligned}
    \|\p_x^{(\leq N)} \calL u\|_{L^\infty L^2[T_0, t]} \aleq \|\p_x^{(\leq N + s_c')} \calL u(T_0)\|_{L^2} + \|\p_x^{(\leq N + s_c^\prime)} [\calL, i\p_t - \calH] u\|_{(\brk{x}^{-1}L^2L^2+L^1L^2)[T_0, t]} \\
    + \|\p_x^{(\leq N + s_c^\prime)}\calL f\|_{(\brk{x}^{-1}L^2L^2+L^1L^2)[T_0, t]}
\end{aligned}
\]
for any $0 \leq N \leq M_c$.
Then we would use Proposition~\ref{prop:good-comm} to unpack the second term on the right and decompose its integration domain according to Remark~\ref{rem:zones}, i.e., the near zone $\{r \lesssim t^{\frac12}\}$, the intermediate zone $\{t^{\frac12} \lesssim r \lesssim t\}$, and the far-away zone, which is their complement. The estimates below work for $N+s_c^\prime \leq M_c$.
In the near zone, we apply Proposition~\ref{prop:elliptic-near} to the $L^2L^2$-based norm, and, therefore, \[\begin{aligned}
    \|\brk{x}\p_x^{(\leq N + s_c^\prime)} [\calL, i\p_t - \calH] u\|&_{L^2 [T_0, t]L^2(r \lesssim t^{\frac12})}  \lesssim \left\|t^{\frac12(\frac12 - 2\eta)} \left( t^{-1}\|r^\frac12 \p_x^{(\leq N + s_c^\prime)} \calL u\|_{L^2(r \lesssim t^{\frac12})} \right.\right.\\
    &\qquad\qquad\qquad\qquad\qquad\qquad\left.\left.+ t\|\p_x^{(\leq N + s_c^\prime)} u\|_{L^\infty(r \aeq t^{\frac12})}\right)\right\|_{L^2_t[T_0, t]} \\
    &\qquad\ \ \lesssim \|t^{-\frac12 - \eta}\p_x^{(\leq N + s_c^\prime)} \calL u\|_{L^2L^2[T_0, t]} + \|t^{-\frac14 - \eta}\|t^{\frac32}\p_x^{(\leq N + s_c^\prime)} u\|_{L^\infty_x}\|_{L^2_t[T_0, t]}.
\end{aligned}
\]
In the intermediate zone, we apply Proposition~\ref{prop:elliptic-int} to the $L^1L^2$-based norm, and, therefore, \[
    \|\p_x^{(\leq N + s_c^\prime)}[\calL, i\p_t - \calH] u\|_{L^1[T_0, t]L^2(t^{\frac12} \lesssim r \lesssim t)}
    \lesssim \|t^{-1 - \eta} \|\p_x^{(\leq N + s_c^\prime)}\calL u\|_{L^2_x}\|_{L^1[T_0, t]} + \|t^{\frac14-\eta-1}\|t^{\frac32}\p_x^{(\leq N + s_c^\prime)} u\|_{L^\infty_x}\|_{L^1_t[T_0, t]}.
\]
In the far-away zone, we apply Proposition~\ref{prop:elliptic-rad} to the $L^1L^2$-based norm. Thus, \[
    \|\p_x^{(\leq N + s_c^\prime)}[\calL, i\p_t - \calH] u\|_{L^1[T_0, t]L^2(r \gtrsim t)}
    \lesssim \|t^{-1 - 2\eta} \|\p_x^{(\leq N + s_c^\prime)}\calL u\|_{L^2_x}\|_{L^1[T_0, t]} + \|t^{-1 - 2\eta}\|\p_x^{(\leq N + s_c^\prime)} u\|_{L^2_x}\|_{L^1_t[T_0, t]}.
\]
This finishes the proof.

{\it Case 2: $\calK \neq \0$. }
Our strategy is still to apply Proposition~\ref{prop:iled*-L1L2+LE*}. In view of the extra derivatives in \ref{hyp:iled*} in this case, we need to invoke Corollaries~ \ref{cor:elliptic-near}, \ref{cor:elliptic-int}, and \ref{cor:elliptic-rad}, respectively, in the three regions above. Therefore, extra terms of the following form will be produced \begin{align*}
    \sum_{2i + j \leq N + s_c'}&\left(\|t^{-\frac12-\eta} \p_t^{(\leq i)} \p_x^{(\leq j)} \calL f\|_{L^2L^2[T_0, t]} +  \|t^{-1 - \eta} \|\p_t^{(\leq i)}\p_x^{(\leq j)}\calL f\|_{L^2_x}\|_{L^1[T_0, t]} \right.\\
    & \left.+ \|t^{\frac54 - \eta}\|\p_t^{(\leq i)}\p_x^{(\leq j)} f\|_{L^\infty_x}\|_{L^2_t[T_0, t]} + \|t^{-1 - 2\eta}\|\p_t^{(\leq i)}\p_x^{(\leq j)} f\|_{L^2_x}\|_{L^1_t[T_0, t]}\right),
\end{align*}
where we bound terms of the form $\p_t^i \p_x^j u$ (arising from the RHS of the weighted elliptic estimates) via \begin{align*}
    i\p_t^i\p_x^j u &= \p_t^{i-1}\p_x^j\calH u + \p_t^{i-1}\p_x^j f.
\end{align*}
This completes the proof of Proposition~\ref{prop:est-calLu}. \hfill \qedsymbol

\begin{remark} \label{rem:est-calLu}
In the special case that no loss of derivatives is assumed, i.e., $s_c^\prime = 0$, one could circumvent the iteration process by absorbing the terms involving $\|\p_x^{(\leq N)} \calL u\|_{L^2}$ to the left-hand side since their coefficients become small when $T_0$ is sufficiently large. This observation will be useful later in the proof of Main Theorem~\ref{thm:nonlinear-2}.
\end{remark}

\section{Testing by wave packets}
\label{sec:test-by-wp}

The goal of this section is to prove Propositions~\ref{prop:diff-u-gmmvarphi0-wo-1/tdv-involved} and \ref{prop:dotgmm_decay}.

We will sometimes work with the variable $U$ in this section but the results might be stated in terms of $u$, where in particular, $\|\calL u\|_{L^2} = 4t^\frac12\|\calH U\|_{L^2}$ (see \eqref{eq:conj_calLu_calHU}). In this section, all the results hold under the assumptions in Main~Theorem~\ref{thm:linear}.

To simplify the discussion of the obstacle case, we introduce a shorthand \begin{equation}\label{def:X_mu}
    X_\mu(x):= \chi^\ob(x)x_\mu + \frac12(\p_\mu \chi^\ob) |x|^2
\end{equation}
and hence $L_\mu = X_\mu + 2it\p_\mu$. Recall that $\p_x \bfPhi^\ob(t, x) = \frac1{2t}X(t, x)$ and we would adopt the convention omitting indices occasionally.

For any multi-index $\alpha$, we write \begin{equation}\label{defn:U-higher-order}
    \p^\alpha u =: t^{-\frac32} e^{i\bfPhi^\ob} U^{(\alpha)}.
\end{equation}
We also use $U^{(k)}$ notation when $|\alpha| = k$.

\subsection{Auxiliary results for wave packet analysis}

We introduce a few auxiliary results to prepare for the proof.
\begin{lemma}\label{lem:L6-H1-embedding}
    The following Sobolev embedding result holds: \[
        t^{-1}\|Lu\|_{L^6(r \ageq R)} \lesssim t^{-2}\|\calL u\|_{L^2(r \ageq R)} + R^{-\frac12} \|u\|_{L^\infty}.
    \]
\end{lemma}
\begin{proof}
    It suffices to show that \[
        \|\nabla U\|_{L^6(r \ageq R)} \lesssim \|\calH U\|_{L^2(r \ageq R)} + R^{-\frac12} \|U\|_{L^\infty}.
    \]
    Introducing $U_{\ageq R}$ as in Corollary~\ref{cor:out-scale-ell-lap}, then \[
        \|\nabla U\|_{L^6(r \ageq R)} \lesssim \|\nabla U_{\ageq R}\|_{L^6} \lesssim \|\Delta U_{\ageq R}\|_{L^2},
    \]
    where the rest follows exactly as how we manage the term $\|\Delta U_{\ageq R}\|_{L^2}$ in Corollary~\ref{cor:out-scale-ell-lap}.
\end{proof}

\begin{lemma}[Higher order Sobolev embedding]\label{lem:weighted-embedding-Linfty}
    For any $R \aleq t$, for each $0 \leq N \leq M_c$, the following holds for
    any function $\psi$:
    \begin{align*}
        \|\p_x^{(\leq N)} \psi\|_{L^\infty(r < R)} \aleq&\, t^{-2}\|r^{\frac12}\p_x^{(\leq N)} L^{(2)}\psi\|_{\ell^\infty L^2(r \aleq R)} + t^{-1}\|r^{-\frac12}\p_x^{(\leq N)} L^{(1)}\psi\|_{\ell^\infty L^2(r \aleq R)} + \|r^{-\frac32}\p_x^{(\leq N)}\psi\|_{\ell^\infty L^2(r \aleq R)}.
    \end{align*}
\end{lemma}
\begin{proof}
    Choose $\chi_{<R}$ such that $\chi_{<R} \equiv 1$ in $\{r < R\}$ and is supported in $\{r < 2R\}$. Set $\psi_R:= \chi_{<R}\psi$. Translating the standard Sobolev embedding $\ell^\infty H^{2, -\frac32} \subset L^\infty$ into conjugated derivatives, it implies that
    \begin{align*}
        \|\p_x^{(\leq N)} \psi_R\|_{L^\infty} \aleq t^{-2}\|r^{\frac12}L^{(2)}\p_x^{(\leq N)} \psi_R\|_{\ell^\infty L^2} + t^{-1}\|r^{-\frac12}L^{(1)}\p_x^{(\leq N)} \psi_R\|_{\ell^\infty L^2} + \|r^{-\frac32}\p_x^{(\leq N)}\psi_R\|_{\ell^\infty L^2}.
    \end{align*}
    The commutator structure is simple in the obstacle-free case \begin{equation}\label{eq:comm-L-k-der}
    L\p_x^k = \p_x L\p_x^{k-1}  - \p_x^{k-1} = \cdots = \p_x^k L - k\p_x^{k-1}.
    \end{equation}
    In general, we have \[
        [L^{(2)}, \p_x^k] = O(1)\p_x^{(\leq k-1)}L + O(1)\p_x^{(\leq k-1)}.
    \]
    Thus, noting our support assumption of $\psi_R$, which allows us to trade $t$-decay with $r$-decay, we have \[
       \|\p_x^{(\leq N)} \psi_R\|_{L^\infty} \aleq t^{-2}\|r^{\frac12}\p_x^{(\leq N)} L^{(2)}\psi_R\|_{\ell^\infty L^2} + t^{-1}\|r^{-\frac12}\p_x^{(\leq N)} L^{(1)}\psi_R\|_{\ell^\infty L^2} + \|r^{-\frac32}\p_x^{(\leq N)}\psi_R\|_{\ell^\infty L^2}.
    \]
    Then we seek to commute the $\chi_{<R}$. It leads to
    \begin{align*}
       &\|r^{-\frac32}\p_x^{(\leq N)}\psi_R\|_{\ell^\infty L^2} \aleq \|r^{-\frac32}\p_x^{(\leq N)}\psi\|_{\ell^\infty L^2(r \aleq R)}, \\
       &t^{-1}\|r^{-\frac12}\p_x^{(\leq N)} L^{(1)}\psi_R\|_{\ell^\infty L^2} \aleq t^{-1}\|r^{-\frac12}\p_x^{(\leq N)} (\chi_{<R} L^{(1)}\psi)\|_{\ell^\infty L^2} + R^{-\frac32}\|\p_x^{(\leq N)} \psi\|_{\ell^\infty L^2(r \aeq R)} \\
       \aleq&\, t^{-1}\|r^{-\frac12}\p_x^{(\leq N)}  L^{(1)}\psi\|_{\ell^\infty L^2(r \aleq R)} + t^{-1}R^{-\frac32}\|\p_x^{(\leq N-1)} L^{(1)}\psi\|_{\ell^\infty L^2(r \aeq R)} + R^{-\frac32}\|\p_x^{(\leq N)} \psi\|_{L^2(r \aeq R)},
    \end{align*}
    where the last term in the second estimate can be controlled in terms of the first estimate.
    Moreover,
    \begin{align*}
        &t^{-2}\|r^{\frac12}\p_x^{(\leq N)} L^{(2)}\psi_R\|_{\ell^\infty L^2} \aleq t^{-2}\|r^{\frac12}\p_x^{(\leq N)} L^{(1)} (\chi_{<R}L^{(1)}\psi)\|_{\ell^\infty L^2} + t^{-1}R^{-\frac12}\|\p_x^{(\leq N)} L^{(1)}\psi\|_{L^2(r \aeq R)} \\
        \aleq&\, t^{-2}\|r^{\frac12}\p_x^{(\leq N)} (\chi_{<R}L^{(2)}\psi)\|_{\ell^\infty L^2} + t^{-1}R^{-\frac12}\|\p_x^{(\leq N)} L^{(1)}\psi\|_{L^2(r \aeq R)} \\
        \aleq&\, t^{-2}\|r^{\frac12}\p_x^{(\leq N)} L^{(2)}\psi\|_{\ell^\infty L^2(r \aleq R)} + t^{-2}R^{-\frac12} \|\p_x^{(\leq N-1)} L^{(2)}\psi\|_{L^2(r \aeq R)} + t^{-1}R^{-\frac12}\|\p_x^{(\leq N)} L^{(1)}\psi\|_{L^2(r \aeq R)},
    \end{align*}
    where the last term can be controlled in terms of the previous estimate. Therefore, combining the three estimates, it leads to the desired estimate.
\end{proof}

\begin{lemma}\label{lem:near-zone-asymp-zero-order}
    There exists $T_0$ such that for all $t > T_0$, for $0 < \delta < \delta_1 \ll 1$, we have \begin{align*}
        t^{-\frac32}\|U - c\varphi_0\|_{L^\infty_x(r \aleq t^{\frac12 + \dlt})} \aleq t^{-2}t^{\frac14+\frac{\dlt_1}{2}}\log t\|\calL u\|_{L^2}
        + (t^{-(\dlt_{1} - \dlt)} + t^{-\eta (\frac{1}{2}+\dlt)})\|u\|_{L^\infty},
    \end{align*}
    where $c = c[U]$ is defined as \begin{equation}\label{eq:const_c}
        c = c[U](t):= \frac1{4\pi} \int |x|^{-1} [\calH, \chi_{<1}(t^{-\frac12 - \dlt_1}x)] U(t, x)\, \ud x
    \end{equation}
    and $\chi_{<1}$ is a cut-off function adapted to unit balls in the sense that $\chi_{<1}(t^{-\frac12-\dlt_1}x) = 1$ on the region $\{r \aleq t^{\frac12+\dlt}\}$ in which the left-hand side is restricted.
\end{lemma}
\begin{proof}
    We consider the formula \begin{equation}\label{eq:identity-to-invert-calH}
        \calH (\chi_{<1}(t^{-\frac{1}{2}-\dlt_{1}} x) U(t, x)) = \chi_{<1}(t^{-\frac{1}{2}-\dlt_{1}} x) \calH U(t, x) + [\calH, \chi_{<1}(t^{-\frac{1}{2}-\dlt_{1}} x)] U(t, x).
    \end{equation}
    Then we could obtain the following refinement in the smaller (original) region:
\begin{equation*}
    \begin{aligned}
        \|U - c\varphi_0\|_{L^\infty_x({r < t^{\frac12 + \delta})}} &\aleq \|\chi_{<1}(t^{-\frac12-\dlt_1}x)U - c\varphi_0\|_{L^\infty_x({r < t^{\frac12 + \delta})}} \aleq \|\chi_{<1}(t^{-\frac12-\dlt_1}x) U - c\varphi_0\|_{\ell^\infty_r H^{2, -\frac{3}{2}}(r \aleq t^{\frac12 + \delta})} \\
        &\aleq \nrm{\calH U}_{\ell^1_r H^{0,\frac12}(r \aleq t^{\frac12 + \dlt_1})} + (t^{-(\dlt_{1} - \dlt)} + t^{- \eta (\frac{1}{2}+\dlt)}) \nrm{U(t, \cdot)}_{H^{2, -\frac{3}{2}}(\set{r \aeq t^{\frac{1}{2}+\dlt_{1}}})},
    \end{aligned}
\end{equation*}
where the first two inequalities are due to the support property $\set{\abs{x} < t^{\frac{1}{2}+\dlt}} \subset \{ \chi_{<1}(t^{-\frac12-\dlt_1}x) = 1\}$ and Lemma~\ref{lem:weighted-embedding-Linfty}. To achieve the third inequality, we first invert $\calH$ in \eqref{eq:identity-to-invert-calH} and apply Propositions~\ref{prop:w-ell-0} and \ref{prop:ell-asymp} to the first and second terms on the right-hand side, respectively.
Rewriting the last term into $u$-notation and applying Proposition~\ref{prop:dyadic-ell-0}, we obtain
\begin{equation}\label{eq:est-U-wt-Sob-norm-aux}
    t^{-\frac32}\nrm{U(t, \cdot)}_{H^{2, -\frac{3}{2}}(\set{r \aeq t^{\frac{1}{2}+\dlt_{1}}})} \lesssim t^{-2} t^{\frac14 + \frac{\dlt_1}{2}}\|\calL u\|_{L^2} + \|u\|_{L^\infty}.
\end{equation}
Combining everything, the proof follows.
\end{proof}
\begin{remark}
    It might be tempting to state that one can directly replace $U$ by $U^{(k)}$ and obtain a similar result. It indeed is the case, but what we would get on the right-hand side is $\|\calL \p_x^{(\leq N)}u\|_{L^2}$ instead of $\|\p_x^{(\leq N)} \calL u\|_{L^2}$. In this setting, Proposition~\ref{prop:elliptic-near} appears to be indispensable for controlling the commutator. However, the primary purpose of exploiting the additional structure in Proposition~\ref{prop:ell-asymp} is to avoid such a loss, rendering the resulting estimate ineffective for our purposes. Instead, we keep track of the function $U - c\varphi_0$, then Proposition~\ref{prop:elliptic-near} is applicable since we already have a gain of decay at the zeroth order for it.
\end{remark}

\begin{corollary}\label{cor-ptwise-est-higher-order}
    Given any $\dlt \in [0, \frac12]$, for each $0 \leq N \leq M_c$, the following holds for
    any function $\psi$:
    \begin{align*}
        \|\p_x^{(\leq N)} \psi\|_{L^\infty(r < t^{\frac12+\dlt})} \aleq t^{-2}\|r^{\frac12}\p_x^{(\leq N)}\calL \psi\|_{\ell^1 L^2(r \aleq t^{\frac12+\dlt})} + \|\psi\|_{L^\infty(r \aleq t^{\frac12 + \dlt})}.
    \end{align*}
\end{corollary}
\begin{proof}
    Setting $R \aeq t^{\frac12 + \dlt}$, the result follows from Lemma~\ref{lem:weighted-embedding-Linfty} and Proposition~\ref{prop:elliptic-near}.
\end{proof}
\begin{remark}
    Note that if Proposition~\ref{prop:elliptic-near} instead has a localized higher order $L^2$-error, then one could estimate higher order errors by \begin{align*}
        R^{-\frac32}\|\p_x^{(1\leq \cdot\leq N)} \psi\|_{L^2(r \aeq R)} &\aleq R^{-\frac32}t^{-1}\|\p_x^{(\leq N-1)} L^{(1)} \psi\|_{L^2(r \aeq R)} + R^{-\frac32}t^{-1}R\|\p_x^{(\leq N-1)} \psi\|_{L^2(r \aeq R)} \\
        &\aleq (t^{-1}R + R^{-1})(t^{-2}\|r^{\frac12}\p_x^{(\leq N-1)}\calL \psi\|_{L^2(r\aeq R)} + \|\p_x^{(\leq N-1)}\psi\|_{L^\infty(r \aeq R)})
    \end{align*}
    and iteratively estimate the rest.
\end{remark}

Note that Corollary~\ref{cor-ptwise-est-higher-order} permits the reduction of higher order $L^\infty$-estimates to zeroth order ones (together with other terms in terms of $\calL u$), where the zeroth order $L^\infty$-estimate has a gain in decay for specific functions. Therefore, this allows us to recover the gain for higher order $L^\infty$-estimates as well as follows.
\begin{corollary}\label{cor:near-zone-asymp-higher-order}
    There exists $T_0$ such that for any $0 < \delta < \delta_1 \ll 1$, and $0 \leq N \leq M_c$, the following estimate holds for all $t > T_0$, \begin{align*}
        \|\p_x^{(\leq N)}(u - ct^{-\frac32}e^{i\bfPhi^\ob}\varphi_0)\|_{L^\infty_x(r \aleq t^{\frac12 + \dlt})} \aleq t^{-2}t^{\frac14+\frac{\dlt_1}{2}}\log t\|\p_x^{(\leq N)}\calL u\|_{L^2}
        + (t^{-(\dlt_{1} - \dlt)} + t^{- \eta (\frac{1}{2}+\dlt)})\|u\|_{L^\infty},
    \end{align*}
    where $c$ is defined as in Lemma~\ref{lem:near-zone-asymp-zero-order}.
\end{corollary}
\begin{proof}
    Note that $N = 0$ is exactly Lemma~\ref{lem:near-zone-asymp-zero-order}. For higher $N$'s, it follows by taking $\psi:= u - ct^{-\frac32} e^{i\bfPhi^\ob}\varphi_0$ in Corollary~\ref{cor-ptwise-est-higher-order} and applying Lemma~\ref{lem:near-zone-asymp-zero-order}. 
\end{proof}

\subsection{Proof of Lemma~\ref{lem:wp-aux}}\label{sec:proof_lem:wp-aux}
First, \eqref{eq:dv-gmm} follows directly from the convolution form of $\gmm$ in \eqref{eq:defn_gmm}; the proof is standard.

Next, we prove \eqref{eq:relation-der-0th-w-higher-order-gmm}. 
We consider the case of $\calK \neq \0$ in detail; the case $\calK = \0$ is simpler. Given $u|_{[0, \infty) \times \p\calK} = 0$, we compute
\begin{equation}\label{eq:prop-diff-u-gmmvarphi0-w-1/tdv-involved-eq-0th}
    \begin{split}
    \frac1{t}\p_v (\gmm(t, v)e^{i\bfPhi^\ob(t, tv)}) &= \left\langle u(t, x), e^{i\bfPhi^\ob(t, x)} (-\tfrac{iX(tv)}{2t} + \tfrac1{t}\p_v)\chi^v\right\rangle_x e^{i\bfPhi^\ob(t, tv)} \\
    &= \left\langle (\tfrac{iX(tv)}{2t}+\p_x) (u(t, x)e^{-i\bfPhi^\ob(t, x)}), \chi^v(t, x)\right\rangle_x e^{i\bfPhi^\ob(t, tv)} \\
    &= \left\langle \tfrac{i(X(tv) - X(x))}{2t} u(t,x) + \p_x u(t, x), \Psi_v(t, x)\right\rangle_x e^{i\bfPhi^\ob(t, tv)} \\
    &= \left\langle \tfrac{i(X(tv) - X(x))}{2t} u(t,x), \Psi_v(t, x)\right\rangle_x e^{i\bfPhi^\ob(t, tv)} + \gmm[\p_x u](t, v) e^{i\bfPhi^\ob(t, tv)}.
    \end{split}
\end{equation}
Though higher derivatives do not satisfy the Dirichlet boundary condition as the zeroth order, one could just keep track of the boundary terms and put them into $L^\infty$-norms due to the compactness of $\p\calK$. Note that the last term in \eqref{eq:prop-diff-u-gmmvarphi0-w-1/tdv-involved-eq-0th} has exactly the same structure as the one to start with, and the penultimate term also has a similar structure with a mild dependence on $v$. Therefore, it is then not hard to see from an induction that \begin{equation*}
\begin{aligned}
    &\left|(\tfrac1{t}\p_v)^k (\gmm(t, v)e^{i\bfPhi^\ob(t, tv)}) - \gmm[\p_x^k u](t, v) e^{i\bfPhi^\ob(t, tv)}\right| \\
    \aleq&\, t\|\p_x^{(\leq k)} u\|_{L^\infty_x} + \|\p_x^{(\leq k)}u(t, \cdot)\|_{L^\infty(\p\calK)} \aleq t\|\p_x^{(\leq k)}u(t, \cdot)\|_{L^\infty_x},
\end{aligned}
\end{equation*}
where we note that $\p_x^{(\leq M_0)} u \in C^0(\overline{\R^3\setminus\calK})$ thanks to \eqref{eq:E-u}. This proves \eqref{eq:relation-der-0th-w-higher-order-gmm}.

\subsection{Proof of Proposition~\ref{prop:diff-u-gmmvarphi0-wo-1/tdv-involved}}\label{sec:proof_prop:diff-u-gmmvarphi0-wo-1/tdv-involved}

We claim that it suffices to prove the following result, and then Proposition~\ref{prop:diff-u-gmmvarphi0-wo-1/tdv-involved} will follow as a corollary:

\begin{proposition}\label{prop:diff-u-gmmvarphi0-w-1/tdv-involved}
    Assume that $\calH$ satisfies \ref{hyp:af-0} and \ref{hyp:se}. Then there exists $T_0$ such that for all $t > T_0$, the following pointwise bound holds \begin{align*}
    &\left\|t^{\frac32}(\p_x^k u)(t, tv) - (\tfrac1{t}\p_v)^k(\gmm(t, v)\varphi_0(t, tv)e^{i\bfPhi^\ob(t, tv)})\right\|_{L^\infty_v} \\
    \aleq&\, t^{-\frac14+\frac12\eta} \|\p_x^{(\leq k)}\calL u(t, \cdot)\|_{L^2} + t^{\frac32-\frac{\eta}{4}} \|\p_x^{(\leq k)} u(t, \cdot)\|_{L^\infty} + t^{-\frac14}\|\p_x^{(\leq k)}u(t, \cdot)\|_{L^2}, \hbox{ for all } 0 \leq k \leq M_0.
\end{align*}
\end{proposition}

Indeed, combining Proposition~\ref{prop:diff-u-gmmvarphi0-w-1/tdv-involved} and \eqref{eq:relation-der-0th-w-higher-order-gmm}, for any multi-index $\alpha$ with $|\alpha| = N$, it follows that \begin{align*}
    &\left\|t^{\frac32}(\p_x^\alpha u)(t, tv) - \sum_{|\beta_1| + |\beta_2| = N} \gmm[\p_x^{\beta_1} u](t, v) e^{i\bfPhi^\ob(t, tv)} (\p_x^{\beta_2} \varphi_0)(t, tv) \right\|_{L^\infty_v} \\
    \aleq\, &t^{-\frac14+\frac12\eta} \|\p_x^{(\leq N)}\calL u(t, \cdot)\|_{L^2_x} + t^{\frac32-\frac{\eta}{4}} \|\p_x^{(\leq N)} u(t, \cdot)\|_{L^\infty_x} + t^{-\frac14}\|\p_x^{(\leq N)}u(t, \cdot)\|_{L^2_x}.
\end{align*}
This completes the proof of Proposition~\ref{prop:diff-u-gmmvarphi0-wo-1/tdv-involved}.

In the following, we prove Proposition~\ref{prop:diff-u-gmmvarphi0-w-1/tdv-involved}.

\subsubsection{The case $N = 0$ with no obstacle}
Recall from \eqref{eq:defn_gmm} that \[
    \gmm(t, v) = U(t, tv) *_v t^{\frac32}\chi(\sqrt{t}v),
\]
we write
\[
    U(t, vt) - \gmm(t, v)\varphi_0(t, tv) = \int \big(U(t, tv) - U(t, t(v-y))\varphi_0(t, tv)\big) t^{\frac32}\chi(\sqrt{t}y)\, \ud y =: \int D[U](t, v, y) t^{\frac32} \chi(\sqrt{t}y)\, \ud y.
\]

We fix $0 < \delta \ll 1$ to be determined in the proof process below. We distinguish two regions for $v$, where a specific constant $2$ is chosen based on the support condition $\supp \chi \subset B(0, 1)$ (see Definition~\ref{defn_wp}). In the second case, such a constant will then lead to the precise choice of the support condition for $\chi_{<1}$, which we used to define $c$.

\noindent \textit{Case 1: $|v| \geq 2t^{-\frac12 + \delta}$. }
We write \[
    D[U](t, v, y) = D_1 + D_2,
\]
where \[
    D_1:= U(t, tv) - U(t, t(v-y)), \quad D_2:= -U(t, t(v-y)) (\varphi_0(t, tv) - 1).
\]
Thanks to the support condition of $\chi$, when testing against $\chi(\sqrt{t}y)$, we have $|t(v-y)| \geq t^{\frac12}(2t^\delta - 1) \geq t^{\frac12+\delta}$ for $t \geq T_0 > 1$. Therefore, Morrey's inequality (Lemma~\ref{lem:Morrey-out-scale}) is applicable to $D_1$, which gives\[
    \begin{split}
        \int t^{\frac32} \chi(\sqrt{t} y) |D_1|\, \ud y &\aleq \int t^{\frac32} \chi(\sqrt{t} y) |ty|^{\frac12} \left(t^{-\frac12}\|\calL u\|_{L^2} + t^{\frac32}t^{-\frac12(\frac12+\dlt)}\|u\|_{L^\infty}\right)\, \ud y \\
        &\aleq t^{-\frac14}\|\calL u(t, \cdot)\|_{L^2} + t^{\frac32-\frac{\delta}{2}} \|u(t, \cdot)\|_{L^\infty}.
    \end{split}
\]
For $D_2$, we choose $\delta' = \eta$ in Corollary~\ref{cor:decay_varphi0} and it yields \[
    \begin{aligned}
        &\left|\int t^{\frac32}\chi(\sqrt{t}y) D_2\, \ud y\right|
        \lesssim \|U\|_{L^\infty} \brk{tv}^{-\eta} \lesssim t^{\frac32-\frac{\eta}{2}}\|u\|_{L^\infty}.
    \end{aligned}
    \]

\noindent \textit{Case 2: $|v| \leq 2t^{-\frac12 + \delta}$. }
In this case, we decompose $D[U](t, v, y)$ into $D = D_3 + D_4 + D_5$, where \[\begin{aligned}
    D_3 &:= - c\varphi_0(t, tv)(\varphi_0(t, t(v-y)) - 1), \quad D_4:= U(t, tv) - c\varphi_0(t, tv), \\ D_5 &:= - \varphi_0(t, tv)(U(t, t(v-y)) - c\varphi_0(t, t(v-y))).
\end{aligned}
\]
Here, we choose $\delta_1 > \delta$ to be determined, and the constant (independent of $x$) $c = c[U](t)$ above is defined by \eqref{eq:const_c}.
In particular, we have
\begin{equation*}
    \abs{c} \aleq \nrm{U(t, \cdot)}_{H^{2, -\frac{3}{2}}(\set{\abs{x} \aeq t^{\frac{1}{2}+\dlt_{1}}} )}.
\end{equation*}
Choose $\delta' = \frac{3\eta}{3-\eta} \in (0, 2\eta)$ in Corollary~\ref{cor:decay_varphi0}, in view of $\brk{x}^{-\delta'} \in L^{\frac{3+\delta'}{\delta'}}$, we estimate $D_3$ as follows:  \[
\begin{aligned}
    &\left|\int t^{\frac32}\chi(\sqrt{t}y) D_3\, \ud y\right|
    \lesssim  |c| \int \frac1{\brk{t(v-y)}^{\delta'}} t^{\frac32}\chi(\sqrt{t}y)\, \ud y \\
    \lesssim&\, \abs{c} \|\brk{ty}^{-\delta'}\|_{L^{\frac{3+\delta'}{\delta'}}_y} \|t^{\frac32}\chi(\sqrt{t}y)\|_{L^{\frac{3+\delta'}{3}}_y} \lesssim t^{-\frac{3\delta'}{2(3+\delta')}} |c|
    \lesssim t^{-\frac{\eta}{2}} \nrm{U(t, \cdot)}_{H^{2, -\frac{3}{2}}(\set{\abs{x} \aeq t^{\frac{1}{2}+\dlt_{1}}} )}.
\end{aligned}
\]
To estimate $D_4$ and $D_5$, it suffices to estimate $|U(t, x) - c\varphi_0(t, x)|$ in the region \begin{equation}\label{eq:supp_varphi_case2}
    \{|x| \leq 3t^{\frac12 + \delta}\} \subset \{x: \chi_{<1}(t^{-\frac{1}{2}-\dlt_{1}} x)  = 1\}, \quad \text{ for } t \geq T_0 > 1,
\end{equation}
where we remark that this actually motivates the precise choice of $\chi_{<1}$ in \eqref{eq:const_c}.
\begin{equation*}
    \begin{aligned}
        t^{-\frac32}\|U - c\varphi_0\|_{L^\infty_x({r \leq 3 t^{\frac12 + \delta})}}
        &\aleq t^{-2}t^{\frac{1}{4}+\frac{\dlt_1}{2}}\log t \nrm{\calL u}_{L^{2}} + (t^{-(\dlt_{1} - \dlt)} + t^{- \eta (\frac{1}{2}+\dlt)}) \|u\|_{L^\infty}.
    \end{aligned}
\end{equation*}
This gives rise to the estimate of terms involving $D_4$ and $D_5$.
By choosing $\delta = \frac12 \eta$ and $\delta_1 = \frac34 \eta$, we conclude that for any $v \in \R^3$, \[
    |U(t, tv) - \gmm(t, v) \varphi_0(t, tv)| \lesssim t^{-\frac14 + \frac12\eta} \|\calL u(t, \cdot)\|_{L^2} + t^{\frac32-\frac{\eta}{4}} \|u(t, \cdot)\|_{L^\infty} + t^{-\frac{\eta}{4}} \|U(t, \cdot)\|_{H^{2,-\frac32}(|x| \simeq t^{\frac12 + \delta_1})},
\]
where the last term can be further estimated by \eqref{eq:est-U-wt-Sob-norm-aux}.
This finishes the proof of Proposition~\ref{prop:diff-u-gmmvarphi0-w-1/tdv-involved} in the case $N = 0$.

\begin{remark}
    In the region $\set{|v| \gtrsim t^{-\frac12 + \delta}}$, $|U(t, vt) - \gmm(t,v)|$ shares the same estimate due to the decay property of $\tilde\varphi_0$ (see Corollary~\ref{cor:decay_varphi0}). However, introducing $\varphi_0(t, tv)$ is crucial in the region $\set{|v| \lesssim t^{-\frac12 + \delta}}$ where the perturbation of coefficients most significantly influences dispersion. Heuristically, the flat wave packet would fail to provide a good approximation in this regime.
\end{remark}

\subsubsection{The case $N = 0$ with obstacle}

With $\calK \neq \0$, \eqref{eq:defn_gmm} gives rise to \[
    \gmm(t, v) = t^{\frac32}\int_{\R^3\setminus (v-t^{-1}\calK)} U(t, t(v-y))\chi(\sqrt{t}y)\, \ud y.
\]
We notice that \[
    U(t, tv) - \gmm(t, v)\varphi_0(t, tv) = \int_{\R^3\setminus (v-t^{-1}\calK)} D[U](t,v,y) t^{\frac32}\chi(\sqrt{t}y)\, \ud y + U(t, tv) \int_{t^{\frac12}v-t^{-\frac12}\calK} \chi(y)\, \ud y,
\]
where a trivial bound can be obtained for the additional error: \begin{align*}
    \left|\int_{t^{\frac12}v-t^{-\frac12}\calK} \chi(y)\, \ud y\right| \aleq \mu(t^{\frac12}v-t^{-\frac12}\calK) \aleq t^{-\frac32} \mu(\calK).
\end{align*}
The rest of the proof stays exactly the same and hence we only focus on the case $\calK = \0$ in the higher order case below.

\subsubsection{Higher order cases}

As above, we consider two cases.
We record a conjugation identity first to motivate our computation below \[
    (\p_x^k u)(t, tv) - (\tfrac1{t}\p_v)^k(\gmm(t, v)e^{i\bfPhi^\ob(t, tv)}\varphi_0(t, tv)) = t^{-\frac32} e^{i\bfPhi^\ob(t, tv)} \left( U^{(k)}(t, tv) - (\tfrac{i}{2t}X(vt) + \tfrac1{t}\p_v)^k \big(\gmm(t,v)\varphi_0(t, tv)\big)\right).
\]

\noindent \textit{Case 1: $|v| \geq 2t^{-\frac12 + \delta}$. }

We first compute \[
    \Big|U^{(k)}(t, tv) - \left((\tfrac{i}{2t}X(vt) + \tfrac1{t}\p_v)^k \gmm(t, v)\right) \varphi_0(t, tv) - \int \left(U^{(k)}(t, tv) -  U^{(k)}(t, t(v-y)) \varphi_0(t, tv)\right) t^{\frac32}\chi(\sqrt{t}y)\, \ud y\Big|,
\]
and we notice that it produces error terms similar to the ones appearing in~\eqref{eq:prop-diff-u-gmmvarphi0-w-1/tdv-involved-eq-0th}. By repeating the same analysis as in the zeroth order case we obtain \begin{align*}
    |U^{(k)}(t, tv) - \left((\tfrac{i}{2t}X(vt) + \tfrac1{t}\p_v)^k \gmm(t, v)\right) \varphi_0(t, tv)| \aleq t^{-\frac14} \|\calL \p_x^k u(t, \cdot)\|_{L^2(r \ageq t^{\frac12+\frac12\eta})} + t^{\frac32-\frac{\eta}{4}} \|\p_x^k u(t, \cdot)\|_{L^\infty}.
\end{align*}
Applying Proposition~\ref{prop:elliptic-int-rad}, we derive \begin{align*}
    \|\calL \p_x^k u(t, \cdot)\|_{L^2(r \ageq t^{\frac12+\frac12\eta})} \aleq \|\p_x^{(\leq k)}\calL u(t, \cdot)\|_{L^2(r \ageq t^{\frac12+\frac12\eta})} + t^{\frac74-\frac14\eta}\|\p_x^{(\leq k)}u\|_{L^\infty} + \|\p_x^{(\leq k)} u\|_{L^2}
\end{align*}
and hence, \begin{align*}
    &|t^{\frac32}(\p_x^k u)(t, tv) - (\tfrac1{t}\p_v)^k(\gmm(t, v)e^{i\bfPhi^\ob(t, tv)})\varphi_0(t, tv)| \\
    \aleq&\, t^{-\frac14} \|\p_x^{(\leq k)}\calL u(t, \cdot)\|_{L^2} + t^{\frac32-\frac{\eta}{4}} \|\p_x^{(\leq k)} u(t, \cdot)\|_{L^\infty} + t^{-\frac14}\|\p_x^{(\leq k)}u\|_{L^2}.
\end{align*}
Performing an induction on $k$, one could then estimate the following difference using Corollary~\ref{cor:decay_varphi0}.
\begin{align*}
    &\left|(\tfrac1{t}\p_v)^k(\gmm(t, v)e^{i\bfPhi^\ob(t, tv)})\varphi_0(t, tv) - (\tfrac1{t}\p_v)^k(\gmm(t, v)e^{i\bfPhi^\ob(t, tv)}\varphi_0(t, tv))\right| \\
    \aleq&\, \left|(\tfrac1{t}\p_v)^{(\leq k-1)}(\gmm(t, v)e^{i\bfPhi^\ob(t, tv)})\right| \cdot |\p_x^{(\geq 1)}\varphi_0(t, tv)| \\
    \aleq&\, t^{-\frac{\eta}{2}}|t^{\frac32}\p_x^{(\leq k-1)}u| + \sum_{j \leq k-1}|t^{\frac32}\p_x^{j}u(t, tv) - (\tfrac1{t}\p_v)^j(\gmm(t, v)e^{i\bfPhi^\ob(t, tv)})|.
\end{align*}
Then the result follows.

\noindent \textit{Case 2: $|v| \leq 2t^{-\frac12 + \delta}$. }

We write \[
    U^{(k)}(t, tv) - \gmm(t, v) (\tfrac{i}{2t}X(vt) + \tfrac1{t}\p_v)^k\varphi_0(t, tv) = \int \left(U^{(k)}(t, tv) -  U(t, t(v-y)) ((\tfrac{iX}{2t}+\p_x)^k\varphi_0)(t, tv)\right) t^{\frac32}\chi(\sqrt{t}y)\, \ud y.
\]
Then the integrand satisfies the following decomposition: \begin{align*}
    &U^{(k)}(t, tv) -  U(t, t(v-y)) ((\tfrac{iX}{2t}+\p_x)^k\varphi_0)(t, tv) \\
    =&\, U^{(k)}(t, tv) - c((\tfrac{iX}{2t}+\p_x)^k\varphi_0)(t, tv) - c((\tfrac{iX}{2t}+\p_x)^k\varphi_0)(t, tv) (\varphi_0(t, t(v-y)) - 1) \\
    &\qquad\qquad\qquad\qquad\qquad\qquad\qquad\qquad\qquad - ((\tfrac{iX}{2t}+\p_x)^k\varphi_0)(t, tv) (U(t, t(v-y)) - c\varphi_0(t, t(v-y))).
\end{align*}
Applying Corollary~\ref{cor:near-zone-asymp-higher-order} with the same $\dlt, \dlt_1, \dlt'$ as in the zeroth order case, one would obtain the following:
\begin{align*}
    |U^{(k)}(t, tv) - \gmm(t, v) \big(\tfrac{i}{2t}X(vt) + \tfrac1{t}\p_v\big)^k\varphi_0(t, tv)| \aleq t^{-\frac14 + \frac12\eta}\|\p_x^{(\leq k)}\calL u\|_{L^2} + t^{\frac32 - \frac{\eta}{4}} \|\p_x^{(\leq k)}u\|_{L^\infty}.
\end{align*}
Switching to the $u$-notation, \begin{align*}
    &|t^{\frac32}(\p_x^k u)(t, tv) - \gmm(t, v)(\tfrac1{t}\p_v)^k(\varphi_0(t, tv)e^{i\bfPhi^\ob(t, tv)})| \\
    \aleq&\, t^{-\frac14+\frac12\eta} \|\p_x^{(\leq k)}\calL u(t, \cdot)\|_{L^2} + t^{\frac32-\frac{\eta}{4}} \|\p_x^{(\leq k)} u(t, \cdot)\|_{L^\infty}.
\end{align*}
To derive the desired estimate, we consider the difference \begin{align*}
    &\left|\gmm(t, v)(\tfrac1{t}\p_v)^k(\varphi_0(t, tv)e^{i\bfPhi^\ob(t, tv)}) - (\tfrac1{t}\p_v)^k(\gmm(t, v)\varphi_0(t, tv)e^{i\bfPhi^\ob(t, tv)})\right| \\
    \aleq&\, \left|(\tfrac1{t}\p_v)^{(\leq k-1)}(\varphi_0(t, tv)e^{i\bfPhi^\ob(t, tv)})\right| \cdot |(\tfrac1{t}\p_v)^{(\geq 1)}\gmm(t, v)| \aleq t\|u\|_{L^\infty},
\end{align*}
where we apply Corollary~\ref{cor:decay_varphi0} and \eqref{eq:dv-gmm} in the last step.

\subsection{Proof of Proposition~\ref{prop:dotgmm_decay}}\label{sec-proof_prop:dotgmm_decay}

We write \[
    \frac{d}{dt} \gmm[\p_{x}^k u](t, v) = -i\brk{\p_{x}^k \calH u + \p_{x}^k f, \Psi_v} + \brk{\p_{x}^k u, \p_t \Psi_v},
\]
We expand \[
    \calL u = 4t^2\calH u + 2it(\p_\mu (a^{\mu\nu}X_\nu u) + X_\mu a^{\mu\nu}\p_\nu u) + a^{\mu\nu} X_\mu X_\nu u - 4tb^\mu X_\mu u.
\]
Therefore, \begin{align*}
\frac{d}{dt} \gmm[\p_{x}^k u](t, v)
= &\underbrace{-\frac{i}{4t^2}\brk{\p_x^k\calL u, \Psi_v}}_{:= R_{1, k}} \underbrace{ -i\brk{\p_x^k f, \Psi_v}}_{:= R_{f, k}} - \frac1{2t} \brk{\p_x^k(\p_\mu( a^{\mu\nu}X_\nu u) + X_\mu a^{\mu\nu}\p_\nu u), \Psi_v} \\
&+ \frac{i}{4t^2}\brk{\p_x^k(a^{\mu\nu} X_\mu X_\nu u - 4tb^\mu X_\mu u), \Psi_v} + \brk{\p_x^k u, \p_t \Psi_v} \\
= &R_{1, k} + R_{f, k} - \frac1{2t}\brk{\p_x^k\big((\p_\mu a^{\mu\nu} + 2ib^\nu) X_\nu u\big), \Psi_v} - \frac1{2t}\brk{\p_x^k\big(a^{\mu\nu} (\p_\mu X_\nu) u\big), \Psi_v} \\
& - \frac1{t}\brk{\p_x^k(X_\mu a^{\mu\nu}\p_\nu u), \Psi_v} + \frac{i}{4t^2}\brk{\p_x^k(a^{\mu\nu}X_\mu X_\nu u), \Psi_v} + \brk{\p_x^k u, \p_t \Psi_v}.\\
\end{align*}
To expand the last term, we first record an integration-by-parts formula: for all $g, h \in C^{1}(\overline{\calM^{3}})$, \[
    \brk{g, (\p_\mu h)e^{i\bfPhi^\ob}} = -\frac1{2it}\brk{g, L_\mu(h e^{i\bfPhi^\ob})} = -\frac1{2it}\brk{L_\mu g, h e^{i\bfPhi^\ob}} + \int_{\p\calK} g \br{h} e^{-i\bfPhi^\ob} n_{\mu} \, \ud\sigma, 
\]
where $n$ is the inner unit normal to $\rd \calK$.
Now we compute \begin{align*}
    \brk{\p_x^k u, \p_t \Psi_v} &= i\brk{\p_x^k u, \chi^{\ob}\frac{|x|^2}{4t^2}\Psi_v} + \brk{\p_x^k u, t^{-\frac32}(\nabla \chi)(\tfrac{x-vt}{\sqrt{t}})(-\frac12 x - \frac12 vt) e^{i\bfPhi^\ob}} \\
    &= i\brk{\p_x^k u, \chi^{\ob}\frac{|x|^2}{4t^2}\Psi_v} + \frac1{2t}\brk{\p_x^k u, \tfrac{x-vt}{\sqrt{t}} \cdot (\nabla \chi)(\tfrac{x-vt}{\sqrt{t}}) e^{i\bfPhi^\ob}} - \frac1{t\sqrt{t}}\brk{\p_x^k u, x^\mu (\p_\mu\chi)(\tfrac{x-vt}{\sqrt{t}})e^{i\bfPhi^\ob}} \\
    &= i\brk{\p_x^k u, \chi^{\ob}\frac{|x|^2}{4t^2}\Psi_v} - \frac1{4it\sqrt{t}} \brk{L \p_x^k u, \tfrac{x-vt}{\sqrt{t}}\Psi_v} - \frac{3}{2t}\brk{\p_x^k u, \Psi_v} + \frac1{2it^2}\brk{L_\mu(x^\mu\p_x^k u), \Psi_v} + B.T. \\
    &= i\brk{\p_x^k u, \chi^{\ob}\frac{|x|^2}{4t^2}\Psi_v} - \frac1{4it\sqrt{t}} \brk{L \p_x^k u, \tfrac{x-vt}{\sqrt{t}}\Psi_v} + \frac{3}{2t}\brk{\p_x^k u, \Psi_v} + \frac1{2it^2}\brk{x^\mu L_\mu\p_x^k u, \Psi_v} + B.T.,
\end{align*}
where \begin{align*}
    B.T.:= \frac1{2\sqrt{t}} \int_{\p\calK}\p_x^k u (\tfrac{x-vt}{\sqrt{t}} \cdot n) \Psi_v \, \ud \sigma - \frac1{t}\int_{\p\calK} \p_x^k u (x \cdot n)\Psi_v\, \ud\sigma.
\end{align*}
Therefore, \begin{align*}
    \frac{d}{dt} \gmm[\p_{x}^k u](t, v) =&\, R_{1, k} + R_{f, k} - \frac1{2t}\brk{\p_x^k\big((\p_\mu a^{\mu\nu} + 2ib^\nu) X_\nu u\big), \Psi_v} - \frac1{2t}\brk{\p_x^k\big(a^{\mu\nu} (\p_\mu X_\nu) u\big), \Psi_v} \\
    & - \frac1{2it^2}\brk{\p_x^k(X_\mu a^{\mu\nu}L_\nu u), \Psi_v} - \frac{i}{4t^2}\brk{\p_x^k(a^{\mu\nu}X_\mu X_\nu u), \Psi_v} + \brk{\p_x^k u, \p_t \Psi_v} \\
    =&\, R_{1, k} + R_{f, k} \underbrace{- \frac1{2t}\brk{\p_x^k\big((\p_\mu a^{\mu\nu} + 2ib^\nu) X_\nu u\big), \Psi_v} - \frac1{2t}\brk{\p_x^k ((a^{\mu\nu} \p_\mu X_\nu - 3)u), \Psi_v}}_{:= R_{2, k}} \\
    &+ \underbrace{\frac{i}{4t^2}\brk{\p_x^k u, (\chi^{\ob}|x|^2 - a^{\mu\nu}X_\mu X_\nu) \Psi_v}}_{:= R_{3, k}} - \frac{i}{4t^2}\brk{[\p_x^k, a^{\mu\nu}X_\mu X_\nu]u, \Psi_v} + B.T. \\
    & - \frac1{2it^2}\brk{[\p_x^k,X_\mu a^{\mu\nu}L_\nu] u, \Psi_v} \underbrace{- \frac1{2it^2}\brk{(X_\mu a^{\mu\nu}- x^\nu) L_\nu \p_x^k u, \Psi_v}}_{:= R_{4, k}} \underbrace{- \frac1{4it\sqrt{t}} \brk{L \p_x^k u, \tfrac{x-vt}{\sqrt{t}}\Psi_v}}_{:= R_{5,k}}.
\end{align*}
Here, thanks to \eqref{eq:E-u}, $\p_x^{(\leq M_0)} u \in C^0(\overline{\calM^3})$ and hence \[
    |B.T.| \aleq t^{-\frac12}\|\p_x^{(\leq N)}u\|_{L^\infty_x}.
\]
In addition, we combine the two extra terms and notice that they add up to \begin{align*}
    &-\frac{i}{4t^2}\brk{[\p_x^k, a^{\mu\nu} X_\mu X_\nu - 2X_\mu a^{\mu\nu} L_\nu]u, \Psi_v} = -\frac{i}{4t^2}\brk{[\p_x^k, -a^{\mu\nu}L_\mu L_\nu + 2ita^{\mu\nu}(\p_\mu X_\nu)]u, \Psi_v} \\
    =&\, \underbrace{\frac{i}{4t^2}\brk{[\p_x^k, a^{\mu\nu}L_\mu L_\nu]u, \Psi_v}}_{:= R_{6, k}} + \underbrace{\frac1{2t}\brk{[\p_x^k, (a^{\mu\nu}\p_\mu X_\nu - 3)]u, \Psi_v}}_{:= R_{7, k}}.
\end{align*}

Note that $R_{6, k}$ and $R_{7,k}$ vanish trivially when $k = 0$. We estimate these error terms respectively in the following.
The first term is the simplest to handle: \[
    |R_{1,k}| \lesssim t^{-\frac54}\|\p_x^k \calL u\|_{L^2}.
\]
Moreover,
\begin{align*}
    |R_{2,k}| &\lesssim t^{-1-\eta} t^{\frac32}\|\p_x^{(\leq k)}u\|_{L^\infty(r \ageq t^{\frac12})} + t^{-\frac14}\|\brk{x}^{-2\eta}\p_x^{(\leq k)}u\|_{L^2(r \aleq t^{\frac12})}\\
    &\lesssim t^{-\frac54 - \eta} \|\p_x^{(\leq k)} \calL u\|_{L^2} + t^{-1-\eta} t^{\frac32}\|\p_x^{(\leq k)}u\|_{L^\infty},
\end{align*}
where, in the second inequality, we apply Proposition~\ref{prop:elliptic-near} together with Cauchy--Schwarz inequalities in $\ell^p$-based spaces. Indeed, \[
    \begin{split}
    \|\brk{x}^{-2\eta}\p_x^{(\leq k)}u\|_{L^2(r \aleq t^{\frac12})}
    &\aleq t^{\frac12(\frac32 - 2\eta)} \|r^{-\frac32} \p_x^{(\leq k)}u\|_{\ell^\infty L^2(r \aleq t^{\frac12})} \aleq t^{\frac12(\frac32 - 2\eta)} \left( t^{-2}\|r^{\frac12} \p_x^{(\leq k)} \calL u\|_{\ell^1 L^2(r \aleq t^{\frac12})} + \|u\|_{L^\infty} \right) \\
    &\aleq t^{-1-\eta} \| \p_x^{(\leq k)} \calL u\|_{L^2(r \aleq t^{\frac12})} + t^{\frac34 - \eta} \|u\|_{L^\infty}.
    \end{split}
\]
\begin{remark}
    To control the last term in the last inequality, one can also simply use the non-endpoint case of Proposition~\ref{prop:elliptic-near} to avoid the use of the $\ell^p$ spaces.
\end{remark}

Separating the integral in $R_{3, k}$ into two regions $\{r \aleq t^{\frac12}\}$ and $\{r \ageq t^{\frac12}\}$, one obtains \begin{align*}
    |R_{3, k}| \lesssim t^{\frac12 - \eta}\|\p_x^k u\|_{L^\infty}
\end{align*}
by applying \eqref{eq:af-extra1-0}. One then needs to exploit \eqref{eq:af-extra2-0} to take care of $R_{4, k}$. For $R_{4, k}$, applying \eqref{eq:comm-L-k-der} to commute derivatives, \begin{align*}
    |R_{4, k}| &\lesssim t^{-\frac54}\|\brk{x}^{1-2\eta}L\p_x^k u\|_{L^2(r \aleq t^{\frac12})} + t^{-\frac14-\eta}\|L\p_x^k u\|_{L^6(r \ageq t^{\frac12})} \\
    &\lesssim t^{-1}\|\brk{x}^{\frac12-2\eta}\p_x^k L u\|_{L^2(r \aleq t^{\frac12})} + t^{-\frac12}\|\brk{x}^{-\frac12-2\eta}\p_x^{(\leq k-1)} u\|_{L^2(r \aleq t^{\frac12})} + t^{-\frac54-\eta}\|\calL\p_x^k u\|_{L^2(r \ageq t^{\frac12})} + t^{\frac12-\eta}\|\p_x^k u\|_{L^\infty} \\
    &\lesssim t^{-\frac54-\eta}\|\p_x^{(\leq k)} \calL u\|_{L^2} + t^{\frac12 - \eta}\|\p_x^{(\leq k)}u\|_{L^\infty} + t^{-\frac54-\eta} \|\p_x^{(\leq k)} u\|_{L^2},
\end{align*}
where we apply Lemma~\ref{lem:L6-H1-embedding} in the second line and Proposition~\ref{prop:elliptic-int-rad} in the last line.

The term $R_{5, k}$ is one of the most difficult to treat since we need to take advantage of \ref{hyp:se}, or more precisely, Proposition~\ref{prop:ell-asymp} and Corollary~\ref{cor:near-zone-asymp-higher-order}. Fix $\dlt > 0$ small enough to be determined in the computation below.
\begin{enumerate}[ wide = 0pt, align = left ]
    \item[Case 1: $|v| \geq 2 t^{-\frac12 + \delta}$.]
    Due to the support condition of $\chi$, the integrand is nontrivial when $|x| \geq |vt| - |x-vt| \geq t^{\frac12 + \delta}$. Therefore, thanks to Lemma~\ref{lem:L6-H1-embedding} and Proposition~\ref{prop:elliptic-int-rad}, we have \begin{align*}
        |R_{5, k}| &\lesssim t^{-\frac14}\|L\p_x^k u\|_{L^6(r \ageq t^{\frac12 + \delta})} \lesssim t^{-\frac54}\|\calL\p_x^k u\|_{L^2(r \ageq t^{\frac12 + \delta})} + t^{\frac12 - \frac12\delta}\|\p_x^k u\|_{L^\infty} \\
        &\lesssim t^{-\frac54}\|\p_x^{(\leq k)} \calL u\|_{L^2(r \ageq t^{\frac12 + \delta})} + t^{\frac12 - \frac12\delta}\|\p_x^{(\leq k)} u\|_{L^\infty} + t^{-\frac54}\|\p_x^{(\leq k)} u\|_{L^2}.
    \end{align*}
    \item[Case 2: $|v| \leq 2 t^{-\frac12 + \delta}$.] In this case, according to the localization scale, the integrand is nontrivial when $|x| \leq 3t^{\frac12 + \delta}$. For simplicity, we denote $\tilde\chi^v:= \tfrac{x-vt}{\sqrt{t}} \chi^v$.

    \item[Case 2.1: $k = 0$.]
    Recall that $c$ is defined by \eqref{eq:const_c}. We handle the case $k = 0$ separately first. \begin{align*}
        |R_{5, 0}| &\lesssim t^{-2}|\brk{\nabla_x U, \tilde\chi^v}| \lesssim t^{-2}|\brk{ U - c\varphi_0, \div_x \tilde\chi^v}| + |c|t^{-2}|\brk{\tilde\varphi_0, \div_x \tilde\chi^v}| + |c|t^{-2}|\brk{\eta_0, \div_x \tilde\chi^v}| + B.T.\\
        &\lesssim t^{-1}\|U - c\varphi_0\|_{L^\infty(r \aleq t^{\frac12 + \delta})} + |c| t^{-\frac52} \|\tilde\varphi_0\|_{L^{\frac{3+\delta'}{\delta'}}} \|(\div\tilde\chi^v)(\tfrac{x-vt}{\sqrt{t}})\|_{L^{\frac{3+\delta'}{3}}} + |c|t^{-2}\|\nabla \eta_0\|_{L^1} \|\tilde\chi^v\|_{L^\infty} + B.T.,
    \end{align*}
    where we turn $\varphi_0$ into $\tilde\varphi_0$ in the last step to exploit the spatial decay of $\tilde\varphi_0$. (In the special case $\calK = \0$, the last term is vanishing due to the divergence structure.)
    Choose $\delta = \frac12\eta$, $\delta_1 = \frac34\eta$ in Lemma~\ref{lem:near-zone-asymp-zero-order} and choose $\delta' = \frac{3\eta}{3-\eta}\in(0, 2\eta)$ in Corollary~\ref{cor:decay_varphi0}, we obtain \begin{align*}
        |R_{5, 0}| &\lesssim t^{-\frac54+\frac{\eta}{2}}\|\calL u\|_{L^2} + t^{\frac12-\frac{\eta}{4}}\|u\|_{L^\infty} + |c|t^{-1-\frac{\eta}{2}} + |c|t^{-2}.
    \end{align*}
    One can read off from \eqref{eq:const_c} that \[
        |c| \aleq \|U\|_{\ell^\infty H^{2,-\frac32}(r \aeq t^{\frac12+\dlt_1})} \aleq t^{\frac32}(t^{\frac14+\frac{\dlt_1}{2}}t^{-2}\|\calL u\|_{L^2} + \|u\|_{L^\infty}),
    \]
    where we apply \eqref{eq:est-U-wt-Sob-norm-aux} in the last step. Thus, \[
        |R_{5, 0}| \lesssim t^{-\frac54+\frac{\eta}{2}}\|\calL u\|_{L^2} + t^{\frac12-\frac{\eta}{4}}\|u\|_{L^\infty}.
    \]

    \item[Case 2.2: $k \geq 1$.]
    For general $k$, we have
    \begin{align*}
        |R_{5, k}| &\lesssim t^{-2}|\brk{\nabla_x U^{(k)}, \tilde\chi^v}| \lesssim t^{-2}|\brk{ U^{(k)} - c(\tfrac{iX}{2t} + \p)^k\varphi_0, \div_x \tilde\chi^v}| + |c|t^{-2}|\brk{(\tfrac{iX}{2t} + \p)^k\varphi_0, \div_x \tilde\chi^v}| \\
        &\lesssim t^{\frac12}\|\p_x^k (u - c t^{-\frac32} e^{i\bfPhi^\ob}\varphi_0)\|_{L^\infty(r \aleq t^{\frac12 + \delta})} + |c| t^{-3+\dlt} \|\p_x^{(\leq k)}\varphi_0\|_{L^\infty} \|(\div\tilde\chi^v)(\tfrac{x-vt}{\sqrt{t}})\|_{L^1} \\
        &\qquad\qquad\qquad\qquad\qquad\qquad\qquad\qquad\qquad + |c|t^{-\frac52}\|\p_x^{(1\leq \cdot\leq k)}\varphi_0\|_{L^{\frac{3+\delta'}{\delta'}}} \|(\div\tilde\chi^v)(\tfrac{x-vt}{\sqrt{t}})\|_{L^{\frac{3+\delta'}{3}}}.
    \end{align*}
    Taking $\dlt, \dlt_1, \dlt'$ to be exactly the same as the case $k = 0$ in Corollary~\ref{cor:near-zone-asymp-higher-order}, it follows that
    \[
    |R_{5, k}| \lesssim t^{-\frac54+\frac{\eta}{2}}\|\p_x^{(\leq k)}\calL u\|_{L^2} + t^{\frac12-\frac{\eta}{4}}\|\p_x^k u\|_{L^\infty}.
    \]
\end{enumerate}

\smallskip
Then we come back to discuss $R_6$.
Thanks to \eqref{eq:comm-L-k-der}, we have \begin{align*}
    [\p_x^k, a^{\mu\nu}L_\mu L_\nu] &= \brk{x}^{-1-2\eta} \p_x^{(\leq k-1)}L^{(2)} + O(1) \p_x^{(\leq k-1)} L^{(1)} + O(1) \p_x^{(\leq k-1)} \\
    &= t\brk{x}^{-1-2\eta} \p_x^{(\leq k)}L^{(1)} + O(1) \p_x^{(\leq k-1)} L^{(1)} + O(1) \p_x^{(\leq k-1)}.
\end{align*}
Thus, \begin{align*}
    |R_{6, k}| \lesssim t^{-2} |\brk{L \p_x^{(\leq k-1)} u, \Psi_v}| + t^{-\frac54}\|\p_x^{(\leq k-1)}u\|_{L^2} + t^{-1}|\brk{\brk{x}^{-1-2\eta} L \p_x^{(\leq k)} u, \Psi_v}|.
\end{align*}
The first term can be handled exactly the same way as for $R_{5, k}$ and has better decay. We now deal with the last term, which requires an observation of the separation of scales in three ways: \begin{align*}
    &t^{-1}|\brk{\brk{x}^{-1-2\eta} L \p_x^{(\leq k)} u, \Psi_v}| \aleq t^{-\frac32}|\brk{U^{(k)}, \brk{x}^{-1-2\eta} \div_x \chi^v + O(1)\brk{x}^{-2-2\eta} \chi^v}| \\
    \aleq&\,  t^{-2}\|U^{(k)}\|_{L^\infty} \|\brk{x}^{-1-2\eta}\|_{L^{\frac{3}{1+\eta}}} \|(\div\chi^v)(\tfrac{x-vt}{\sqrt{t}})\|_{L^{\frac{3}{2-\eta}}} \\
    &\qquad\qquad+ t^{-\frac32}\|U^{(k)}\|_{L^\infty} \|\brk{x}^{-2-2\eta}\|_{L^{\frac{3}{2+\eta}}} \|\chi^v(\tfrac{x-vt}{\sqrt{t}})\|_{L^{\frac{3}{1-\eta}}} \\
    \aleq&\,  t^{-1-\frac{\eta}{2}}\|U^{(k)}\|_{L^\infty} \aleq t^{\frac12-\frac{\eta}2}\|\p_x^{(\leq k)} u\|_{L^\infty}.
\end{align*}

Finally, the last term can be easily reduced to previous estimates: \begin{align*}
    |R_{7, k}| \lesssim t^{-\frac14}\|\brk{x}^{-1-2\eta} \p_x^{(\leq k-1)} u\|_{L^2(r \aleq t^{\frac12})} + t^{-\eta} \|\p_x^{(\leq k-1)} u\|_{L^\infty(r \ageq t^{\frac12})},
\end{align*}
where the right-hand side is weaker than the estimate for $|R_{2, k}|$.

Combining all the discussion above, \begin{align*}
    \left|\tfrac{d}{dt}\gmm[\p^{k} u](t, v) - R_{f, k}\right| \aleq t^{-\frac54 + \frac{\eta}{2}}\|\p_x^{(\leq k)} \calL u\|_{L^2} + t^{\frac12 - \frac{\eta}{4}}\|\p_x^{(\leq k)} u\|_{L^\infty} + t^{-\frac54}\|\p_x^{(\leq k)}u\|_{L^2},
\end{align*}
we complete the proof.

\section{Nonlinear applications}\label{sec-nonlinear}

\subsection{Proof of Main~Theorem~\ref{thm:nonlinear-1}}
Set $c_0 = \lceil\frac{8}{\eta_\sharp}\rceil$. Given $T_f > 0$, we assume that the following assumptions hold for $t \in [0, T_f)$:  \begin{align*}
    \|\p_x^{(\leq N)}u(t, \cdot)\|_{L^\infty} &\leq 2C \eps \brk{t}^{-\frac32}, \hbox{ for all } 0 \leq N \leq M_0 - 3 - c_0 s_c',\tag{BA1-NL1}\label{eq:BA1-NL1}\\
    \|\p_x^{(\leq N)} \calL u(t, \cdot)\|_{L^2} &\leq 2C \eps \brk{t}^{\frac14-\eta}, \hbox{ for all } 0 \leq N \leq M_0 - 3 - c_0 s_c', \tag{BA2-NL1}\label{eq:BA2-NL1}\\
    \|\p_x^{(\leq N)} u(t, \cdot)\|_{L^2} &\leq 2C \eps, \hbox{ for all } 0 \leq N \leq M_0 + 2. \tag{BA3-NL1}\label{eq:BA3-NL1}
\end{align*}

\subsubsection{Step 1: Verifying the assumption \eqref{eq:E-u} with $A_0 \aleq \eps$ under bootstrap assumptions, i.e., improving \eqref{eq:BA3-NL1}}
In this proof, we present in detail how the top order estimate is obtained, where the $N < M_0 + 2$ non-top order estimates could be handled using the same way, without the necessity of taking care of the cancellation in \eqref{eq:regroup-step1-thm3} to prevent any derivative losses. In fact, a more direct way to improve all the non-top order \eqref{eq:BA3-NL1} is the following: One could easily conclude the zeroth order improvement first by invoking \ref{hyp:en} and bootstrap assumptions. Then other non-top order estimates improvement can be obtained by interpolation. 

Now we focus on the top order $N = M_0 + 2 \in 2\bbZ_{\geq 0}$ by commuting the equation with $\p_t^{(\leq N/2)}$. We remark that one could directly commute coordinate derivatives when $\calK = \0$, where the use of $\p_t$ is crucial to the obstacle case since it preserves the Dirichlet boundary condition. To pass control of $\p_t^{(\leq \frac{N}{2})} u$ to that of $\p_x^{(\leq N)} u$, we introduce the following lemma.
\begin{lemma}\label{lem:pass-pt-to-px}
    Given a solution $u$ to \eqref{eq:nls} with \ref{hyp:nl-1} such that $u \in L^\infty H^{M_0 + 2}(\calM)$ and $u|_{[0, \infty) \times \p\calK} = 0$. Suppose that $u$ satisfies \eqref{eq:BA1-NL1}, \eqref{eq:BA2-NL1} and \eqref{eq:BA3-NL1}, then the following holds for $t \in [0, T_f)$ \begin{align*}
        \|\p_x^{(\leq N)} u(t, \cdot)\|_{L^2} \aleq \|\p_t^{(\leq \frac{N}{2})}u(t, \cdot)\|_{L^2} + \eps, \quad N \in 2\bbZ_{\geq 0}, \quad N \leq M_0 + 2.
    \end{align*}
\end{lemma}
\begin{proof}
    The case $N = 0$ is trivial. Suppose that it holds for $N - 2$ and we would like to prove this for $N$. It follows from Lemma~\ref{lem:unweight-ell} and bootstrap assumptions that \begin{align*}
        \|\p_x^{(\leq N)} u\|_{L^2} \aleq \|\p_x^{(\leq N - 2)}\calH u\|_{L^2} + \|u\|_{L^2} \aleq \|\p_x^{(\leq N - 2)}\p_t u\|_{L^2} + \|u\|_{L^2} + \eps.
    \end{align*}
    Iterating the previous step by applying Lemma~\ref{lem:unweight-ell} to $\p_t^k u$, we arrive at \[
        \|\p_x^{(\leq N)} u\|_{L^2} \aleq \|\p_t^{(\leq \frac{N}{2})} u\|_{L^2} + \|u\|_{L^2} + \eps.
    \]
    Finally, estimating $\|u\|_{L^2}$ by \ref{hyp:en} and bootstrap assumptions, we finish the proof.
\end{proof}

Now, to close the estimate for the top order derivative ($N = M_0 + 2 \in 2\bbZ_{\geq 0}$), we set $\td{\calH}:= \calH + D_\mu G^{\mu\nu} |u|^2 D_\nu$. We commute the equation with $\p_t^{\frac{N}{2}}$.
\begin{align*}
    (i\p_t - \td{\calH})\p_t^{\frac{N}{2}}u + G^{\mu\nu}D_\mu \br u D_\nu u \p_t^{\frac{N}{2}}u  = [i\p_t - \td\calH+ G^{\mu\nu}D_\mu \br u D_\nu u, \p_t^{\frac{N}{2}}]u + \p_t^{\frac{N}{2}} F.
\end{align*}
We expand the right-hand side further to exploit the structure:
\begin{align*}
    [-\td\calH, \p_t^{\frac{N}{2}}]u = D_\mu G^{\mu\nu} (\p_t^{\frac{N}{2}}u \br u + u \p_t^{\frac{N}{2}} \br u) D_\nu u
    &+ O\big((\p_t^{(\leq N/2-1)} D^{(\leq 1)}u) (\p_t^{(\leq \frac{M_0 - 5 - c_0 s_c'}{2})}D^{(\leq 1)}u)^2\big) \\
    &+ O\big((\p_t^{(1 \leq \cdot \leq \frac{N}{2})} D^{(\leq 1)} (a^{\mu\nu} + b^\mu) + \p_t^{(1 \leq \cdot \leq \frac{N}{2})}c) \p_t^{(\leq N/2-1)} D^{(\leq 2)} u\big) \\
    = G^{\mu\nu} (D_\mu\p_t^{\frac{N}{2}}u \br u + u D_\mu\p_t^{\frac{N}{2}} \br u) D_\nu u
    &+ O\big((D^{(\leq N - 1)}u) (D^{(\leq M_0 - 3 - c_0 s_c')}u)^2\big) \\
    &+ O\big((\p_t^{(1 \leq \cdot \leq \frac{N}{2})} D^{(\leq N-1)} (a^{\mu\nu} + b^\mu) + \p_t^{(1 \leq \cdot \leq \frac{N}{2})}D^{(\leq N-2)} c) D^{(\leq N)} u\big),
\end{align*}
where $\p_t$ is converted into spatial derivatives by the equation \eqref{eq:nls} in the last step.
Moving on to the last term in the commutator, we write \begin{align*}
    [G^{\mu\nu}D_\mu \br u D_\nu u, \p_t^{\frac{N}{2}}]u = -G^{\mu\nu}D_\mu \p_t^{\frac{N}{2}}\br u u D_\nu u &- G^{\mu\nu} u D_\mu \br{u} D_\nu \p_t^{\frac{N}{2}} u + O\big((D^{(\leq N)}u)(D^{(\leq M_0 - 3 - c_0 s_c')}u)^2\big) \\
    &+ O\big((\p_t^{(1 \leq \cdot \leq \frac{N}{2})} D^{(\leq N-1)} (a^{\mu\nu} + b^\mu) + \p_t^{(1 \leq \cdot \leq \frac{N}{2})}D^{(\leq N-2)} c) D^{(\leq N-1)} u\big).
\end{align*}
By writing the last term into the following form \begin{align*}
    \p_t^{\frac{N}{2}} F(D\br u, Du, \br u, u) =\, &\ul{B}^\mu(D\br u, Du, \br u, u) D_\mu\br{\p_t^{\frac{N}{2}}u} + B^\mu(D\br u, Du, \br u, u) D_\mu\p_t^{\frac{N}{2}} u
    + O\big((D^{(\leq N)}u)(D^{(\leq M_0 - 3 - c_0 s_c')}u)^2\big) \\&+ O\big((\p_t^{(1 \leq \cdot \leq \frac{N}{2})} D^{(\leq N-1)} (a^{\mu\nu} + b^\mu) + \p_t^{(1 \leq \cdot \leq \frac{N}{2})}D^{(\leq N-2)} c) D^{(\leq N-1)} u\big),
\end{align*}
where $B^\mu$ and $\ul B^\mu$ are defined in \ref{hyp:nl-1}.
We regroup all the terms and derive the following 
\begin{equation}\label{eq:regroup-step1-thm3}
    \begin{split}
        &(i\p_t - \td{\calH})\p_t^{\frac{N}{2}}u - 2 G^{\mu\nu} \Im(\bar u \p_\nu u) D_\mu \p_t^{\frac{N}{2}}u - \ul B^\mu D_\mu\br{\p_t^{\frac{N}{2}}u} - B^\mu D_\mu\p_t^{\frac{N}{2}}u \\
        =&\, O\big((D^{(\leq N)}u)(D^{(\leq M_0 - 3 - c_0 s_c')}u)^2\big) + O\big((\p_t^{(1 \leq \cdot \leq \frac{N}{2})} D^{(\leq N-1)} (a^{\mu\nu} + b^\mu) + \p_t^{(1 \leq \cdot \leq \frac{N}{2})}D^{(\leq N-2)} c) D^{(\leq N)} u\big),
    \end{split}
\end{equation}
where the cancellation involved in the term $- 2 G^{\mu\nu} \Im(\bar u \p_\nu u) D_\mu \p_t^{\frac{N}{2}}u$ is essential. The potentially linear error can be handled as follows. Under the assumptions \ref{hyp:nl-1} and \ref{hyp:ult-stat}, the energy estimate \ref{hyp:en} is applicable to the equation above in $\p_t^{\frac{N}{2}}u$ with $\td A^{\mu\nu} = G^{\mu\nu}|u|^2, \td B^{\mu} = -B^{\mu} - 2G^{\mu\nu} \Im(\bar u \p_\nu u), \ul{\td{B}}^{\mu} = -\ul{B}^{\mu}$. 

Thanks to \ref{hyp:ult-stat}, we choose $T_0$ sufficiently large so that the tail satisfies smallness condition \[
    \|\p_t^{(1 \leq \cdot \leq \frac{N}{2})} D^{(\leq N-1)} (a^{\mu\nu} + b^\mu) + \p_t^{(1 \leq \cdot \leq \frac{N}{2})}D^{(\leq N-2)} c\|_{L^1 L^\infty[T_0, \infty)} \leq \eps.
\]
Therefore, by applying \ref{hyp:en}, the nonlinear errors are easily handled by the bootstrap assumptions. Moreover, for the linear error, when $t \in [0, T_0]$, we have \[
    \|\p_t^{(\leq \frac{N}{2})} u(t, \cdot)\|_{L^2} \aleq \eps + \int_0^t \|\p_t^{(1 \leq \cdot \leq \frac{N}{2})} D^{(\leq N-1)} (a^{\mu\nu} + b^\mu) + \p_t^{(1 \leq \cdot \leq \frac{N}{2})}D^{(\leq N-2)} c\|_{L^\infty_x}(s) \|\p_t^{(\leq \frac{N}{2})}u\|_{L^2}\, \ud s
\]
and hence can be handled by Gronwall's inequality: \[
    \|\p_t^{(\leq \frac{N}{2})} u(t, \cdot)\|_{L^2} \aleq_{A_c, T_0} \eps.
\]
For $t \in [T_0, T_f)$, we could use the smallness of the tail to absorb the linear error. Thus, this finishes the proof, i.e., \eqref{eq:E-u} holds with $A_0 \aleq \eps$.

\subsubsection{Step 2: Base case and improving bootstrap assumptions when $t < T_0$}
Based on the assumptions, all the results in Section~\ref{sec:good-comm} still hold without modification with $f$ replaced by $f + \calN(u)$. We now establish the base case in the nonlinear setting. Note that, due to a loss of derivatives, the second estimate is only valid up to the order $M_0 - 1$, in contrast to the linear case.

\begin{lemma}[Base case for the iteration]\label{lem:base_iteration_nl1}
    Under the assumptions of Main~Theorem~\ref{thm:nonlinear-1} and bootstrap assumptions \eqref{eq:BA1-NL1}, we have \[
        \|\p_{x}^{(\leq M_0)} u(t, \cdot)\|_{L^\infty} \lesssim \eps, \quad \|\p_{x}^{(\leq M_0 - 1)} \calL u(t, \cdot)\|_{L^2} \lesssim \eps\brk{t}^2, \hbox{  for all } 0 < t < T_f.
    \]
\end{lemma}
\begin{proof}
    The first estimate follows from a simple Sobolev embedding \[
        \|\p_{x}^{(\leq N)}u\|_{L^\infty} \lesssim \|\p_{x}^{(\leq N+2)}u\|_{L^2} \aleq A_0 \aleq \eps, \hbox{ for all } 0 \leq N \leq M_0.
    \]
    Regarding the second estimate, we proceed as follows according to the structure of $\calL$.

    We commute $x$ into the equation and obtain \begin{align*}
       (i\p_t - \calH)x u = 2\nabla u +[x, D_\mu \bfh^{\mu\nu} D_\nu + b^\mu D_\mu + D_\mu b^\mu + c]u + x \calN(u).
    \end{align*}
    In view of the linear case, it suffices to estimate $\|\p_x^{(\leq N)} x\calN(u)\|_{L^2}$. We group $x$ with the entry without the highest order of derivatives and one could commute $\p_x^{(\leq N)}$ with $N \leq M_0$ into the equation and then apply Step 1. Therefore,
    it follows from \eqref{eq:BA1-NL1} that \begin{align*}
        \|\p_{x}^{(\leq M_0)}(xu)(t, \cdot)\|_{L^2} \aleq D_0 + A_0 \brk{t} + \eps \aleq \eps \brk{t},
    \end{align*}
    where we remark that the reason for discarding the treatment in Step~1 involving commutation with $\p_t^{\frac{N}{2}}$ is that $M_0+1 \notin 2\bbZ_{\geq 0}$, and for simplicity we choose to accept this derivative loss rather than pursuing a more delicate analysis.

    In addition, commuting $x$ into the equation again, it follows by handling $x^2\calN(u)$ term in the same way that \begin{align*}
        \|\p_{x}^{(\leq M_0-1)}(x^2u)(t, \cdot)\|_{L^2} \aleq \eps \brk{t}^2.
    \end{align*}
    In particular, the second estimate in the lemma follows readily.
\end{proof}

\subsubsection{Step 3: Nonlinear counterparts of Propositions~\ref{prop:est-calLu}~and~\ref{prop:dotgmm_decay}}
The goal of this step is to formulate and establish the nonlinear counterparts of Propositions~\ref{prop:est-calLu} and~\ref{prop:dotgmm_decay}. To prepare for the proof, we first introduce the following lemma.

\begin{lemma}\label{lem:L3L6-est}
    Assume that $\calH$ satisfies \ref{hyp:af-0}. For all $0 \leq N \leq M_c$, we have  \[
        \begin{split}
             t^{-1}\|L\p_x^{(\leq N)} u\|_{L^6} \aleq t^{-\frac74}\|\p_x^{(\leq N)}\calL u\|_{\ell^1 L^2} +  \|\p_x^{(\leq N)} u\|_{L^\infty} + t^{-\frac32}\|\p_x^{(\leq N)} u\|_{L^2}.
        \end{split}
    \]
\end{lemma}
\begin{proof}
    By applying standard Sobolev embeddings in dyadic regions, we arrive at \begin{align*}
        &t^{-1} \nrm{L^{(1)} \p_x^{(\leq N)} u}_{L^{6}(r \ageq t^{\frac12})} 
        \aleq t^{-2} \nrm{L^{(2)} \p_x^{(\leq N)} u}_{ L^{2}(r \ageq t^{\frac12})}
        + t^{-1} \nrm{r^{-1} L^{(1)} \p_x^{(\leq N)} u}_{ L^{2}(r \ageq t^{\frac12})},
    \end{align*}
    where the right-hand side can be handled using Proposition~\ref{prop:elliptic-int-rad}.

    For the region $\set{r \aleq t^{\frac12}}$, we start with the case $N = 0$.
    Applying standard Sobolev embeddings in dyadic regions as above,
    \begin{equation*}
        t^{-1} \nrm{r^{\frac{1}{2}} L^{(1)} u}_{\ell^{\infty} L^{6}(r \aleq t^{\frac12})} \aleq t^{-2} \nrm{r^{\frac{1}{2}} L^{(2)} u}_{\ell^{\infty} L^{2}(r \aleq t^{\frac12})}
        + t^{-1} \nrm{r^{-\frac{1}{2}} L^{(1)} u}_{\ell^{\infty} L^{2}(r \aleq t^{\frac12})}.
    \end{equation*}
    Thus, combining with Lemma~\ref{lem:elliptic-near}  finishes the proof for the case $N = 0$. To handle the commutator arising during the induction, we have \[
        t^{-1} \|\p_x^{(\leq N)} u\|_{\ell^\infty L^6} \aleq t^{-1}\left(\|\p_x^{(\leq N)} u\|_{L^\infty}^{\frac13}\|\p_x^{(\leq N)} u\|_{L^2}^{\frac23} + \|\p_x^{(\leq N)} u\|_{L^\infty}^{\frac23}\|\p_x^{(\leq N)} u\|_{L^2}^{\frac13}\right),
    \]
    which completes the proof.
\end{proof}

\begin{proposition}
    Under the assumptions of Main~Theorem~\ref{thm:nonlinear-1} and bootstrap assumptions \eqref{eq:BA1-NL1}, \eqref{eq:BA2-NL1}, there exists $T_0$ (independent of $T_f$) such that for all $T_0 + 1 < t < T_f$, for all $s_c' \leq N \leq M_0 - 3 \leq M_c - 3$, for any solution $u$ to \eqref{eq:nls}, \begin{equation*}
    \begin{aligned}
        \|\p_x^{(\leq N - s_c')} \calL u\|_{L^\infty L^2[T_0, t]} \aleq\, &\|\p_x^{(\leq N)} \calL u(T_0, \cdot)\|_{L^2} + \|t^{-\frac12 - \eta}\p_x^{(\leq N)} \calL u\|_{L^2L^2[T_0, t]} \\
        &+ \|t^{-\frac14 - \eta}\|t^{\frac32}\p_x^{(\leq N)} u\|_{L^\infty_x}\|_{\ell^1 L^2_t[T_0, t]} + \|t^{-1 - \eta} \|\p_x^{(\leq N)}\calL u\|_{L^2_x}\|_{L^1[T_0, t]} \\
        &+ \|t^{-1 - 2\eta}\|\p_x^{(\leq N)} u\|_{L^2_x}\|_{L^1_t[T_0, t]} + \eps \log t.
    \end{aligned}
\end{equation*}
\end{proposition}
\begin{proof}
    Based on Proposition~\ref{prop:est-calLu} and bootstrap assumptions, it suffices to establish that \begin{align*}
        \sum_{2i + j \leq N}\|\p_t^{(\leq i)}\p_x^{(\leq j)}\calL \calN(u)\|_{\brk{x}^{-1}L^2L^2+L^1L^2[T_0, t]} \aleq \eps \log t.
    \end{align*}
    In view of the bootstrap assumptions and the equation, $\p_t$ can be replaced by $\calH$ at the expense of more favorable error terms. Hence, it suffices to establish that
    \begin{align}\label{eq:est-logt-nl1}
        \|\p_x^{(\leq N)}\calL \calN(u)\|_{\brk{x}^{-1}L^2L^2+L^1L^2[T_0, t]} \aleq \eps \log t, \hbox{ for all } 0 \leq  N \leq M_0 - 3,
    \end{align}
    which will be proved in the sequel.

    Under the assumption \ref{hyp:nl-1}, we first write
    \begin{align*}
        L \calN(u) = &\, [L, D_{\mu}] (G^{\mu\nu}|u|^2 D_{\nu} u) + D_{\mu} G^{\mu \nu} (\br u Lu - u \br{Lu}) D_{\nu} u + D_{\mu} G^{\mu\nu}|u|^2 LD_{\nu} u \\
        &- G^{\mu\nu}Lu D_\mu \br u D_\nu u - G^{\mu\nu} u \br{LD_\mu u} D_\nu u - G^{\mu\nu}u D_\mu \br u LD_\nu u + LF.
    \end{align*}
    Therefore,
    \begin{align*}
        \|\p_x^{(\leq N)} (\calL - a^{\mu\nu} L_\mu L_\nu) \calN(u)\|_{L^2} \aleq &\, t^2\|\p_x^{(\leq N)}u\|_{L^2} \|\p_x^{(\leq M_0 - 3 - c_0 s_c')}u\|_{L^\infty}^2 + t\|r^{-1}\p_x^{(\leq N+1)}L u\|_{L^2} \|\p_x^{(\leq M_0 - 3 - c_0 s_c')}u\|_{L^\infty}^2 \\
        &+ t\|r^{-1}\p_x^{(\leq M_0 - 3 - c_0 s_c')} Lu\|_{L^2} \|\p_x^{(\leq N+1)}u\|_{L^\infty} \|\p_x^{(\leq M_0 - 3 - c_0 s_c')} u\|_{L^\infty}.
    \end{align*}
    In view of Step 1 and the form of $L$, we obtain \begin{equation}\label{eq:nl-1st-order-est1}
        \|r^{-1}\p_x^{(\leq N+1)}Lu\|_{L^2} \aleq \|\p_x^{(\leq N+1)} u\|_{L^2} + t\|\p_x^{(\leq N+2)}u\|_{L^2} \aleq C\eps t.
    \end{equation}
    On the other hand, Propositions~\ref{prop:elliptic-near},~\ref{prop:elliptic-int}~and~\ref{prop:elliptic-rad} imply that, for $N_0 = M_0 - 3 - c_0 s_c'$, \begin{equation}\label{eq:nl-1st-order-est2}
        \|r^{-1}\p_x^{(\leq N_0)}L u\|_{L^2} \aleq t^{-\frac34}\|\p_x^{(\leq N_0)}\calL u\|_{L^2} + t\|\p_x^{(\leq N_0)}u\|_{L^\infty} + t^{-1}\|\p_x^{(\leq N_0)}u\|_{L^2} \aleq C\eps t^{-\frac12},
    \end{equation}
    where $C$ is the constant in bootstrap assumptions.
    Combining these preceding estimates, we have \begin{align*}
        \|\p_x^{(\leq N)} (\calL - a^{\mu\nu}L_\mu L_\nu) \calN(u)\|_{L^2} \aleq \eps t^{-1}, \hbox{ for all } 0 \leq N \leq M_0 - 1 \leq M_c - 1.
    \end{align*}

    Using the covariant trilinear structure of $\calN(u)$, we arrive at \begin{align*}
        &\|\p_x^{(\leq N)} (a^{\mu\nu}L_\mu L_\nu) \calN(u)\|_{L^2} \aleq \sum_{k = 0}^2\|L^{(2-k)} \p_x^{(\leq N-k)}\calN(u)\|_{L^2}\\
        \aleq&\, \|L^{(2)} \p_x^{(\leq N + 2)}u\|_{L^2}\|\p_x^{(\leq M_0 - 3 - c_0s_c')} u\|_{L^\infty}^2 + \|\p_x^{(\leq N + 2)}u\|_{L^\infty}\|L^{(2)}\p_x^{(\leq M_0 - 3 - c_0s_c')} u\|_{L^2}\|\p_x^{(\leq M_0 - 3 - c_0s_c')} u\|_{L^\infty} \\
        &+ \|L\p_x^{(\leq N+2)} u\|_{L^3} \|L\p_x^{(\leq M_0 - 3 - c_0s_c')}u\|_{L^6} \|\p_x^{(\leq M_0 - 3 - c_0s_c')}u\|_{L^\infty} + \|L\p_x^{(\leq M_0 - 3 - c_0s_c')}u\|_{L^6}^2\|\p_x^{(\leq N+2)} u\|_{L^6} \\
        \aleq&\, \eps t^{-3}\|L^{(2)} \p_x^{(\leq N + 2)}u\|_{L^2} + \eps t^{-\frac32} \|L^{(2)}\p_x^{(\leq M_0 - 3 - c_0s_c')} u\|_{L^2} \\
        &+ \eps t^{-2}(t^{-\frac12}\|L^{(2)} \p_x^{(\leq N+2)} u\|_{L^2}^{\frac12} \|L \p_x^{(\leq N+2)} u\|_{L^2}^{\frac12}) + \eps t^{-1} \hbox{ for all } 0 \leq N \leq M_0  - 2,
    \end{align*}
    where we apply \eqref{eq:BA1-NL1}, Lemma~\ref{lem:base_iteration_nl1} and Lemma~\ref{lem:L3L6-est} in the last inequality. Furthermore, for $N_0 = N + 2$ (with $N \leq M_0 - 3$) or $N_0 = M_0 - 3 - c_0 s_c'$, we have \begin{align*}
        &\|L^{(2)} \p_x^{(\leq N_0)}u\|_{L^2} \aleq \|a^{\mu\nu} L_\mu L_\nu \p_x^{(\leq N_0)}u\|_{L^2} \\
        \aleq&\,  \|\p_x^{(\leq N_0)}\calL u\|_{L^2} + \|r^{-1}\p_x^{(\leq N_0)} Lu\|_{L^2} + \|\p_x^{(\leq N_0)} u\|_{L^2} + \|L\p_x^{(N_0 - 1)}u\|_{L^2},
    \end{align*}
    where we invoke \eqref{eq:nl-1st-order-est1} or \eqref{eq:nl-1st-order-est2} to treat the second term and we use interpolation to treat the last term $\|L\p_x^{(N_0 - 1)}u\|_{L^2} \aleq \eps \|L^2\p_x^{(N_0 - 1)}u\|_{L^2} + \eps^{-1}\|\p_x^{(N_0 - 1)}u\|_{L^2}$.

    Finally, combining everything, $\|\p_x^{(\leq N)} \calL \calN(u)\|_{L^1L^2} \aleq \eps \log t$ for $N \leq M_0 - 3$.
\end{proof}

\begin{proposition}
    Under the assumptions of Main~Theorem~\ref{thm:nonlinear-1} and bootstrap assumptions \eqref{eq:BA1-NL1}, \eqref{eq:BA2-NL1}, then there exists $T_0$ (independent of $T_f$) such that for all $T_0 < t < T_f$, the time derivative of $\gmm$ satisfies the following bound uniformly in $v$ for any $0 \leq N \leq M_0$: \begin{align*}
        \|\tfrac{d}{dt}\gmm[\p^N u](t, v)\|_{L^\infty_v} \lesssim t^{-\frac54 + \frac{\eta}2} \|\p_x^{(\leq N)}\calL u(t, \cdot)\|_{L^2} + t^{\frac12-\frac{\eta}{4}}\|\p_x^{(\leq N)} u(t, \cdot)\|_{L^\infty} + \eps t^{-\frac94}.
    \end{align*}
\end{proposition}
\begin{proof}
    Using \eqref{eq:BA1-NL1} and \eqref{eq:BA3-NL1}, we have \begin{align*}
    \left|\brk{\p_x^{(\leq N)} \calN(u), \Psi_v}\right| \aleq t^{\frac34}\|\p_x^{(\leq N)} \calN(u)\|_{L^2} \aleq t^{\frac34} \|\p_x^{(\leq N+2)} u\|_{L^2} \|\p_x^{(\leq M_0 - 3 - c_0s_c')}u\|_{L^\infty}^2 \aleq \eps t^{-\frac94}.
\end{align*}
\end{proof}

\subsubsection{Step 4: Closing the bootstrap}
From the previous steps, we could run the main iteration lemma (Lemma~\ref{lem:main_iteration}) to derive that \eqref{eq:BA1-NL1} and \eqref{eq:BA2-NL1} are open conditions. In addition, Steps~1-2 reveal that \eqref{eq:BA2-NL1} holds on a nonempty set $[0, t_0]$ for some $t_0$ provided that \eqref{eq:BA1-NL1} holds true, which in turn reduces to \eqref{eq:BA3-NL1} when $t_0 \ll 1$.

Finally, we require local well-posedness in $H^{M_0 + 2}$ in order to establish the openness of the bootstrap assumptions. This is a standard result in the case $\calK = \0$ (see, for instance, \cite{MMT14}), and we will derive the corresponding statement for the general setting below. With this ingredient in place, we complete the proof of global existence and sharp decay. The scattering result then follows by replicating the corresponding steps in the proof of Main~Theorem~\ref{thm:linear} in Section~\ref{subsec-main-iter}. Therefore, we finish the proof by establishing the local well-posedness as follows.

\begin{proposition}
    Under the assumptions \ref{hyp:en} and \ref{hyp:nl-1},
    the equation \eqref{eq:nls} is wellposed in $H^{M_0 + 2}$.
\end{proposition}
\begin{proof}
    Set \begin{align*}
    F^1(D\br u, Du, \br u, u, \p_t\br u, \p_t u):= \p_t F(D\br u, Du, \br u, u) - \ul{B}^\mu(D\br u, Du, \br u, u) D_\mu \br{\p_t u} - B^\mu(D\br u, Du, \br u, u) D_\mu \p_t u.
    \end{align*}
    We perform an iteration argument. To initiate, we set $u^0 = v^0 = 0$. Suppose we are given $(u^j, v^j)_{j \leq n}$ satisfying the following on $[0, T]$ with $T$ sufficiently small to be determined \begin{align*}
        \p_t u^j = v^j, &\quad u^j|_{t = 0} = u_0, \quad v^j|_{\p\calK} = 0, \quad \|u^j\|_{H^{M_0}} + \|v^j\|_{H^{M_0}} \leq \eps,
    \end{align*}
    where we remark that the smallness assumption on $\eps$ is not necessary for this part.
    Now we aim to construct $(u^{n+1}, v^{n+1})$ with the same properties as above. We consider \begin{equation}\label{eq:vn-lin}
        \begin{split}
        (i\p_t - \calH)v^{n+1} =\, & \ul{B}^\mu(D\br u^n, Du^n, \br u^n, u^n) D_\mu \br{v^{n+1}} - B^\mu(D\br u^n, Du^n, \br u^n, u^n) D_\mu v^{n+1} \\
        &+ D_\mu G^{\mu\nu} |u^n|^2 D_\nu v^{n+1} - G^{\mu\nu} v^{n+1} D_\mu \br u^n  D_\nu u^n - [\calH, \p_t] u^{n} \\
        &- [D_\mu G^{\mu\nu} |u^n|^2 D_\nu - G^{\mu\nu} D_\mu \br u^n  D_\nu u^n, \p_t] u^n + F^1(D\br u^n, Du^n, \br u^n, u^n, \p_t\br u^n, \p_t u^n)
        \end{split}
    \end{equation}
    with initial-boundary condition \[
    iv^{n+1}(t = 0) = \calH|_{t = 0} u_0 + D_\mu G^{\mu\nu} |u_0|^2 D_\nu u_0 - G^{\mu\nu} u_0 D_\mu \br u_0 D_\nu u_0 + F(D\br u_0, D u_0, \br u_0, u_0), \quad v^{n+1}|_{\p\calK} \equiv 0.
    \]
    This equation is linear in $v^{n+1}$, which hence ensures the existence of such $v^{n+1}$ with $\|v^{n+1}\|_{L^2} \leq \eps$. Moreover,  one can obtain control of $\|\p_t^{(\leq \frac{M_0}{2})} v^{n+1}\|_{L^2}$ by commuting the equation with $\p_t$'s.
    In view of the equation, following the idea of Lemma~\ref{lem:pass-pt-to-px}, one could derive an a priori estimate and hence conclude that  \begin{align*}
       \|v^{n+1}\|_{H^{M_0}} \leq \eps.
    \end{align*}
    Now we solve $u^{n+1}$ from \[
        \p_t u^{n+1} = v^{n+1}, \quad u^{n+1}|_{t = 0} = u_0,
    \]
    which satisfies $\|u^{n+1}\|_{H^{M_0}} \leq \eps$.

    Now we want to show that $(v^n, \p_t v^n)$ is Cauchy in $L^2$. To establish this, we consider the equation satisfied by $v^{n+1} - v^n$. Thanks to the energy estimate and an elliptic estimate as in the proof of Lemma~\ref{lem:unweight-ell}, we have \begin{align*}
        &\peq \|v^{n+1} - v^n\|_{L^\infty L^2[0, T]} \aleq \|\p_x^{(\leq 2)} (u^n - u^{n-1})\|_{L^1L^2[0, T]} \\
        &\aleq \|\p_t^{(\leq 1)} (v^n - v^{n-1})\|_{L^1L^2[0, T]} + \|(i\p_t - \calH)(v^n - v^{n-1})\|_{L^1 L^2[0, T]} \aleq \sum_{j = 1}^{n} T^j\|\p_t^{(\leq 1)} (v^{n-j+1} - v^{n-j})\|_{L^\infty L^2[0, T]}.
    \end{align*}
    Moreover, commuting $\p_t$ with the linear equation \eqref{eq:vn-lin} and  converting each $\p_t u^j$ to $v^j$, one would reach the same type of estimate \[
    \|\p_t v^{n+1} - \p_t v^n\|_{L^\infty L^2[0, T]} \aleq \sum_{j = 1}^{n} T^j\|\p_t^{(\leq 1)} (v^{n-j+1} - v^{n-j})\|_{L^\infty L^2[0, T]}.
    \]
    The recursion gives \[
    \|\p_t^{(\leq 1)} (v^{n+1} - v^n)\|_{L^\infty L^2[0, T]} \aleq 2^{n-1} T^{n} \|\p_t^{(\leq 1)} (v^{1} - v^0)\|_{L^\infty L^2[0, T]}
    \]
    and hence by choosing $T$ sufficiently small, we know that $\p_t^{(\leq 1)} v^n$ is Cauchy in $L^2$.
    Thus, there exists a limit $v:= \lim_{n \to \infty} v^n \in H^{M_0}$ and $\p_t u = v \in H^{M_0}$. From the equation of $v$, we read off \[
        \p_t ((i\p_t - \calH)u) = \p_t (D_\mu G^{\mu\nu}|u|^2 D_\nu u - G^{\mu\nu} u D_\mu \br u D_\nu u + F).
    \]
    Since $iv|_{t = 0} = (\calH u + (D_\mu G^{\mu\nu}|u|^2 D_\nu u - G^{\mu\nu} u D_\mu \br u D_\nu u + F))|_{t = 0}$, we know that $u$ is indeed a solution to \eqref{eq:nls}. This then allows us to upgrade the regularity to $H^{M_0 + 2}$ by repeating the argument in Step 1.
\end{proof}

\subsection{Proof of Main~Theorem~\ref{thm:nonlinear-2}}
We first review the basics of Lorentz spaces, to be used below.

\subsubsection{Lorentz spaces}
For completeness, we record a direct proof of the interpolation of Lorentz spaces without referring to the real interpolation theory.
\begin{lemma}
    For $\theta \in (0,1)$, $p_1 < p < p_2$ and \[
        \frac1{p} = \frac{\theta}{p_1} + \frac{1-\theta}{p_2},
    \]
    we have \[
        \|g\|_{L^{p,q}} \lesssim \|g\|_{L^{p_1, \infty}}^{\theta} \|g\|_{L^{p_2,\infty}}^{1-\theta}.
    \]
\end{lemma}
\begin{proof}
    Let $d(\lambda) = \mu(|g| \geq \lambda)$. We start with the case $p_2 < \infty$ and hence $d(\lambda)$ satisfies \[
        d(\lambda) \leq \min \left(\frac{\|g\|_{L^{p_1,\infty}}^{p_1}}{\lambda^{p_1}}, \frac{\|g\|_{L^{p_2,\infty}}^{p_2}}{\lambda^{p_2}} \right).
    \]
    It then follows that \[\begin{aligned}
        \|g\|_{L^{p,q}}^q
        &= \int_0^\infty \lambda^{q} d(\lambda)^{\frac{q}{p}}\frac{\, \ud\lambda}{\lambda}
        \leq \int_0^B \lambda^{q(1 -\tfrac{p_1}{p})-1} \|g\|_{L^{p_1, \infty}}^{q\tfrac{p_1}{p}} \, \ud\lambda + \int_B^\infty \lambda^{q(1 -\tfrac{p_2}{p})-1} \|g\|_{L^{p_2, \infty}}^{q\tfrac{p_2}{p}} \, \ud\lambda \\&\lesssim B^{q(1 -\tfrac{p_1}{p})}\|g\|_{L^{p_1,\infty}}^{q\tfrac{p_1}{p}} + B^{q(1 -\tfrac{p_2}{p})}\|g\|_{L^{p_2,\infty}}^{q\tfrac{p_2}{p}}.
    \end{aligned}
    \]
    Optimizing the right-hand side by choosing $B = \|g\|_{L^{p_2, \infty}}^{\frac{p_2}{p_2 - p_1}} / \|g\|_{L^{p_1, \infty}}^{\frac{p_1}{p_2 - p_1}}$, we finish the proof of the lemma when $p_2 \neq \infty$.

    The case $p_2 = \infty$ is easier since $d(\lambda) = 0$ when $\lambda > \|g\|_{L^\infty}$.
\end{proof}

\begin{lemma}[{\cite[Proposition 1.4.10]{Gr08}}]
    Suppose $0 < p \leq \infty$ and $0 < q < r \leq \infty$. Then we have $L^{p, q} \subset L^{p, r}$ and \[
        \|f\|_{L^{p, r}} \lesssim_{p, q, r} \|f\|_{L^{p, q}} .
    \]
\end{lemma}

\begin{corollary}[Interpolation of Lorentz spaces]\label{cor:Lorentz_interpolation}
    Given $\theta \in (0,1)$, $p_1 < p < p_2$,  and \[
        \frac1{p} = \frac{\theta}{p_1} + \frac{1-\theta}{p_2}.
    \]
    For any $0 < q_1, q_2 \leq \infty$, we have \[
        \|g\|_{L^{p,q}} \lesssim \|g\|_{L^{p_1, q_1}}^{\theta} \|g\|_{L^{p_2, q_2}}^{1-\theta}.
    \]
\end{corollary}

\begin{lemma}[H\"{o}lder's inequality]{\cite[Theorem 2.6]{ON63}}\label{lem:Holder_Lorentz}
    Given \[
        \frac1{p} = \frac{1}{p_1} + \frac{1}{p_2}, \quad \frac1{q} = \frac{1}{q_1} + \frac{1}{q_2},
    \]
    we have \[
        \|gh\|_{L^{p,q}} \lesssim \|g\|_{L^{p_1, q_1}} \|h\|_{L^{p_2,q_2}}.
    \]
\end{lemma}
For more classical results on Lorentz spaces, see for instance \cite{Gr08, Tao17}.

\subsubsection{Preliminaries}
Next, we present several preliminary results pertaining to the Hartree model.

\begin{lemma}[An $L^2$-based difference bound]\label{lemma_L2_control_Uminusgmm}
    The $L^2$ distance between $U$ and $\gmm$ can be characterized by:  \[
    \|U(t, vt) - \gmm(t, v)\|_{L^2_v} \lesssim t^{-1} \|\nabla_x U\|_{L^2_x} \lesssim t^{-\frac{1}{2}} \|\calL u(t,\cdot)\|_{L^2} + t \|u(t,\cdot)\|_{L^\infty} + t^{-\frac{1}{2}} \|u(t,\cdot)\|_{L^2}.
\]
\end{lemma}
\begin{proof}
    Recall that \[
        U(t, vt) - \gmm(t, v) = \int \big(U(t, tv) - U(t, t(v-y))\big)t^{\frac32} \chi(\sqrt{t}y)\, \ud y,
    \]
    where the integrand can be expanded in terms of \[
        U(t, tv) - U(t, t(v-y)) = \int_0^1 \nabla_x U(t, t(v-y + s y)) \cdot t y\, \ud s.
    \]
    Hence, \[
    \|U(t, vt) - \gmm(t, v)\|_{L^2_v} \lesssim t^{\frac12}\|\nabla_x U(t, tv)\|_{L^2_v} \lesssim t^{-1} \|\nabla_x U\|_{L^2_x}.
    \]
    Moreover, it follows from \eqref{eq:unit-scale-ell-U} and \eqref{eq:ell-r>t-U} (with $R = t$, and passing to the variable $U$) that \[
        \|\nabla_x U\|_{L^2_x} \lesssim \|r\calH U\|_{L^2(r\aleq t)} + t^{\frac12}\|U\|_{L^\infty} + t\|\calH U\|_{L^2(r \ageq t)} + t^{-1} \|U\|_{L^2},
    \]
    which concludes the proof.
\end{proof}

\begin{lemma}[Estimates on Hartree-type nonlinearity]\label{lem:est-Hartree-Nu}
    Given any $g, h \in \scrS(\R^3_x)$, we have \begin{align*}
        \|\tfrac{x}{|x|^3} * (gh)\|_{L^3} + \||x|^{-1}*(gh)\|_{L^\infty} &\aleq \|g\|_{L^2}^{\frac23}\|g\|_{L^\infty}^{\frac13}\|h\|_{L^2}^{\frac23}\|h\|_{L^\infty}^{\frac13}.
    \end{align*}
\end{lemma}
\begin{proof}
    Since $|x|^{-1} \in L^{3, \infty}$, applying Lemma~\ref{lem:Holder_Lorentz} and Corollary~\ref{cor:Lorentz_interpolation} gives \[
       \||x|^{-1}*(gh)\|_{L^\infty} \aleq \|gh\|_{L^{\frac32, 1}} \aleq \|g\|_{L^{3,2}}\|h\|_{L^{3,2}} \aleq \|g\|_{L^2}^{\frac23}\|g\|_{L^\infty}^{\frac13}\|h\|_{L^2}^{\frac23}\|h\|_{L^\infty}^{\frac13}.
    \]
    Similarly, \[
        \|\tfrac{x}{|x|^3} * (gh)\|_{L^3} \aleq \|gh\|_{L^{\frac32, 1}}.
    \]
\end{proof}

\begin{lemma}[Conservation of mass]\label{lem:L2-conservation-Hartree}
    Under the assumptions of Main~Theorem~\ref{thm:nonlinear-2}, given any solution $u \in \scrS$ to \eqref{eq:nls}, we have \[
        \|u\|_{L^2} = \|u_0\|_{L^2} \aleq \eps.
    \]
\end{lemma}
\begin{proof}
    The conservation of mass is standard.
\end{proof}

\subsubsection{Main proof}
Without loss of generality, we assume that $0 < \eta \leq \frac{1}{10}$; this will simplify the numerology later on (see, e.g., the proof of Proposition~\ref{prop:hartree-gamma-evol}).

The main proof relies on the following bootstrap assumptions, which are assumed  to hold for $t \in [0, T_f)$. \begin{align*}
    \|u\|_{L^\infty} &\leq 2C \eps \brk{t}^{-\frac32}, \tag{BA1-NL2}\label{eq:BA1-NL2}\\
    \|\calL u\|_{L^2} &\leq 2C \eps \brk{t}^{\frac14-\eta}. \tag{BA2-NL2}\label{eq:BA2-NL2}
\end{align*}

We first comment on the absence of an analogue of \eqref{eq:BA3-NL1}. In the setting of Main~Theorem~\ref{thm:nonlinear-1}, this assumption was specifically introduced as a substitute for \eqref{eq:E-u} in the linear case, which had been employed in the proof of the base iteration lemma (Lemma~\ref{lem:base_iteration}). In the present setting, one has access to an energy estimate without loss of derivatives easily by commuting $\p_x$ into the equation; see the proof of Lemmas~\ref{lem:LWP-Hartree} and \ref{lem:base-iter-main-thm-2} .

Before proceeding to the proof of Main~Theorem~\ref{thm:nonlinear-2},
we state the nonlinear counterparts of Propositions~\ref{prop:est-calLu}~and~\ref{prop:dotgmm_decay} under the assumption \ref{hyp:nl-2}.

\begin{proposition}\label{prop:est-LN(u)-Hartree}
    Under the assumptions of Main~Theorem~\ref{thm:nonlinear-2} and bootstrap assumptions \eqref{eq:BA1-NL2}, \eqref{eq:BA2-NL2}, there exists $T_0$ (independent of $T_f$) such that for all $T_0 < t < T_f$, for any solution $u$ to \eqref{eq:nls}, \begin{equation*}
    \begin{aligned}
        \|\calL u\|_{L^\infty L^2[T_0, t]} \aleq\, &\|\calL u(T_0, \cdot)\|_{L^2} + \|t^{-\frac14 - \eta}\|t^{\frac32} u\|_{L^\infty_x}\|_{L^2_t[T_0, t]} + \eps t^{\frac14 - \eta}.
    \end{aligned}
\end{equation*}
\end{proposition}
\begin{proof}
    It follows from Proposition~\ref{prop:est-calLu} and the corresponding absorption process (when $s_c' = 0$) stated in Remark~\ref{rem:est-calLu}, for $T_0$ sufficiently large, \begin{align*}
        \|\calL u\|_{L^\infty L^2[T_0, t]} \aleq\, &\|\calL u(T_0, \cdot)\|_{L^2} + \|t^{-\frac14 - \eta}\|t^{\frac32} u\|_{L^\infty_x}\|_{L^2_t[T_0, t]} \\
        &+ \|t^{-1 - 2\eta}\| u\|_{L^2_x}\|_{L^2_t[T_0, t]} +  \|\calL \calN(u)\|_{\brk{x}^{-1}L^2L^2+L^1L^2[T_0, t]}.
    \end{align*}
    Under our assumption $\calH = -\Delta + c$, using the fact that $|x|^{-1}$ is a multiple of the fundamental solution to $-\Delta$, and Lemma~\ref{lem:est-Hartree-Nu}, we have \begin{equation}\label{eq:est-calL-Nu-Hartree}
        \begin{aligned}
            &\peq \|\calL\calN(u)\|_{L^2_x} = 4t^{-\frac52}\|(\Delta - c)\big((|\cdot|^{-1}*|U(t, \cdot)|^2)(x) U(t, x) \big)\|_{L^2_x} \\
            &\aleq t^{-\frac52}\big(\|(|\cdot|^{-1}*|U(t, \cdot)|^2)(x)\|_{L^\infty}\|\calH U\|_{L^2} 
            %%+ \|\brk{x}^{-2-2\eta}(|\cdot|^{-1}*|U(t, \cdot)|^2)(x) U\|_{L^2}  
            + \|U\|_{L^\infty}^2 \|U\|_{L^2} \\ 
            &\peq + \|\nabla(|x|^{-1}*|U(t, \cdot)|^2)(x)\|_{L^6} \|\nabla U\|_{L^3(r \aleq t)} + \|\nabla(|x|^{-1}*|U(t, \cdot)|^2)(x)\|_{L^3} \|\nabla U\|_{L^6(r \ageq t)} \big) \\
            &\aleq \|\calL u\|_{L^2}\|u\|_{L^2}^{\frac43}\|u\|_{L^\infty}^{\frac23} 
            %%+ t^{-\frac52}\|\brk{x}^{-2-2\eta}(|\cdot|^{-1}*|U(t, \cdot)|^2)(x) U\|_{L^2} 
            + t^2\|u\|_{L^\infty}^2\|u\|_{L^2} \\
            &\peq + t^{\frac12}\|u\|_{L^\infty} \|u\|_{L^2} \|\nabla U\|_{L^3(r \aleq t)} + t^{\frac12}\|u\|_{L^2}^{\frac43}\|u\|_{L^\infty}^{\frac23} \|\nabla U\|_{L^6(r \ageq t)}.
        \end{aligned}
    \end{equation}
    %% The second term is redundant after putting in the cross terms.
    % In view of Lemma~\ref{lem:est-Hartree-Nu}, we estimate the second term on the ``intermediate- \& far-away zones'' (slightly enlarged compared to Remark~\ref{rem:zones}) as follows:\begin{align*}
    %     t^{-\frac52}\|\brk{x}^{-2-2\eta}(|\cdot|^{-1}*|U(t, \cdot)|^2)(x) U\|_{L^2(r \ageq t^{\frac12 - 2\dlt})} &\aleq t^{-\frac52} \|r^{-2-2\eta} U\|_{L^2(r \ageq t^{\frac12 - 2\dlt})} \|U\|_{L^2}^{\frac43} \|U\|_{L^\infty}^{\frac23} \\
    %     &\aleq t^2 \|u\|_{L^2}^{\frac43}\|u\|_{L^\infty}^{\frac53} \|r^{-2-2\eta}\|_{L^2(r \ageq t^{\frac12 - 2\dlt})} \aleq \eps t^{-\frac34 -(\eta - \dlt(1 + 4\eta))}.
    % \end{align*}
    % For the ``near zone'', we apply the weighted Sobolev estimate for free Laplacian. Thus, \begin{align*}
    %     &t^{-\frac52}\|\brk{x}^{-2-2\eta}(|\cdot|^{-1}*|U(t, \cdot)|^2)(x) U\|_{L^2(r \aleq t^{\frac12 - 2\dlt})} \\ \aleq&\,  t^{-\frac52}\|U\|_{L^\infty} \big(\|r^{\frac12 + \dlt} |U|^2\|_{L^2(r \aleq t^{\frac12 - 2\dlt})}  + \|r^{-\frac32+\dlt}(|\cdot|^{-1}*|U(t, \cdot)|^2)(x)\|_{L^\infty(r \sim t^{\frac12 -2\dlt})}\big)  \\
    %     \aleq&\, t^{\frac94-\frac{\dlt}{2}}\|u\|_{L^2}\|u\|_{L^\infty}^2 + t^2 t^{-(\frac32 - \dlt)(\frac12 - 2\dlt)}\|u\|_{L^\infty}^{\frac53} \|u\|_{L^2}^{\frac43} \aleq \eps t^{-\frac34-\frac{\dlt}{2}}.
    % \end{align*}
    The last two terms in \eqref{eq:est-calL-Nu-Hartree} can be handled by Sobolev estimates and weighted elliptic estimates. Indeed, on each dyadic region: \begin{align*}
        &\|\nabla U\|_{L^3(R < r < 2R)} \aleq \|\nabla U\|_{L^2(r \aeq R)}^{\frac12} (\|\nabla^2 U\|_{L^2(r \aeq R)} + R^{-1}\|\nabla U\|_{L^2(r \aeq R)})^{\frac12} \aleq R^{\frac12}( \|\calH U\|_{L^2(r \aeq R)} + R^{-2} \|U\|_{L^2(r \aeq R)}), \\
        &\|\nabla U\|_{L^6(r > R)} \aleq \|\Dlt U\|_{L^2(r \ageq R)} + R^{-2}\|U\|_{L^2(r \ageq R)} \aleq \|\calH U\|_{L^2(r \ageq R)} + R^{-2} \|U\|_{L^2(r \ageq R)}
    \end{align*}
    and hence \[
        \|\nabla U\|_{L^3(r \aleq t)} \aleq \|\calL u\|_{L^2} + t^{\frac32} (\log t)\|u\|_{L^\infty}, \quad \|\nabla U\|_{L^6(r > t)} \aleq t^{-\frac12}\|\calL u\|_{L^2} + t^{-\frac12} \|u\|_{L^2},
    \]
    where in the first summation we use $\ell^2 \subset \ell^3 \subset \ell^\infty$.
    
    % Choosing $\dlt = \frac{\eta}{2(1 + 4\eta)} \in (\frac{\eta}{4}, \frac{\eta}{2})$, 
    Thanks to the estimate above,  \eqref{eq:BA1-NL2}, \eqref{eq:BA2-NL2}, and Lemma~\ref{lem:L2-conservation-Hartree}, the proof follows.
\end{proof}

\begin{proposition} \label{prop:hartree-gamma-evol}
    Under the assumptions of Main~Theorem~\ref{thm:nonlinear-2} and bootstrap assumptions \eqref{eq:BA1-NL2}, \eqref{eq:BA2-NL2}, we define a modified profile \begin{align*}
        \tilde\gmm(t, v):= \gmm(t, v) \exp\left(i\int_1^t \frac1{s}\big(|\cdot|^{-1} * |\gmm(s, \cdot)|^2\big)(v) \, \ud s \right).
    \end{align*}
    Then there exists $T_0$ (independent of $T_f$) such that for all $T_0 < t < T_f$, for any solution $u$ to \eqref{eq:nls},
    \begin{align*}
        \left\| \tfrac{d}{dt}\tilde\gmm(t, v)\right\|_{L^\infty_v} \aleq t^{-\frac54+\frac{\eta}{2}}\|\calL u\|_{L^2} + t^{\frac12 - \frac{\eta}{4}}\|u\|_{L^\infty} + \eps t^{-1-\eta}.
    \end{align*}
\end{proposition}
\begin{proof}
    Going through the computation in Section~\ref{sec-proof_prop:dotgmm_decay}, one notices that \[
    \p_t\gmm = R_{1,0} + R_{5,0} + R_\calF = -\frac{i}{4t^2}\brk{\calL u, \Psi_v} - \frac1{4it\sqrt{t}}\brk{Lu, \td\chi^v e^{i\bfPhi^\ob}} + R_\calF,
    \]
    where $R_{\calF}:= -i\brk{\calN(u), \Psi_v}$.
    We write \begin{equation}\label{eq:Hartree_nonlin_est}
    \begin{aligned}
        \brk{\calN(u), \Psi_v} - \frac1{t}\big(|\cdot|^{-1} &* |\gmm(t, \cdot)|^2\big)(v)\gmm(t, v) = \gmm(t,v) \int \frac1{|vt - v't|} (|U(t, v't)|^2 - |\gmm(t, v')|^2)\, \ud v' \\
        &\quad\quad\qquad+ t^{-\frac32}\int \left(\int \left(\frac1{|x-x'|} - \frac1{|vt - x'|}\right) |u(t, x')|^2\, \ud x'\right)\, U(t, x) \chi(\tfrac{x-vt}{\sqrt{t}})\, \ud x.
        \end{aligned}
    \end{equation}
    \noindent\textit{The first term in \eqref{eq:Hartree_nonlin_est}. }
    Applying Lemma~\ref{lem:Holder_Lorentz}, Corollary~\ref{cor:Lorentz_interpolation} and the trivial bound from convolution \eqref{eq:defn_gmm}, we arrive at \[\begin{aligned}
        &\left| \int \frac1{|v - v'|} (|U(t, v't)|^2 - |\gmm(t, v')|^2)\, \ud v'  \right| \lesssim \left\| |U(t, vt)|^2 - |\gmm(t, v)|^2 \right\|_{L^{\frac32, 1}_v} \\
        \lesssim&\, \|U(t, vt) - \gmm(t, v)\|_{L^{3,2}_v} \|U(t, vt) + \gmm(t, v)\|_{L^{3,2}_v} \\
        \lesssim&\, \|U(t, vt) - \gmm(t, v)\|_{L^2_v}^{\frac23} \|U(t, vt) - \gmm(t, v)\|_{L^\infty_v}^{\frac13} \|U(t, tv)\|_{L^2_v}^{\frac23}\|U(t, tv)\|_{L^\infty_v}^{\frac13}\\
        \lesssim&\, \|U(t, vt) - \gmm(t, v)\|_{L^2_v}^{\frac23} \|u(t, \cdot)\|_{L^2_x}^{\frac23}\|U(t, \cdot)\|_{L^\infty}^{\frac23} \aleq C^2\eps^2 t^{-\frac16+\frac{2}{3}\eta},
    \end{aligned}
    \]
    where the last line can be further bounded by Lemma~\ref{lemma_L2_control_Uminusgmm}. Note that this decay rate is acceptable since $\eta \leq \frac{1}{10}$.

    \noindent\textit{The second term in \eqref{eq:Hartree_nonlin_est}. }
    Moreover, using the support condition $|x - vt| \aleq t^{\frac12}$, we know that for sufficiently large $T_0$ (compared to the annuli of the support of $\chi$),
    \[
        \{x': |x - x'| + |vt - x'| \gtrsim t^{1-\dlt}\} = \{x': |x - x'| \gtrsim t^{1-\dlt}\} = \{x': |vt - x'| \gtrsim t^{1-\dlt}\}, \hbox{ for } t > T_0.
    \]
    Therefore, \[
        \left|\int_{|x - x'| + |vt - x'| \gtrsim t^{1-\dlt}} \frac{|vt-x'| - |x-x'|}{|x-x'||vt - x'|} |u(t, x')|^2\, \ud x'\right| \lesssim \frac{t^{\frac12}}{t^{2-2\dlt}} \|u\|_{L^2}^2 .
    \]
    On the other hand, a dyadic decomposition in $|x - x'|$ and $|vt - x'|$, respectively, implies \[\begin{aligned}
        &\left|\int_{|x - x'| + |vt - x'| \lesssim t^{1-\dlt}} \left(\frac1{|x-x'|} - \frac1{|vt - x'|}\right) |u(t, x')|^2\, \ud x'\right| \\ \lesssim& \int_{|x - x'| \lesssim t^{1-\dlt}} \frac1{|x-x'|}|u(t, x')|^2\, \ud x' + \int_{|vt - x'| \lesssim t^{1-\dlt}} \frac1{|vt - x'|} |u(t, x')|^2\, \ud x' \lesssim t^{2-2\dlt} \|M(|u|^2)\|_{L^\infty} \lesssim t^{-1-2\dlt} \|U(t, \cdot)\|_{L^\infty}^2.
    \end{aligned}
    \]

    Therefore, thanks to \eqref{eq:BA1-NL2} and  \eqref{eq:BA2-NL2}, by choosing $\dlt = \frac{\eta}{2}$, there exists $R_H$ such that \[
        \p_t\gmm(t, v) = -i\frac1{t}\big(|\cdot|^{-1} * |\gmm(t, \cdot)|^2\big)(v)\gmm(t, v) + R_{1,0} + R_{5, 0} + R_H, \quad |R_H| \lesssim \eps t^{-1-\eta}.
    \]
    It is then clear that one should consider \[
        \tilde\gmm(t, v):= \gmm(t, v) \exp\left(i\int_1^t \frac1{s}\big(|\cdot|^{-1} * |\gmm(s, \cdot)|^2\big)(v) \, \ud s \right),
    \]
    which satisfies \[
        \p_t{\tilde\gmm}(t, v) = (R_{1,0} + R_{5, 0} + R_H)\exp\left(i\int_1^t \frac1{s}\big(|\cdot|^{-1} * |\gmm(s, \cdot)|^2\big)(v) \, \ud s \right)
    \]
    and hence the result follows by combining with the estimates for $R_{1, 0}$ and $R_{5,0}$ in Section~\ref{sec-proof_prop:dotgmm_decay}.
\end{proof}

Now we need to establish an appropriate local well-posedness result to facilitate the bootstrap argument.
\begin{lemma}\label{lem:LWP-Hartree}
    Given initial data $u(t_0) = u_0$ such that \[
        \begin{cases}
            \|\brk{x}^2\p_x^{(\leq 2)} u_0\|_{L^2} \aleq \eps \ll 1, \hbox{ when $t_0 = 0$}, \\
            \|u_0\|_{L^2} + \|\calL|_{t = t_0} u_0\|_{L^2} \aleq \eps \ll 1, \hbox{ when $t_0 > 0$}.
        \end{cases}
    \]
    Then the equation is locally wellposed with a solution $u$ on $[t_0, t_0 + T]$ for some $T > 0$ such that $u \in L^\infty L^2$ and $\calL u \in L^\infty L^2$.
\end{lemma}
\begin{proof}
    We will focus on the case $t_0 > 0$ in what follows. Note that one can replace $L_\mu$ in the following computations by $\p_\mu$ to establish the standard $H^{2,2}$-well-posedness to pass from $t = 0$ to nonzero $t_0$.

    For $t_0 > 0$, we consider $w = e^{-i\mathring{\bfPhi}}u$ and then $\calL u = 4t^2\calH w$. This leads to the conjugated version of the Gagliardo--Nirenberg inequality \begin{equation}
        \label{eq:GN-ineq-conj-ver}
        t^{\frac32}\|u(t, \cdot)\|_{L^\infty} \aleq \|\calL u(t,\cdot)\|_{L^2}^{\frac34} \|u(t, \cdot)\|_{L^2}^{\frac14} + t^{\frac32}\|u(t, \cdot)\|_{L^2}.
    \end{equation}

    Applying Lemma~\ref{lem:est-Hartree-Nu}, we have \[
        \|\calN(u)\|_{L^1L^2} \aleq T\|u\|_{L^2}^{\frac73}\|u\|_{L^\infty}^{\frac23} \aleq_{t_0} T(\|u\|_{L^2}^{\frac52}\|\calL u\|_{L^2}^{\frac12} + \|u\|_{L^2}^3) .
    \]
    Moreover, the equation for $\calL u$ has the form \[
    (i\p_t - \calH) \calL u = \calL\calN(u) + [i\p_t - \calH, \calL]u.
    \]
    Thanks to \eqref{eq:est-calL-Nu-Hartree}, \[
        \|\calL\calN(u)\|_{L^1L^2} \aleq_{t_0} T(T + t_0)\|\calL u\|_{L^2}^{\frac32} \|u\|_{L^2}^{\frac32} + T(T + t_0)^2 \|u\|_{L^2}^{\frac52}\|\calL u\|_{L^2}^{\frac12} + T(T + t_0)^2\|u\|_{L^2}^3.
    \]
    In addition, according to \eqref{eq:comm_id_in_proof} and our special form $\calH = -\Delta + c$, we have \[
        \|[i\p_t - \calH, \calL]u\|_{L^1L^2} \aleq T(T + t_0)\|u\|_{L^2}.
    \]
    Similar bounds can be proved for differences and hence, one can apply contraction mapping theorem on the space \[
        \set{u \in L^\infty L^2: \calL u \in L^\infty L^2, \|u\|_{L^\infty L^2} + \|\calL u\|_{L^\infty L^2} \leq M}, \quad d(u, v) = \|u - v\|_{L^\infty L^2}.
    \]
    Therefore, local well-posedness holds true.
\end{proof}
\begin{remark}
    If Strichartz estimates hold for $\calH$, then one first obtains local well-posedness in $L^2$, and consequently in the stronger regularity space
    \[
    \|u\|_{L^2} + \|\calL|_{t_0} u\|_{L^2} \ll 1, \qquad \text{for any } t_0 \geq 0.
    \]
\end{remark}

\begin{lemma}[Base iteration]\label{lem:base-iter-main-thm-2}
    Under the assumptions of Main~Theorem~\ref{thm:nonlinear-2} and \eqref{eq:BA1-NL2}, \eqref{eq:BA2-NL2}, we have \[
        \|\calL u\|_{L^\infty L^2[0, t]} \aleq \eps \brk{t}^2, \quad \|u(t, \cdot)\|_{L^\infty} \aleq \eps
    \]
    on the interval of existence.
\end{lemma}
\begin{proof}
    Commuting the equation with $\calL$ and invoking \eqref{eq:comm_id_in_proof}, we obtain \begin{align*}
        (i\p_t - \calH)\calL u = O(t\brk{r}^{-2-2\eta} u) + \calL \calN(u).
    \end{align*}
    Standard energy estimates give rise to \begin{align*}
        \|\calL u\|_{L^\infty L^2[0, t]} \aleq \eps + t^2 \|u\|_{L^\infty L^2} + t\|\calL \calN(u)\|_{L^\infty L^2}.
    \end{align*}
    Applying conservation of mass, \eqref{eq:est-calL-Nu-Hartree} and bootstrap assumptions, one concludes that $\|\calL u(t, \cdot)\|_{L^2} \aleq \eps \brk{t}^2$.

    On the other hand, the pointwise estimate follows from \eqref{eq:GN-ineq-conj-ver}.
\end{proof}

\begin{lemma}[Main iteration]
    Under the assumptions of Main~Theorem~\ref{thm:nonlinear-2} and \eqref{eq:BA1-NL2}, \eqref{eq:BA2-NL2}, there exists $T_0$ such that for any $t > T_0$, the same conclusion holds as in Lemma~\ref{lem:main_iteration} with $A \aleq \eps$ and $A' \aleq \eps$.
\end{lemma}
\begin{proof}
    We only highlight the differences from the proof of Lemma~\ref{lem:main_iteration}. First, although a stronger bound \eqref{eq:E-u} was assumed in that lemma, its proof shows that this additional information is in fact unnecessary. Moreover, the modification here is minor, since $|\tilde\gmm| = |\gmm|$.
\end{proof}

Now we are ready to prove Main~Theorem~\ref{thm:nonlinear-2}. Thanks to Lemma~\ref{lem:LWP-Hartree} and the conservation of mass, \eqref{eq:BA1-NL2} and \eqref{eq:BA2-NL2} are open conditions. It then should be noted that the base iteration improves the bootstrap assumptions on the interval $[0, T_0]$, once $T_0$ is fixed in the main iteration lemma. In conjunction with the main iteration lemma, this establishes the sharp pointwise decay globally in time and concludes the bootstrap argument. Hence, global existence and sharp pointwise decay have been established.

Then we take advantage of the improved bootstrap estimates and would arrive at a profile $\tilde\gmm_\infty \in L^\infty_v$ such that \[
    \gmm(t, v)\exp\left(i\int_1^t \frac1{s}\big(|\cdot|^{-1} * |\gmm(s, \cdot)|^2\big)(v) \, \ud s \right) \to \tilde\gmm_\infty(v) \hbox{ in } L^\infty_v \hbox{ as } t \to \infty.
\]
Denote \[
    \theta(t, v):= -\int_1^t \frac1{s}\big(|\cdot|^{-1} * (|\gmm(s, \cdot)|^2 - |\gmm_\infty(\cdot)|^2)\big)(v) \, \ud s.
\]
It then follows from Lemma~\ref{lem:est-Hartree-Nu} that \[
    |\tfrac{d}{dt}\theta(t, v)| \aleq t^{-1}t^{-\frac{\eta}{12}}.
\]
Thus, there exists $\theta_\infty(v):= \lim_{t\to\infty} \theta(t, v) \in L^\infty_v$. Set \[
    \gmm_\infty(v):= \tilde\gmm_\infty(v) e^{i\theta_\infty(v)}.
\]
This allows us to write \begin{align*}
    &|U(t, tv) - \gmm_\infty(v)\exp\left(-i\big(|\cdot|^{-1} * |\gmm_\infty(\cdot)|^2(v)\big) \log t\right)\varphi_0(tv)| \\
    \aleq&\, |U(t, tv) - \gmm(t, v)\varphi_0(tv)| + |\varphi_0(tv)|\left|\gmm(t,v)\exp\left(i\big(|\cdot|^{-1} * |\gmm_\infty(\cdot)|^2(v)\big) \log t\right) - \gmm_\infty(v)\right| \\
    \aleq &\, \eps t^{-\frac{\eta}{4}} + \eps \left| e^{i\theta(t, v)} - e^{i\theta_\infty(v)}\right| \aleq \eps t^{-\frac{\eta}{12}}.
\end{align*}
This finishes the proof.

\appendix
\section{On \ref{hyp:iled-0}, \ref{hyp:iled*} and global-in-time local smoothing} \label{sec:smoothing}
We state the sharp global-in-time local smoothing when $\calH$ is a small perturbation of $-\lap$, whose exact form is borrowed from \cite{Tataru}. Description of the relevant function spaces for the Schr\"odinger equation necessarily involves fractional derivatives. For this purpose, we use Littlewood--Paley projections $P_{k}$ to frequencies $\aeq 2^{k}$ based on the usual Fourier transform on $\bbR^{3}$ and make the following definitions:
\begin{align*}
\nrm{u}_{LE_{k}} &= \begin{cases}
\nrm{u}_{L^{2} L^{2}(\abs{x} < 2)} + \sup_{j > 0} \nrm{\abs{x}^{-\frac{1}{2}} u}_{L^{2} L^{2}(2^{j-1} < \abs{x} < 2^{j+1})}, & \hbox{ if } k \geq 0, \\
2^{\frac{k}{2}} \nrm{u}_{L^{2} L^{2}(\abs{x} < 2^{-k+1})} + \sup_{j > -k} \nrm{\abs{x}^{-\frac{1}{2}} u}_{L^{2} L^{2}(2^{j-1} < \abs{x} < 2^{j+1})}, & \hbox{ if } k < 0, \\
\end{cases} \\
\nrm{u}_{LE^{\frac{1}{2}}} &= \bb( \sum_{k \in \bbZ} (2^{\frac{k}{2}} \nrm{P_{k} u}_{LE_{k}})^{2} \bb)^{\frac{1}{2}} \\
\nrm{f}_{LE_{k}^{\ast}} &= \begin{cases}
\nrm{f}_{L^{2} L^{2}(\abs{x} < 2)} + \sum_{j > 0} \nrm{\abs{x}^{\frac{1}{2}} f}_{L^{2} L^{2}(2^{j-1} < \abs{x} < 2^{j+1})}, & \hbox{ if } k \geq 0, \\
2^{\frac{k}{2}} \nrm{f}_{L^{2} L^{2}(\abs{x} < 2^{-k+1})} + \sum_{j > -k} \nrm{\abs{x}^{\frac{1}{2}} f}_{L^{2} L^{2}(2^{j-1} < \abs{x} < 2^{j+1})}, & \hbox{ if } k < 0,
\end{cases} \\
\nrm{f}_{(LE^{\ast})^{-\frac{1}{2}}} &= \bb( \sum_{k \in \bbZ} (2^{-\frac{k}{2}} \nrm{P_{k} f}_{LE^{\ast}_k})^{2} \bb)^{\frac{1}{2}}
\end{align*}
In \cite[Lemma 1]{Tataru}, the following Hardy-type inequalities were established:
\begin{equation} \label{eq:iled-le}
	\nrm{\abs{x}^{-1} u}_{L^{2} L^{2}} \aleq \nrm{u}_{LE^{\frac{1}{2}}}, \quad
	\nrm{f}_{(LE^{\ast})^{-\frac{1}{2}}} \aleq \nrm{\abs{x} f}_{L^{2} L^{2}}.
\end{equation}

We restate the following result from \cite[Theorem 2]{Tataru}:
\begin{proposition} [Global-in-time local smoothing for small perturbations of $-\lap$ when $\calK = \0$]\label{prop:kato-small}
There exists $\eps_{c} > 0$ such that if $\calH$ satisfies \ref{hyp:af} with $\eta > 0$ and $A_{c} < \eps_{c}$, then the following estimate holds:
\begin{equation} \label{eq:kato-small}
	\nrm{u}_{L^{\infty} L^{2} \cap LE^{\frac{1}{2}}} \aleq \nrm{u(0)}_{L^{2}} + \nrm{(i \rd_{t} - \calH) u}_{L^{1} L^{2} + (LE^{\ast})^{-\frac{1}{2}}}.
\end{equation}
\end{proposition}

In view of \eqref{eq:iled-le}, \eqref{eq:kato-small} implies \ref{hyp:iled*} with $M_c^* = 0, \bt_{c} = 0$ and $s_c^* = 0$.

Next, we start by formulating a version of \ref{hyp:iled-0}, which also takes into account the possible presence of an obstacle. We say that $L = i\p_t - \calH$ satisfies the \emph{integrated local energy decay estimates} up to order $M_{c}$ if the following holds:
\begin{enumerate}[label=$(\mathrm{ILED})$]
\item (Integrated local energy decay)\label{hyp:iled}
    When $\calK = \0$, assume that \ref{hyp:iled-0} holds (up to order $M_{c}$). When $\calK \neq \0$, assume that, every $N = 0, 1, \ldots, M_{c}$, the following holds. For every $F: [t_0, t_0 + T] \times \calM^3 \to \bbC$ such that \begin{align*}
        \supp F \subseteq \set{(t, x) \in [t_0, t_0 + T] \times \calM^{3}: \abs{x} < R}, \quad \sum_{2i + j \leq N} \nrm{\rd_{t}^{(\leq i)}\rd_{x}^{(\leq j+s_{c})} F}_{L^{2} L^{2}([t_0, t_0 + T] \times \set{\abs{x} < R})} < +\infty,
    \end{align*}
    there exists a unique forward solution $\psi \in C_t([t_0, t_0 + T]; H^N)$ to $L \psi = F$ in $[t_0, t_0 + T] \times \calM^3$ with $\psi = 0$ on $[t_0, \infty)_{t} \times \rd \calK$. Moreover, the solution $\psi$ satisfies
    \begin{equation*}
    	\nrm{\rd_{x}^{(\leq N)} \psi}_{L^{2} L^{2}([t_0, t_0 + T] \times \set{\abs{x} < R})} \leq C_{R, N} \sum_{2i + j \leq N} \nrm{\rd_{t}^{(\leq i)}\rd_{x}^{(\leq j+s_{c})} F}_{L^{2} L^{2}([t_0, t_0 + T] \times \set{\abs{x} < R})},
    \end{equation*}
    where $C_{R, N}$ is independent of $t_0 \geq 0, T > 0$.
\end{enumerate}

\begin{proposition} \label{prop:iled2iled-star}
Assume that \ref{hyp:af} holds true. In the case of $\calK \neq \0$, we additionally impose the hypotheses \ref{hyp:ult-stat} and \ref{hyp:en}.
If \ref{hyp:iled} holds true, then \ref{hyp:iled*} holds with an arbitrary $\bt_{c} \geq 0$, $s_{c}^{\ast} = s_{c} + \delta_s + 3$ and $M_{c}^* = M_c - 2 - s_c$, where \[\delta_s:=
        \begin{cases}
             1, \quad \bfh^{\mu\nu} \not\equiv 0, \\
             0, \quad \bfh^{\mu\nu} \equiv 0.
        \end{cases}
    \]
\end{proposition}
\begin{proof}[Proof of Proposition~\ref{prop:iled2iled-star}]
    The uniqueness part follows by considering the difference between two possible solutions. The following is devoted to the existence part.
    For simplicity, we set $t_0 = 0$.
    Consider $(i\p_t - \calH)u = f$ with $u(0) = u_0$.
    Introducing a cut-off $\chi:= \chi_{> R}$ adapted to the region $\{|x| > R\}$, we consider
    \begin{align}\label{eq:defn-Hout}
        \calH_{\mathrm{out}} u &:= - \lap u + D_{\mu} (\chi_{>R} \bfh^{\mu\nu} D_{\nu} u) \\
	    & \peq + \chi_{>R} b^{\mu}(t, x) D_{\mu} u + D_{\mu} (\chi_{>R} b^{\mu}(t, x)  u) + \chi_{>R} c(t, x) u
    \end{align}
    and fix $R$ sufficiently large so that Proposition~\ref{prop:kato-small} applies to $i \rd_{t} - \calH_{\mathrm{out}}$. Let $v: [0, \infty) \times \R^3 \to \bbC$ be the solution (existence and uniqueness follow from \cite{Tataru}) to
    \begin{equation*}
	   (i \rd_{t} - \calH_{\mathrm{out}}) v = f, \quad v(0) = u_0,
    \end{equation*}
    where we extend $u_0$ to the full space when $\calK \neq \0$.

    It then follows from Proposition~\ref{prop:kato-small} that for any $N \geq 0$, \[\begin{aligned}
        \|\p_x^{(\leq N)} v\|_{L^\infty L^2 \cap LE^{\frac12}([0, \infty) \times \R^3)} &\lesssim \|\p_x^{(\leq N)} v_0\|_{L^2} + \|(i\p_t - \calH_{\mathrm{out}})\p_x^{(\leq N)}v\|_{L^1L^2 + (LE^*)^{-\frac12}} \\
        &\lesssim \|\p_x^{(\leq N)} v_0\|_{L^2} + \|\p_x^{(\leq N)}f\|_{L^1L^2 + (LE^*)^{-\frac12}} + \eps_c \|\brk{x}^{-1-2\eta}\p_x^{(\leq N+1)} v\|_{(LE^*)^{\frac12}} \\
        &\lesssim \|\p_x^{(\leq N)} v_0\|_{L^2} + \|\p_x^{(\leq N)}f\|_{L^1L^2 + (LE^*)^{-\frac12}} + \eps_c \|\p_x^{(\leq N)} v\|_{(LE)^{\frac12}}, \\
    \end{aligned}
    \]
    where the last step follows from \cite[Remark 5]{Tataru}.
    Then one can absorb the last term to the left-hand side and hence \begin{equation}\label{eq:iled*-Hout}
        \|\p_x^{(\leq N)}v\|_{L^\infty L^2 \cap LE^{\frac12}} \lesssim \|\p_x^{(\leq N)} v_0\|_{L^2} + \|\p_x^{(\leq N)}f\|_{L^1L^2 + (LE^*)^{-\frac12}}, \quad \forall N \geq 0.
    \end{equation}
    Set $w:= u - \eta_0 v$, where $\eta_0$ is given by \eqref{eq:defn-eta_0}. Then $w$ solves  \begin{equation}\label{eq:iled*-H-w}
	   (i \rd_{t} - \calH) w = (1 - \eta_0)f + \eta_0(\calH - \calH_{\mathrm{out}}) v + [\calH, \eta_0]v, \quad w(0) = (1 - \eta_0) u_0, \quad w|_{[0, \infty) \times \p\calK} = 0,
    \end{equation}
    where the right-hand side is of compact support in $x$.

    {\it Case 1: $\calK = \0$. }
    Due to the zero initial data and the definition of $\calM$, we know that $w$ is a forward solution and hence \ref{hyp:iled-0} is applicable. We commute the equation with $\p_x^{(\leq N)}$ and use the multiplier $\p_x^{(\leq N)}\br w$, which leads to (an energy estimate with possible derivative loss)
    \begin{equation}
        \begin{aligned}\label{eq:iled*-w-eng}
            \|\p_x^{(\leq N)} w\|_{L^\infty L^2} &\aleq_R \|\brk{x}^{-1-2\eta}\p_x^{(\leq N+1)} w\|_{L^2 L^2} + \|\p_x^{(\leq N)} (\calH - \calH_{\mathrm{out}}) v\|_{L^2L^2}^{\frac12} \|\brk{x}^{-1}\p_x^{(\leq N)} w\|_{L^2 L^2}^{\frac12} \\
            &\aleq_R \|\brk{x}^{-1}\p_x^{(\leq N+1)} w\|_{L^2 L^2} + \|\p_x^{(\leq N + \dlt_s + 1)}  v\|_{L^2L^2}, \quad \forall 0 \leq N \leq M_c.
        \end{aligned}
    \end{equation}
    We first localize the first term on the right-hand side by considering \[
        (i \rd_{t} - \calH_{\mathrm{out}}) \chi_{> 2R} w = - [\calH, \chi_{>2R}]w,
    \]
    where $\chi_{> 2R}$ is such that $\calH = \calH_{\mathrm{out}}$ on $\supp \chi_{> 2R}$.
    Then this can be treated by Proposition~\ref{prop:kato-small}. Hence, an estimate similar to \eqref{eq:iled*-Hout} could be derived: \begin{equation}\label{eq:iled*-Hout-cut-off}
        \|\brk{x}^{-1}\p_x^{(\leq N+1)} \chi_{>2R}w\|_{L^2 L^2} \aleq_R  \|\p_x^{(\leq N + 2)} w\|_{L^2L^2(r \aeq R)}, \quad \forall N \geq 0.
    \end{equation}
    Combining with \eqref{eq:iled*-w-eng}, it yields that \begin{align}\label{eq:iled*-w-eng-update}
        \|\p_x^{(\leq N)} w\|_{L^\infty L^2}
        &\aleq_R \|\brk{x}^{-1}\p_x^{(\leq N + 2)} w\|_{L^2 L^2(r \aleq R)} + \|\p_x^{(\leq N + \dlt_s + 1)}  v\|_{L^2L^2}, \quad \forall 0 \leq N \leq M_c.
    \end{align}
    To handle the first term on the right-hand side, we apply \ref{hyp:iled-0} to \eqref{eq:iled*-H-w}. Thus, \begin{align}\label{eq:iled*-Hin-L2L2}
        \|\p_x^{(\leq N+2)} w\|_{L^2 L^2(r \lesssim R)} &\lesssim_R \|\p_x^{(\leq N + 2 + s_c)} (\calH - \calH_{\mathrm{out}}) v\|_{L^2L^2}\\ &\lesssim_R  \|\p_x^{(\leq N + s_c + \delta_s + 3)} v\|_{L^2L^2(r \aleq R)}, \quad \forall 0 \leq N \leq M_c - 2 - s_c.
    \end{align}

    Combining \eqref{eq:iled*-Hin-L2L2} and \eqref{eq:iled*-w-eng-update}, it yields that \begin{align}\label{eq:iled*-H-w-cpt-R}
        \|\p_x^{(\leq N)} w\|_{L^\infty L^2} \aleq \|\p_x^{(\leq N + s_c + \delta_s + 3)}v\|_{L^2L^2(r\aleq R)}, \quad \forall 0 \leq N \leq M_c - 2 - s_c.
    \end{align}
    Together with
    \eqref{eq:iled*-Hout}, the result follows.

    {\it Case 2: $\calK \neq \0$. }
    Since we allow for loss of derivatives on the right-hand side, it suffices to handle $N \in 2\bbZ_{\geq 0}$ while odd cases follow from interpolation, where we might experience a further derivative loss.
    Consider \[
        (i\p_t - \calH) w_1 = 0, \quad w_1(0) = u_0 - \eta_0 v_0, \quad w_1|_{[0, \infty) \times \p\calK} = 0.
    \]
    Note that $w_1 \equiv 0$ in the case that $\calK = \0$. We use $\p_t$ as commutator and derive the equation that $\p_t^k w_1$ solves with $k \leq \frac{N}{2}$. We examine the initial condition by restricting the following computation at $t = 0$: \begin{align*}
        \p_t^{k} w_1 &= -i\p_t^{k-1} \calH w_1 = -i[\p_t^{k-1}, \calH] w_1 - \calH \p_t^{k-2} \calH w_1 = \cdots \\
        &= O\big(\p_t^{(\leq k-1)}\p_x^{(\leq 2k-1)} a^{\mu\nu}, \p_t^{(\leq k-1)}\p_x^{(\leq 2k-1)} b^\mu, \p_t^{(\leq k-1)}\p_x^{(\leq 2k-2)} c \big) O(\p_x^{(\leq 2k)}w_1).
    \end{align*}
    Hence, applying \ref{hyp:en} to the equation which $\p_t^k w_1$ satisfies, one could iteratively (in $k$) obtain
    \begin{align*}
        \|\p_t^{(\leq \frac{N}{2})}w_1\|_{L^\infty L^2} \aleq \|\p_t^{(\leq \frac{N}{2})}w_1(0)\|_{L^2} \aleq \|\p_x^{(\leq N)}u_0\|_{L^2},
    \end{align*}
    where \ref{hyp:ult-stat} and \ref{hyp:af} are used in the first and second step, respectively.

    On the other hand, we set $w_2 = w - w_1$. To derive an analog of \eqref{eq:iled*-w-eng} and \eqref{eq:iled*-w-eng-update}, we consider the equation satisfied by $\p_t^{(\leq \frac{N}{2})} w_2$ instead of $\p_x^{(\leq N)} w_2$. Then we are able to bound $\|\p_t^{(\leq \frac{N}{2})} w_2\|_{L^\infty L^2}$ by \begin{align*}
        \|\p_t^{(\leq \frac{N}{2})} w_2\|_{L^\infty L^2} &\aleq \|\brk{x}^{-1}\p_t^{(\leq \frac{N}{2})} \p_x^{(\leq 1)} w_2\|_{L^2 L^2} + \|\p_t^{(\leq \frac{N}{2})}\p_x^{(\leq \dlt_s + 1)} v\|_{L^2L^2} + \|\brk{x}\p_t^{(\leq \frac{N}{2})}f\|_{L^2L^2} \\
        &\aleq \|\p_t^{(\leq \frac{N}{2})} \p_x^{(\leq 2)} w_2\|_{L^2 L^2(r \aleq R)} + \|\p_t^{(\leq \frac{N}{2})}\p_x^{(\leq \dlt_s + 2)} v\|_{L^2L^2} + \|\brk{x}\p_t^{(\leq \frac{N}{2})}\p_x^{(\leq 1)}f\|_{L^2L^2} \\
        &\aleq \|\p_x^{(\leq N + 2)} w_2\|_{L^2 L^2(r \aleq R)} + \|\p_t^{(\leq \frac{N}{2})}\p_x^{(\leq \dlt_s + 2)} v\|_{L^2L^2} + \|\brk{x} \p_t^{(\leq \frac{N}{2})} \p_x^{(\leq 1)} f\|_{L^2L^2} \\
        &\aleq \|\p_x^{(\leq N + 2)} w_2\|_{L^2 L^2(r \aleq R)} + \|\p_x^{(\leq N + \dlt_s + 2)} v\|_{L^2L^2} + \|\brk{x} \p_t^{(\leq \frac{N}{2})} \p_x^{(\leq \dlt_s + 2)} f\|_{L^2L^2}
    \end{align*}
    where we simply use the equations for $w_2$ and $v$ to convert time derivatives to spatial derivatives. Now we handle the $L^2L^2(r \aleq R)$ term in the same way as \eqref{eq:iled*-Hin-L2L2} in the obstacle-free case by applying \ref{hyp:iled}:
    \begin{align*}
        &\quad\|\p_x^{(\leq N+2)} w_2\|_{L^2 L^2(r \lesssim R)} \lesssim_R \sum_{2i + j \leq N + 2}\|\p_t^{(\leq i)}\p_x^{(\leq j + s_c)} (\p_x^{(\leq \dlt_s + 1)}v + f)\|_{L^2L^2(r \aleq R)}\\ &\lesssim_R \sum_{2i + j \leq N + 2} \|\p_t^{(\leq i)}\p_x^{(\leq j + s_c + \delta_s + 1)} v\|_{L^2L^2(r \aleq R)} + \|\p_t^{(\leq i)}\p_x^{(\leq j + s_c)} f\|_{L^2L^2(r \aleq R)}, \quad \forall 0 \leq N \leq M_c - 2 - s_c,
    \end{align*}
    which in turn gives \[
        \|\p_t^{(\leq \frac{N}{2})} w_2\|_{L^\infty L^2} \aleq \|\p_x^{(\leq N + s_c^*)} v\|_{L^2L^2(r \aleq R)} + \sum_{2i + j \leq N + s_c^*}\|\brk{x} \p_t^{(\leq i)}\p_x^{(\leq j)} f\|_{L^2L^2}, \quad \forall 0 \leq N \leq  M_c - 2 - s_c.
    \]
    Iteratively applying Lemma~\ref{lem:unweight-ell} as in the proof of Lemma~\ref{lem:pass-pt-to-px}, we derive that \begin{align*}
        \|\p_x^{(\leq N)} w_2\|_{L^\infty L^2} &\aleq \|\p_t^{(\leq \frac{N}{2})} w_2\|_{L^\infty L^2} + \sum_{2i + j \leq N} \|\p_t^{(\leq i)} \p_x^{(\leq j)} (f + \p_x^{(\leq \dlt_s + 1)} v)\|_{L^\infty L^2} \\
        &\aleq  \|\p_t^{(\leq \frac{N}{2})} w_2\|_{L^\infty L^2} + \sum_{2i + j \leq N} \|\p_t^{(\leq i + 1)} \p_x^{(\leq j)} f\|_{L^2 L^2} + \|\p_t^{(\leq i)} \p_x^{(\leq j + \dlt_s + 1)} v\|_{L^\infty L^2},
    \end{align*}
    where we use Sobolev embedding in the last step with a uniform constant for $T > 1$.

    Together with the preceding estimates, we conclude the proof.
\end{proof}

% ----------------------------------------------------------------
% \bibliographystyle{amsplain}
\bibliographystyle{alpha}
\bibliography{bibliography}
% ----------------------------------------------------------------

\end{document}